\documentclass[11pt,english]{amsart}

\usepackage{amsmath,amsfonts,amssymb,amsthm,amscd}
\usepackage{dsfont, mathrsfs,enumerate,eucal,fontenc}
\usepackage[varg]{txfonts}
\usepackage{xspace}
\usepackage[english]{babel}

\usepackage{xy, xypic}
\usepackage{color} 
\usepackage{graphicx}
\usepackage{psfrag}  
\usepackage{pst-all}
\usepackage{youngtab}

\usepackage{xcolor}
\definecolor{rouge}{rgb}{0.85,0.1,.4}
\definecolor{bleu}{rgb}{0.1,0.2,0.9}
\definecolor{violet}{rgb}{0.7,0,0.8}
\usepackage[colorlinks=true,linkcolor=bleu,urlcolor=violet,citecolor=rouge]{hyperref}
\usepackage{breakurl}
\usepackage{xspace}

\usepackage{fullpage}

\theoremstyle{plain}
\newtheorem{theorem}{Theorem}[section]
\newtheorem{lemma}[theorem]{Lemma}
\newtheorem{coro}[theorem]{Corollary}
\newtheorem{prop}[theorem]{Proposition}
\newtheorem{conj}[theorem]{Conjecture}
\theoremstyle{definition}
\newtheorem{defi}[theorem]{Definition}
\theoremstyle{remark}
\newtheorem{example}[theorem]{Example}

\newtheorem{rema}[theorem]{Remark}
\newtheorem{claim}[theorem]{Claim}
\theoremstyle{theorem}

\newtheorem{quest_alpha}{Question}
\def\quest_alpha{\Alph{quest_alpha}}

\def\g{{\mathfrak{g}}} 
  
\def\z{{\mathfrak{z}}}  
\def\k{{\Bbbk}}
\def\V{\mathbb{V}}
\def\rg{\ell}

\def\med{\medskip}
\def\small{\smallskip}
\def\noi{\noindent}
\def\geq{\geqslant}
\def\leq{\leqslant}

\def\N{\mathbb{N}}

\def\no{n$^{\circ}$}

\def\poie#1#2#3#4#5#6#7#8#9{\def\un{#5#6#7#8#9}\def\deux{#6#7#8#9}\def\trois{#2#4#8#9}
\def\quatre{#8#9}\def\cinq{#5#6#7}\def\six{#6#7}\def\sept{#2#4}
\ifx\un\empty {#1}_{#2}{#3 \hskip 0.15em}{#1}_{#4} \else \ifx\deux\empty 
{#5}(#1_{#2}){#3 \hskip 0.15em}{#5}(#1_{#4})
\else \ifx\trois\empty {#5}_{#6}(#1){#3 \hskip 0.15em}{#5}_{#7}(#1) 
\else \ifx\quatre\empty {#5}_{#6}(#1_#2){#3 \hskip 0.15em}{#5}_{#7}(#1_#4) 
\else \ifx\cinq\empty {#1}_{#2}^{#8}{#3 \hskip 0.15em}#1_#4^{#9} 
\else \ifx\six\empty {#5}(#1_{#2}^{#8}){#3 \hskip 0.15em}{#5}(#1_{#4}^{#9}) 
\else \ifx\sept\empty {#5}_{#6}(#1)^{#8}{#3 \hskip 0.15em}{#5}_{#7}(#1)^{#9} \else
{#5}_{#6}(#1_{#2}^{#8})^{#9}{#3 \hskip 0.15em}{#5}_{#7}(#1_{#4}^{#8})^{#9} 
\fi \fi \fi \fi \fi \fi \fi}
\def\poi#1#2#3#4#5#6#7{\def\un{#5#6#7}\def\deux{#6#7}
\def\trois{#2#4} \def\cinq{#3#4#5}
\ifx\un\empty {#1}_{#2}{#3 \hskip 0.15em}{#1}_{#4} \else
\ifx\deux\empty {#5}(#1_{#2}){#3 \hskip 0.15em}{#5}(#1_{#4}) \else
\ifx\trois\empty {#5}_{#6}(#1){#3 \hskip 0.15em}{#5}_{#7}(#1) \else
{#5_{#6}}(#1_{#2}){#3 \hskip 0.15em}{#5_{#7}}(#1_{#4}) \fi \fi \fi}
\def\rond{\raisebox{.3mm}{\scriptsize$\circ$}}
\def\mul{\raisebox{.3mm}{\scriptsize\hskip 0.15em$\times$\hskip 0.15em}}
\def\tens{\raisebox{.3mm}{\scriptsize$\otimes$}}

\def\dv#1#2{\langle {#1}\,,{#2}\rangle}

\def\tk#1#2{{#2}\otimes _{#1}}

\def\ec#1#2#3#4#5{\def\un{#3#4#5}\def\deux{#3#5}\def\trois{#3}
\def\four{#2#4#5}\def\five{#2#5}\def\six{#2}\def\seven{#3#4}
\def\eight{#2#4} \def\nine{#2#3#4}
\ifx\nine\empty {\rm #1}_{#5} \else
\ifx\un\empty {\rm #1}({\goth #2}) \else
\ifx\deux\empty {\rm #1}({\goth #2}_{#4}) \else
\ifx\trois\empty {\rm #1}_{#5}({\goth #2}_{#4}) \else
\ifx\four\empty {\rm #1}(#3) \else
\ifx\five\empty {\rm #1}(#3_{#4}) \else
\ifx\six\empty {\rm #1}_{#5}(#3_{#4}) \else
\ifx\seven\empty {\rm #1}_{#5} ({\goth#2})\else
\ifx\eight\empty {\rm #1}_{#5}({#3})
\fi \fi \fi \fi \fi \fi \fi \fi \fi}
\def\hec#1#2#3#4#5{\def\un{#3#4#5}\def\deux{#3#5}\def\trois{#3}
\def\four{#2#4#5}\def\five{#2#5}\def\six{#2}\def\seven{#3#4}
\def\eight{#2#4} \def\nine{#2#3#4}
\ifx\nine\empty \hat{{\rm #1}}_{#5} \else
\ifx\un\empty \hat{{\rm #1}}({\goth #2}) \else
\ifx\deux\empty \hat{{\rm #1}}({\goth #2}_{#4}) \else
\ifx\trois\empty \hat{{\rm #1}}_{#5}({\goth #2}_{#4}) \else
\ifx\four\empty \hat{{\rm #1}}(#3) \else
\ifx\five\empty \hat{{\rm #1}}(#3_{#4}) \else
\ifx\six\empty \hat{{\rm #1}}_{#5}(#3_{#4}) \else
\ifx\seven\empty \hat{{\rm #1}}_{#5} ({\goth#2})  \else
\ifx\eight\empty \hat{{\rm #1}}_{#5}({#3})
\fi \fi \fi \fi \fi \fi \fi \fi \fi}
\def\e#1#2{\ec {#1}#2{}{}{}}
\def\es#1#2{\ec {#1}{}{#2}{}{}}

\def\ai#1#2#3{\def\deux{#2#3} \def\trois{#3} \def\quatre{#2} 
\ifx\deux\empty \es S{{\goth #1}}^{{\goth #1}} \else
\ifx\trois\empty \es S{{\goth #1}^{#2}}^{{\goth #1}^{#2}} \else
\ifx\quatre\empty \es S{{\goth #1}_{#3}}^{{\goth #1}_{#3}} \else
\es S{{\goth #1}_{#3}^{#2}}^{{\goth #1}_{#3}^{#2}} \fi \fi \fi}
\def\aii#1#2#3#4{\def\deux{#2#3} \def\trois{#3} \def\quatre{#2} 
\ifx\deux\empty \sy {#4}{{\goth #1}}^{{\goth #1}} \else
\ifx\trois\empty \sy {#4}{{\goth #1}^{#2}}^{{\goth #1}^{#2}} \else
\ifx\quatre\empty \sy {#4}{{\goth #1}_{#3}}^{{\goth #1}_{#3}} \else
\sy {#4}{{\goth #1}_{#3}^{#2}}^{{\goth #1}_{#3}^{#2}} \fi \fi \fi}

\def\Bbb{\mathbb}
\def\goth{\mathfrak}
\def\cal{\mathcal}

\def\gi#1#2#3#4{\def\trois{#3#4} \def\quatre{#4}\def\cinq{#3}\ifx\trois\empty {\rm i}_{#1,{\goth #2}}
\else \ifx\quatre\empty {\rm i}_{#1_{#3},{\goth #2}} \else\ifx\cinq\empty {\rm i}_{#1,{\goth #2}_{#4}} \else {\rm i}_{#1_{#3},{\goth #2}_{#4}} \fi \fi \fi}
\def\j#1#2{\def\deux{#2} \ifx\deux\empty {\rm rk}\hskip .125em{{\goth #1}} \else {\rm rk}\hskip .125em{{\goth #1}_{#2}} \fi}
\def\aj#1#2{\def\deux{#2} \ifx\deux\empty {\rm j}_{{\goth #1}} \else {\rm j}_{{\goth #1}_{#2}} \fi}
\def\an#1#2{\def\deux{#2} \ifx\deux\empty {\cal O}_{#1} \else {\cal O}_{#1,#2} \fi }
\def\han#1#2{\def\deux{#2} \ifx\deux\empty {\hat{{\cal O}}}_{#1} \else {\hat{{\cal O}}}_{#1,#2} \fi }

\def\dim{{\rm dim}\hskip .125em}
\def\dd{{\rm d}}
\def\ad{{\rm ad}\hskip .1em}

\def\det{{\rm det}\hskip .125em}
\def\deg{{\rm deg}\hskip .125em}

\def\n{{\rm n}}
\def\s{{\rm s}}

\def\uu{{\rm u}}

\def\sy#1#2{{\rm S}^{#1}(#2)}
\def \ex #1#2{\bigwedge ^{#1}#2}

\def\ie#1{\hskip .125em ^{e}\hskip -.125em{#1}}

\begin{document}

\title[The symmetric invariants]
{The symmetric invariants of centralizers and Slodowy grading}

\author[Jean-Yves Charbonnel]{Jean-Yves Charbonnel}
\address{Jean-Yves Charbonnel, Universit\'e Paris Diderot - CNRS \\
Institut de Math\'ematiques de Jussieu - Paris Rive Gauche\\
UMR 7586 \\ Groupes, repr\'esentations et g\'eom\'etrie \\
B\^atiment Sophie Germain \\ Case 7012 \\ 
75205 Paris Cedex 13, France}
\email{jean-yves.charbonnel@imj-prg.fr}

\author[Anne Moreau]{Anne Moreau}
\address{Anne Moreau, Laboratoire de Math\'ematiques et Applications\\
T\'el\'eport 2 - BP 30179\\
Boulevard Marie et Pierre Curie\\
86962 Futuroscope Chasseneuil Cedex, France}
\email{anne.moreau@math.univ-poitiers.fr}



\date\today

\begin{abstract}
Let $\g$ be a finite-dimensional simple Lie algebra of rank $\rg$ over an 
algebraically closed field $\k$ of characteristic zero, and let $e$ be a nilpotent 
element of $\g$. Denote by $\g^{e}$ the centralizer of $e$ in $\g$ 
and by $\ai g{e}{}$ the algebra of symmetric invariants of $\g^{e}$. 
We say that $e$ is {\em good} if the nullvariety of some $\rg$ homogenous elements of 
$\ai g{e}{}$ in $({\goth g}^{e})^{*}$ has codimension $\rg$. If $e$ is good then 
$\ai g{e}{}$ is a polynomial algebra. The main result of this paper stipulates that if 
for some homogenous generators of $\ai g{}{}$, the initial homogenous components of 
their restrictions to $e+\g^{f}$ are algebraically independent, with 
 $(e,h,f)$ an $\mathfrak{sl}_2$-triple of $\g$, then $e$ is good. 
As applications, we pursue the investigations of \cite{PPY} 
and we produce (new) examples of nilpotent elements 
that satisfy the above polynomiality condition, in simple Lie algebras of both classical 
and  exceptional types. We also give a counter-example in type {\bf D}$_{7}$.
\end{abstract}

\maketitle

\tableofcontents

\section{Introduction} \label{i}

\subsection{}\label{i1}
Let $\g$ be a finite-dimensional simple Lie algebra of rank $\rg$ over an 
algebraically closed field $\k$ of characteristic zero, let $\dv ..$ be the Killing form 
of $\g$ and let $G$ be the adjoint group of $\g$. If ${\goth a}$ is a subalgebra of $\g$,
we denote by S$({\goth a})$ the symmetric algebra of ${\goth a}$. For $x\in\g$, we 
denote by $\g^{x}$ the centralizer of $x$ in $\g$ and by $G^{x}$ the stabilizer of $x$ 
in $G$. Then 
${\rm Lie}(G^{x})={\rm Lie}(G^{x}_{0})=\g^x$ where $G^{x}_{0}$ is the 
identity component of $G^{x}$.  Moreover, $\es S{\g^{x}}$ is a $\g^x$-module and 
$\ai gx{}=\es S{\g^{x}}^{G^{x}_{0}}$. An interesting question, first raised by A.~Premet,
is the following: 

\begin{quest_alpha}\label{qi1}
Is $\es S{\g^{x}}^{\g^{x}}$ a polynomial algebra in $\rg$ variables?
\end{quest_alpha}

In order to answer this question, thanks to the Jordan decomposition, we can assume that
$x$ is nilpotent. Besides, if $\es S{\g^{x}}^{\g^{x}}$ is 
polynomial for some $x\in\g$, then it is so for any element in the adjoint orbit $G.x$
of $x$. If $x=0$, it is well-known since Chevalley that $\ai gx{}=\e Sg^{\g}$ 
is polynomial in $\rg$ variables. At the opposite extreme, if $x$ is a regular nilpotent 
element of $\g$, then $\g^x$ is abelian of dimension $\rg$, \cite{DV}, and 
$\ai gx{}=\es S{\g^{x}}$ is polynomial in $\rg$ variables too. 

For the introduction, let us say most simply that $x\in\g$  
{\em satisfies the polynomiality condition} if $\ai gx{}$ is a polynomial algebra in 
$\rg$ variables.

A positive answer to Question \ref{qi1} was suggested in \cite[Conjecture~0.1]{PPY} 
for any simple $\g$ and any $x\in\g$.  
O.~Yakimova has since discovered a counter-example in type ${\bf E}_8$,~\cite{Y3}, 
disconfirming the conjecture. More precisely, the elements of the minimal nilpotent orbit 
in ${\bf E}_8$ do not satisfy the polynomiality condition. 
The present paper contains another counter-example in type {\bf D}$_7$  
(cf.~Example~\ref{erc3}). In particular, we cannot expect a positive answer 
to~\cite[Conjecture 0.1]{PPY} for the simple Lie algebras of classical type. 
Question~\ref{qi1} still remains interesting and has a positive answer for a large number
of nilpotent elements $e\in\g$ as it is explained below. 

\subsection{Review of known results}\label{i2}
We briefly review in this paragraph what has been achieved so far about Question 
\ref{qi1}. Recall that the {\em index} of a finite-dimensional Lie algebra ${\goth q}$, 
denoted by ${\rm ind}\,{\goth q}$, is the minimal dimension of the stabilizers of linear 
forms on ${\goth q}$ for the coadjoint representation, (cf.~\cite{Di1}):
$${\rm ind}\,{\goth q}:=\min \{\dim{\goth q}^{\xi} \;;\; \xi\in{\goth q}^{*}\}\;  
\text{ where }\;{\goth q}^\xi:=\{x\in {\goth q}\;;\;\xi([x,{\goth q}])=0\}.$$
By \cite{Ro}, if $\goth q$ is algebraic, i.e., ${\goth q}$ is the Lie algebra of some 
algebraic linear group $Q$, then the index of ${\goth q}$ is the transcendence degree of
the field of $Q$-invariant rational functions on ${\goth q}^*$. 
The following result will be important for our purpose. 

\begin{theorem}[{\cite[Theorem 1.2]{CM}}]\label{ti1}
The index of $\g^x$ is equal to $\rg$ for any $x\in\g$.
\end{theorem}

Theorem \ref{ti1} was first conjectured by Elashvili in the 
90's motivated by a result of Bolsinov, \cite[Theorem 2.1]{Bolb}. 
It was proven by O.~Yakimova when $\g$ is a simple Lie algebra 
of classical type, \cite{Y1}, and checked by a program by W. de Graaf 
when $\g$ is a simple Lie algebra of exceptional type, \cite{DeG}. 
Before that, the result was established for some particular classes of nilpotent elements
by D.~Panyushev, \cite{Pa2}. 

Theorem \ref{ti1} is deeply related to Question~\ref{qi1}. 
First of all, it implies that if $S(\g^{e})^{\g^{e}}$ is polynomial, it is so 
in $\ell$ variables. Further, according 
to Theorem \ref{ti1}, the main results of \cite{PPY} that we summarize below 
apply (see Theorem \ref{ti2}).  

Let $e$ be a nilpotent element of $\g$. 
By the Jacobson-Morosov Theorem, $e$ is embedded into a ${\goth {sl}}_2$-triple 
$(e,h,f)$ of $\g$. Denote by ${\cal S}_e:=e+\g^{f}$ the  
{\em Slodowy slice associated  with $e$}. Identify $\g^*$ with $\g$, and 
$({\g}^{e})^*$ with ${\goth g}^{f}$, through the Killing form $\dv ..$. For 
$p$ in $\e Sg\simeq\k[\g^*]\simeq\k[\g]$, denote by $\ie{p}$ the initial homogenous 
component of its restriction to ${\cal S}_e$. According to~\cite[Proposition 0.1]{PPY}, 
if $p$ is in $\ai g{}{}$, then $\ie{p}$ is in $\ai ge{}{}$. 
Let $(\g^{e})^*_{\rm sing}$ be the set of nonregular linear forms 
$x\in (\g^{e})^*$, i.e., 
$$ 
(\g^{e})^*_{\rm sing}:=\{x\in (\g^e)^*  \; \vert \; 
\dim (\g^{e})^x > {\rm ind}\,\g^{e}=\rg\}.
$$ 
If $(\g^{e})^*_{\rm sing}$ has codimension at least $2$ in 
$(\g^{e})^*$, we say that ${\goth g}^{e}$ is {\em nonsingular}.

\begin{theorem}[{\cite[Theorem 0.3]{PPY}}] \label{ti2}
Suppose that the following two conditions are satisfied:
\begin{itemize}
\item[(1)] for some homogenous generators $\poi q1{,\ldots,}{\rg}{}{}{}$ of 
$\ai g{}{}$, the polynomial functions $\poi {\ie q}1{,\ldots,}{\rg}{}{}{}$ are 
algebraically independent, 
\item[(2)] $\g^{e}$ is nonsingular. 
\end{itemize}
Then $S(\g^{e})^{\g^{e}}$ 
is a polynomial algebra with generators $\ie{q_1},\ldots,\ie{q_\ell}$. 
\end{theorem}

As a consequence of Theorem \ref{ti2}, if $\g$ is simple of type {\bf A} or {\bf C}, 
then all nilpotent elements of $\g$ satisfy the polynomiality condition, 
cf.~\cite[Theorems 4.2 and 4.4]{PPY}. The result for 
the type {\bf A} was independently obtained by Brown and Brundan,~\cite{BB}.  
In \cite{PPY}, the authors also provide some examples of nilpotent elements satisfying 
the polynomiality condition in the simple Lie algebras of types {\bf B} and {\bf D}, and 
a few ones in the simple exceptional Lie algebras. 

At last, note that the analogue question to Question~\ref{qi1} for the positive 
characteristic was dealt with by L.~Topley for the simple Lie algebras of types {\bf A} 
and {\bf C},~\cite{To}. 

\subsection{Main results}\label{i3}
The main goal of this paper is to continue the investigations of \cite{PPY}. 
Let us describe our main results. The following definition is central 
in our work (cf.~Definition \ref{dge1}):

\begin{defi} \label{di}
An element $x\in{\goth g}$ is called a {\em good element of ${\goth g}$} if for some 
homogenous sequence $(\poi p1{,\ldots,}{\rg}{}{}{})$ in $\ai g{x}{}$, the nullvariety of 
$\poi p1{,\ldots,}{\rg}{}{}{}$ in $({\goth g}^{x})^{*}$ has codimension $\rg$ in 
$({\goth g}^{x})^{*}$. 
\end{defi}

For example, regular nilpotent elements are good. 
Indeed, for $e$ in the regular nilpotent orbit of $\g$ and $(q_1,\ldots q_\rg)$ a 
homogenous generating family of $\ai g{}{}$, it is well-known that 
$\ie{q_i}={\rm d}_e q_i$ for $i=1,\ldots,\rg$ 
and that $({\rm d}_e q_1,\ldots,{\rm d}_e q_\rg)$ 
forms a basis of $\g^{e}$, \cite{Ko}. Hence $e$ is good. 

Also, by \cite[Theorem 5.4]{PPY}, all nilpotent elements of a simple Lie algebra 
of type {\bf A} are good. Moreover, according to \cite[Corollary 8.2]{Y4},  
{\em even}\footnote{i.e., 
this means that the Dynkin grading of $\g$ associated with the nilpotent element has no 
odd term.} nilpotent elements without odd (respectively even) Jordan blocks of 
$\g$ are good if $\g$ is of type {\bf C} (respectively {\bf B} or {\bf D}). 
We generalize these results (cf.~Theorem~\ref{tca1}, Corollary \ref{cca2} 
and Remark \ref{rca2}). 

The good elements satisfy the polynomiality condition (cf.~Theorem~\ref{tge1}): 

\begin{theorem} \label{ti3}
Let $x$ be a good element of ${\goth g}$. Then $\ai gx{}$ is a polynomial algebra and 
$\es S{{\goth g}^{x}}$ is a free extension of $\ai gx{}$. 
\end{theorem}

Furthermore, $x$ is good if and only if so is its nilpotent 
component in the Jordan decomposition (cf.~Proposition~\ref{pge1}). 
As a consequence, we can restrict the study to the case of nilpotent elements. 

The main result of the paper 
is the following (cf.~Theorem~\ref{tge2}) whose proof is outlined in Subsection \ref{i4}:

\begin{theorem} \label{ti4}
Suppose that for some homogenous generators $\poi q1{,\ldots,}{\rg}{}{}{}$ of 
$\ai g{}{}$, the polynomial functions $\poi {\ie q}1{,\ldots,}{\rg}{}{}{}$ are 
algebraically independent. Then $e$ is a good element of ${\goth g}$. In particular, 
$\ai ge{}$ is a polynomial algebra and $\es S{{\goth g}^{e}}$ is a free extension of 
$\ai ge{}$. Moreover, $(\poi {\ie q}1{,\ldots,}{\rg}{}{}{})$ is a regular sequence in 
$\es S{{\goth g}^{e}}$.
\end{theorem}

In other words, Theorem \ref{ti4} says that Condition (1) of Theorem \ref{ti2} 
is sufficient to ensure the polynomiality of $S(\g^{e})^{\g^{e}}$. However, 
if only Condition (1) of Theorem \ref{ti2} is satisfied, the (polynomial) algebra  
$S(\g^{e})^{\g^{e}}$ is not necessarily generated by the polynomial functions 
$\poi {\ie q}1{,\ldots,}{\rg}{}{}{}$. 
As a matter of fact, there are nilpotent elements $e$ satisfying Condition 
(1) and for which $S(\g^{e})^{\g^{e}}$ is not generated by some 
$\poi {\ie q}1{,\ldots,}{\rg}{}{}{}$, for any 
choice of homogenous generators $\poi q1{,\ldots,}{\rg}{}{}{}$ of 
$\ai g{}{}$ (cf. Remark \ref{r5ca3}). 

Theorem \ref{ti4} can be applied to a great number of nilpotent orbits in the simple 
classical Lie algebras (cf.~Section~\ref{ca}), and for some nilpotent orbits in the 
exceptional Lie algebras (cf.~Section~\ref{et}). 
For example, according to \cite[Proposition 6.3]{PY} and Theorem \ref{ti4}, 
the elements of the subregular nilpotent orbit of $\g$ are good. 

To state our results for the simple Lie algebras of
types {\bf B} and {\bf D}, let us introduce some more notations. Assume that 
$\g={\goth {so}}(\V)\subset {\goth {gl}}(\V)$ for some vector space $\V$ of dimension 
$2\rg+1$ or $2\rg$. For an endomorphism $x$ of ${\Bbb V}$ and for 
$i\in\{1,\ldots,\dim {\Bbb V}\}$, denote by $Q_{i}(x)$ the coefficient of degree 
$\dim {\Bbb V}-i$ of the characteristic polynomial of $x$. Then for any $x$ in 
${\goth g}$, $Q_{i}(x)=0$ whenever $i$ is odd. Define a generating family 
$\poi q1{,\ldots,}{\rg}{}{}{}$ of the algebra $\ai g{}{}$ as follows. For 
$i=1,\ldots,\rg-1$, set $q_{i}:=Q_{2i}$. If $\dim {\Bbb V}=2\rg +1$, set 
$q_{\rg}:=Q_{2\rg}$, and if $\dim {\Bbb V}=2\rg$, let $q_\ell$ be the Pfaffian 
that is a homogenous element 
of degree $\rg$ of $\ai g{}{}$ such that $Q_{2\rg}=q_{\rg}^{2}$. 
Denote by $\poi {\delta }1{,\ldots,}{\rg}{}{}{}$ the degrees of 
$\poi {\ie q}1{,\ldots,}{\rg}{}{}{}$ respectively. By \cite[Theorem 2.1]{PPY}, if 
$$\dim\g^{e}+\rg-2(\poi {\delta }1{+ \cdots +}{\rg}{}{}{}) = 0,$$
then the polynomials $\poi {\ie q}1{,\ldots,}{\rg}{}{}{}$ are algebraically independent. 
In that event, by Theorem~\ref{ti4}, $e$ is good and we will say that $e$ is 
{\em very good} (cf.~Corollary \ref{cca2} and Definition \ref{d3ca2}).  
The notion of very good element is specific to this setting  
where there are natural generators of $\ai g{}{}$. 

The very good nilpotent elements of $\g$ can be characterized in term of their 
associated partitions of $\dim\V$ (cf.~Lemma \ref{l4ca2}).
Theorem~\ref{ti4} also allows to obtain examples of good, but not very good, nilpotent 
elements of $\g$; for them, there are a few more work to do (cf.~Subsection~\ref{ca3}). 

In this way, we obtain a large number of good nilpotent elements, including all even 
nilpotent elements in type {\bf B}, or in type {\bf D} with odd rank 
(cf.~Corollary \ref{cca2}). For the type {\bf D} with even rank, we obtain the statement
for some particular cases (cf.~Theorem \ref{tca3}). 
On the other hand, there are examples of elements that satisfy the polynomiality 
condition but that are not good; see Examples~\ref{erc2} and \ref{e2rc2}. 
To deal with them, we use different techniques, more similar to those used 
in \cite{PPY}. These alternative methods are presented in Section \ref{rc}. 

As a result of all this, we observe for example that all nilpotent elements of 
${\goth {so}}(\k^{7})$ are good, and that all nilpotent 
elements of ${\goth {so}}(\k^{n})$, with $n\leq 13$, satisfy the polynomiality 
condition (cf.~Table \ref{tabrc2}).  In particular, by \cite[\S3.9]{PPY}, this provides
examples of good nilpotent elements for which $\g^{e}$ is singular. For such nilpotent 
elements, note that \cite[Theorem 0.3]{PPY} (cf. Theorem \ref{ti2}) cannot be applied. 

Our results do not cover all nilpotent orbits in type {\bf B} and {\bf D}. 
As a matter of fact, we obtain a counter-example in type {\bf D}$_7$ to Premet's 
conjecture (cf.~Example~\ref{erc3}). 

\begin{prop} \label{pi}
The nilpotent elements of ${\goth {so}}(\k^{14})$ associated with the partition 
$(3,3,2,2,2,2)$ of $14$ do not satisfy the polynomiality condition. 
\end{prop}

\subsection{Outline of the proof of Theorem \ref{ti4}}\label{i4} 
Let $\poi q1{,\ldots,}{\rg}{}{}{}$ be homogenous generators of $\ai g{}{}$ of degrees
$\poi d1{,\ldots,}{\rg}{}{}{}$ respectively. The sequence 
$(\poi q1{,\ldots,}{\rg}{}{}{})$ is ordered so that 
$\poi d1{\leq \cdots \leq }{\rg}{}{}{}$. Assume that the polynomial functions 
$\poi {\ie q}1{,\ldots,}{\rg}{}{}{}$ are algebraically independent.

According to Theorem \ref{ti3}, it suffices to show that 
$e$ is good, and more accurately that the nullvariety of
$\poi {\ie q}1{,\ldots,}{\rg}{}{}{}$ in ${\goth g}^{f}$ has codimension $\rg$,  
since $\poi {\ie q}1{,\ldots,}{\rg}{}{}{}$ are invariant homogenous polynomials.  
To this end, it suffices to prove that 
$\es S{{\goth g}^{e}}$ is a free extension of  the $\k$-algebra generated by 
$\poi {\ie q}1{,\ldots,}{\rg}{}{}{}$ (see Proposition~\ref{pgf1},(ii)). 
We are led to find a subspace $V_0$ of $S$ such that the linear map
$$ \tk {\k}{V_0}\k[\poi {\ie q}1{,\ldots,}{\rg}{}{}{}] \longrightarrow S, 
\qquad v\tens a \longmapsto va$$
is a linear isomorphism. We explain below the construction of the subspace $V_0$.  

Let $\poi x1{,\ldots,}{r}{}{}{}$ be a basis of ${\goth g}^{e}$ such that for 
$i=1,\ldots,r$, $[h,x_{i}]=n_{i}x_{i}$ for some nonnegative integer $n_{i}$. For 
${\bf j}=(\poi j1{,\ldots,}{r}{}{}{})$ in ${\Bbb N}^{r}$, set:
$$ \vert {\bf j} \vert := \poi j1{+\cdots +}{r}{}{}{}, \qquad 
\vert {\bf j} \vert_{e} := j_{1}n_{1}+\cdots +j_{r}n_{r}+2\vert {\bf j} \vert, \qquad
x^{{\bf j}} = \poie x1{\cdots }{r}{}{}{}{j_{1}}{j_{r}}.$$
The algebra $\es S{{\goth g}^{e}}$ has two gradings: the standard one and the 
{\em Slodowy grading}. For all ${\bf j}$ in ${\Bbb N}^{r}$, $x^{{\bf j}}$ is 
homogenous with respect to these two gradings. It has standard degree 
$\vert {\bf j} \vert$ and, by definition, it has Slodowy degree 
$\vert {\bf j} \vert_{e}$. For $m$ nonnegative integer, denote by 
$\es S{{\goth g}^{e}}^{[m]}$ the subspace of $\es S{{\goth g}^{e}}$ of Slodowy degree $m$.

Let us simply denote by $S$ the algebra $\es S{{\goth g}^{e}}$ and let $t$ be an 
indeterminate. For any subspace $V$ of $S$, set:
$$ V[t] := \tk {\k}{\k[t]}V, \qquad V[t,t^{-1}] := \tk {\k}{\k[t,t^{-1}]}V, \qquad
V[[t]] := \tk {\k}{\k[[t]]}V, \qquad V((t)) := \tk {\k}{\k((t))}V ,$$
with $\k((t))$ the fraction field of $\k[[t]]$. For $V$ a subspace of $S[[t]]$, denote by
$V(0)$ the image of $V$ by the quotient morphism 
$$S[t]\longrightarrow S, \qquad a(t) \longmapsto a(0) .$$  

The Slodowy grading of $S$ induces a grading of the algebra $S((t))$ with $t$ 
having degree $0$. 
Let $\tau $ be the morphism of algebras 
$$ S \longrightarrow S[t], \qquad x_{i} \mapsto tx_{i}, \quad i=1,\ldots,r .$$
The morphism $\tau$ is a morphism of graded algebras. 
Denote by $\poi {\delta }1{,\ldots,}{\rg}{}{}{}$ the standard degrees of 
$\poi {\ie q}1{,\ldots,}{\rg}{}{}{}$ respectively, and set for $i=1,\ldots,\rg$:
$$Q_{i} := t^{-\delta _{i}} \tau (\kappa (q_{i})) \quad  \text{ with } \quad
\kappa (q_{i})(x) := q_{i}(e+x), \quad  \forall\,x\in {\goth g}^{f}.$$
Let $A$ be the subalgebra of $S[t]$ generated by 
$\poi Q1{,\ldots,}{\rg}{}{}{}$. Then $A(0)$ is the subalgebra of $S$ generated by 
$\poi {\ie q}1{,\ldots,}{\rg}{}{}{}$. 
For $(j_1,\ldots,j_\rg)$ in ${\Bbb N}^{\rg}$, 
$\poie q1{\cdots }{\rg}{\kappa }{}{}{j_{1}}{j_{\rg}}$ and 
$\poie {\ie q}1{\cdots }{\rg}{}{}{}{j_{1}}{j_{\rg}}$ 
are Slodowy homogenous of Slodowy degree $2d_{1}j_{1}+\cdots +2d_{\rg}j_{\rg}$ 
(cf.~\cite{Pr,PPY} or Proposition \ref{psg1},(i)). Hence, $A$ and $A(0)$ are graded subalgebras 
of $S[t]$ and $S$ respectively.
Denote by $A(0)_{+}$ the augmentation ideal of $A(0)$, and let $V_0$ be a graded 
complement to $SA(0)_{+}$ in $S$. The main properties of our data $A$ and $A(0)$ are 
the following ones:
\begin{itemize}
\item [{\rm (1)}] $A$ is a graded polynomial algebra, 
\item [{\rm (2)}] the canonical morphism $A\rightarrow A(0)$ is a homogenous isomorphism,
\item [{\rm (3)}] the algebra $S[t,t^{-1}]$ is a free extension of $A$,
\item [{\rm (4)}] the ideal $S[t,t^{-1}]A_{+}$ of $S[t,t^{-1}]$ is radical. 
\end{itemize}
With these properties we first obtain that $S[[t]]$ is a free extension of 
$A$ (cf.\ Corollary \ref{c2sg4}) and that $S[[t]]$ is a free extension of the 
subalgebra $\tilde{A}$ of $S[[t]]$ generated by $\k[[t]]$ and $A$ 
(cf.\ Theorem \ref{tsg5},(i)). From these results, we deduce that the linear map
$$ \tk {\k}{V_0}A(0) \longrightarrow S, \qquad v\tens a \longmapsto va$$
is a linear isomorphism, as expected; see Theorem \ref{tsg5},(iii). The key points of the
proof are Lemma~\ref{lsg1}, Lemma~\ref{l2sg2}, Proposition~\ref{psg3} and 
Corollary~\ref{c2sg4}.

\subsection{A related problem}
Let us now mention a recent result of T. Arakawa and A. Premet which resembles 
our results, \cite{AP}. 

Let $V^{\rm cri}(\g^{e})$ be the universal affine Vertex algebra associated with $\g^{e}$ 
at critical level, and let $Z(V^{\rm cri}(\g^{e}))$ be the center of 
$V^{\rm cri}(\g^{e})$. Assume that Conditions (1) et (2) of Theorem \ref{ti2} are 
satisfied. Then $S(\hat{\g}_-^{e})^{\hat{\g}_+^{e}}$ is a polynomial algebra, with 
$\hat{\g}^{e}_-:=\g^{e}[t^{-1}]t^{-1}$. Moreover, $Z(V^{\rm cri}(\g^{e}))$ is a 
polynomial algebra, and explicit generators can be described. 

The particular case where $e=0$ is an old result of B.~Feigin and E.~Frenkel, \cite{FF}. 
Arakawa and Premet have used {\em affine $W$-algebras} to prove the general case.  

It would be interesting to extend the results of Arakawa and Premet to the 
setting of Theorem \ref{ti4}, that is to the cases where only the Conditon (1) of 
Theorem \ref{ti2} is satisfied, at least to the cases where 
we have explicit generators of $S(\g)^{\g^{e}}$, not necessarily of the form 
$\ie{q_1},\ldots,\ie{q_\ell}$ for some generators $q_1,\ldots,q_\ell$ 
of $S(\g)^\g$; cf.~e.g.~Remark \ref{r5ca3}. 
 
\subsection{} \label{i5}
The remainder of the paper will be organized as follows. 

\small

Section \ref{gf} is about general facts on commutative algebra, useful for the Sections 
\ref{ge} and \ref{sg}. In Section~\ref{ge}, the notions of good elements and good orbits 
are introduced, and some properties of good elements are described. Theorem \ref{tge1} 
asserts that the good elements satisfy the polynomiality condition.
The main result (Theorem \ref{tge2}) is also stated in this section.
Section \ref{sg} is devoted to the proof of Theorem \ref{tge2}. 
In Section \ref{ca}, we give applications of Theorem \ref{tge2} to 
the simple classical Lie algebras. In Section \ref{et}, we give applications to the 
exceptional Lie algebras of types {\bf E}$_6$, {\bf F}$_4$ and {\bf G}$_2$. This allows  
us to exhibit a great number of good nilpotent 
orbits. Other examples, counter-examples, remarks and a conjecture are discussed 
in Section \ref{rc}. In this last section, other techniques are developed. 
  
\subsection*{Acknowledgments}
We thank Alexander Premet for 
his important comments on the previous version of this paper. 
We also thank the referee for careful reading and
thoughtful suggestions. 

This work was partially supported by the ANR-project 10-BLAN-0110.

\section{General facts on commutative algebra} \label{gf}
We state in this section preliminary results on commutative algebra. 
Theorem \ref{tgf4} will be particularly important in Sections~\ref{ge} 
for the proof of Theorem \ref{tge1}. As for Proposition \ref{pgf1}, 
it will be used in the proof of Theorem \ref{tge2}.    

\subsection{} \label{gf1}
As a rule, for $A$ a graded algebra over ${\Bbb N}$, we denote by $A_{+}$ the 
ideal of $A$ generated by its homogenous elements of positive degree. For $M$ a graded 
$A$-module, we set $M_{+}:=A_{+}M$. 

Let $S$ be a finitely generated regular graded $\k$-algebra over ${\Bbb N}$. 
If $E$ is a finite dimensional vector space over $\k$, we denote by ${\rm S}(E)$ 
the polynomial algebra generated by $E$. It is a finitely generated regular 
$\k$-algebra, graded over ${\Bbb N}$ by the standard grading. 
Let $A$ be a graded subalgebra of $S$, different from $S$ and 
such that $A=\k+A_{+}$.
Let $X_A$ and $X_S$ be the affine varieties ${\mathrm {Specm}}(A)$ and 
${\mathrm {Specm}}(S)$ respectively, and let 
$\pi_{A,S}$ be the morphism from $X_S$ to $X_A$ whose 
comorphism is the canonical injection from $A$ into $S$. 
Let ${\cal N}_{0}$ be the nullvariety of $A_{+}$ in 
$X_S$ and set 
$$N:=\dim S-\dim A.$$ 

The following lemma is well-known. It is an easy consequence of a Chevalley's theorem
\cite[Ch. II, Exercise 3.22]{Ha} for Assertions (i) and (ii), and of  
\cite[Ch. 5, Theorem 13.5]{Mat} for Assertion (iii). 
 
\begin{lemma}\label{lgf1}
{\rm (i)} The irreducible components of the fibers of $\pi_{A,S}$ have dimension at least
$N$.

{\rm (ii)} If ${\cal N}_{0}$ has dimension $N$, then the fibers of $\pi_{A,S}$ are  
equidimensional of dimension $N$.

{\rm (iii)} Suppose that $S={\rm S}(E)$ for some finite dimensional $\k$-vector space 
$E$. If ${\cal N}_{0}$ has dimension $N$, then for some 
$x_1,\ldots,x_N$ in $E$, the nullvariety of $x_1,\ldots,x_N$ in ${\cal N}_{0}$ is equal 
to $\{0\}$. 
\end{lemma}

Let $\overline{A}$ be the algebraic closure of $A$ in $S$.

\begin{lemma}\label{l2gf1}
Let $M$ be a graded $A$-module and let $V$ be a graded
subspace of $M$ such that $M=V\oplus M_{+}$. Denote by $\tau$ the canonical map
$\tk {\k}{A}V \longrightarrow M$. Then $\tau $ is surjective. Moreover, $\tau $
is bijective if and only if $M$ is a flat $A$-module.
\end{lemma}

\begin{proof}
Let $M'$ be the image of $\tau $.  
Since $M=V\oplus M_+ = V + A_+ M\subset  M' + A_+M$, we get by induction on $k$, 
$$ M \subset M' + A_{+}^{k}M .$$ 
Since $M$ is graded and since $A_{+}$ is generated by elements of positive degree, 
$M=M'$.

If $\tau $ is bijective, then all basis of $V$ is a basis of the $A$-module $M$. In 
particular, it is a flat $A$-module. Conversely, let us suppose that $M$ is 
a flat $A$-module. For $v$ in $M$, denote by $\overline{v}$ the element of $V$ such that
$v-\overline{v}$ is in $A_{+}M$. 

\begin{claim}\label{clgf1}
Let $(\poi v1{,\ldots,}{n}{}{}{})$ be a homogenous sequence in $M$ such that 
$\overline{v_{1}},\ldots,\overline{v_{n}}$ are linearly independent over $\k$. 
Then $\poi v1{,\ldots,}{n}{}{}{}$ are linearly independent over $A$.
\end{claim}

\begin{proof}[Proof of Claim~\ref{clgf1}] 
Since the sequence $(\poi v1{,\ldots,}{n}{}{}{})$ is homogenous, it suffices to prove 
that for a homogenous sequence $(\poi a1{,\ldots,}{n}{}{}{})$ in $A$,
$$ a_{1}v_{1}+\cdots +a_{n}v_{n} = 0 \Longrightarrow \poi a1{=\cdots =}{n}{}{}{} = 0 .$$ 
Prove the statement by induction on $n$. First of all, by flatness, for some
homogenous sequence $(\poi y1{,\ldots,}{k}{}{}{})$ in $M$ and for some homogenous 
sequence $(b_{i,j},i=1,\ldots,n,j=1,\ldots,k)$, 
$$ v_{i} = \sum_{j=1}^{k} b_{i,j}y_{j} \quad  \text{and} \quad 
\sum_{l=1}^{n} a_{l}b_{l,m} = 0$$
for $i=1,\ldots,n$ and $m=1,\ldots,k$. For $n=1$, since $\overline{v_{1}}\neq 0$, 
for some $j$, $b_{1,j}$ is in $\k^{*}$ since $A=\k+A_{+}$. So $a_{1}=0$. Suppose the 
statement true for $n-1$. Since $\overline{v_{n}}\neq 0$, for some $j$, $b_{n,j}$ is 
in $\k^{*}$, whence
$$ a_{n} = -\sum_{i=1}^{n-1} \frac{b_{i,j}}{b_{n,j}} a_{i} \quad  \text{and} \quad
\sum_{i=1}^{n-1} a_{i}(v_{i} - \frac{b_{i,j}}{b_{n,j}} v_{n}) = 0 .$$
Since $\overline{v_{1}},\ldots,\overline{v_{n}}$ are linearly independent over $\k$, so 
are the elements
$$ \left(\overline{v_{i} - (b_{i,j}/b_{n,j}) v_{n}}, \quad i=1,\ldots,n-1 \right) .$$
By induction hypothesis, $\poi a1{=\cdots =}{n-1}{}{}{}=0$, whence $a_{n}=0$.
\end{proof}

According to Claim~\ref{clgf1}, any homogenous basis of $V$ consists of linearly 
independent elements over $A$. Hence any homogenous basis of $V$ is a basis of the 
$A$-module $M$ since $M=AV$.
\end{proof}

\begin{coro}\label{cgf1}  
Suppose that $S={\rm S}(E)$ for some finite dimensional $\k$-vector space 
$E$, and suppose that $\dim {\cal N}_{0}=N$. Then $\overline{A}$ is the integral closure 
of $A$ in $\es SE$. In particular, $\overline{A}$ is finitely generated. 
\end{coro}

\begin{proof}
Since $A$ is finitely generated, so is its integral closure in $\es SE$ by
\cite[\S33, Lemma 1]{Mat}. According to the hypothesis on ${\cal N}_{0}$ and 
Lemma~\ref{lgf1},(iii), for some $\poi x1{,\ldots,}{N}{}{}{}$ in $E$, the nullvariety of 
$\poi x1{,\ldots,}{N}{}{}{}$ in ${\cal N}_{0}$ is equal to $\{0\}$. In particular, 
$\poi x1{,\ldots,}{N}{}{}{}$ are algebraically independent over $A$ since $E$ has 
dimension $N+\dim A$. Let $J$ be the ideal of $\es SE$ generated by $A_{+}$ and 
$\poi x1{,\ldots,}{N}{}{}{}$. Then the radical of $J$ is the augmentation ideal of 
$\es SE$ so that $J$ has finite codimension in $\es SE$. For $V$ a homogenous complement 
to $J$ in $\es SE$, $\es SE$ is the $A[\poi x1{,\ldots,}{N}{}{}{}]$-submodule generated by
$V$ by Lemma~\ref{l2gf1}. Hence $\es SE$ is a finite extension of 
$A[\poi x1{,\ldots,}{N}{}{}{}]$.

Let $p$ be in $\overline{A}$. Since $A[\poi x1{,\ldots,}{N}{}{}{}]$ is finitely 
generated, $A[\poi x1{,\ldots,}{N}{}{}{}][p]$ is a finite extension of 
$A[\poi x1{,\ldots,}{N}{}{}{}]$. Let 
$$ p^{m} + a_{m-1}p^{m-1} + \cdots + a_{0} = 0$$ 
an integral dependence equation of $p$ over $A[\poi x1{,\ldots,}{N}{}{}{}]$. For 
$i=0,\ldots,m$, $a_{i}$ is a polynomial in $\poi x1{,\ldots,}{N}{}{}{}$ with coefficients
in $A$ since $\poi x1{,\ldots,}{N}{}{}{}$ are algebraically independent over $A$. Denote
by $a_{i}(0)$ its constant coefficient. Since $p$ is in $\overline{A}$, 
$\poi x1{,\ldots,}{N}{}{}{}$ are algebraically independent over $A[p]$, whence
$$ p^{m} + a_{m-1}(0)p^{m-1} + \cdots + a_{0}(0) = 0 .$$  
As a result, $\overline{A}$ is the integral closure of $A$ in $\es SE$. 
\end{proof}

Most of the following proposition is contained in~\cite[Corollary 6.2.3]{Ben}. 
Since Proposition \ref{pgf1} is more extensive, we give a proof. 

\begin{prop}\label{pgf1}
Let us consider the following conditions on $A$:
\begin{itemize}
\item [{\rm 1)}] $A$ is a polynomial algebra,
\item [{\rm 2)}] $A$ is a regular algebra, 
\item [{\rm 3)}] $A$ is a polynomial algebra generated by $\dim A$ 
homogenous elements,  
\item [{\rm 4)}] the $A$-module $S$ is faithfully flat,
\item [{\rm 5)}] the $A$-module $S$ is flat,
\item [{\rm 6)}] the $A$-module $S$ is free.
\end{itemize}

{\rm (i)} The conditions {\rm (1)}, {\rm (2)}, {\rm (3)} are equivalent. 

{\rm (ii)} The conditions {\rm (4)}, {\rm (5)}, {\rm (6)} are equivalent. 
Moreover, Condition~{\rm (4)} implies Condition~{\rm (2)}  
and, in that event, ${\cal N}_{0}$ is equidimensional of dimension $N$.

{\rm (iii)} If ${\cal N}_{0}$ is equidimensional of dimension $N$, then the 
conditions~{\rm (1)}, {\rm (2)}, {\rm (3)}, {\rm (4)}, {\rm (5)},~{\rm (6)} are all 
equivalent. 
\end{prop}

\begin{proof}
Let $n$ be the dimension of $A$.

(i) The implications $(3)\!\Rightarrow\!(1)$, $(1)\!\Rightarrow\!(2)$  are 
straightforward. Let us suppose that $A$ is a regular algebra.  Since $A$ is graded and 
finitely generated, there exists a homogenous sequence $(x_1,\ldots,x_n)$ in $A_{+}$ 
representing a basis of $A_{+}/A_{+}^{2}$. Let $A'$ be the subalgebra of $A$ generated by
$x_1,\ldots,x_n$. Then 
$$A_{+} \subset A' + A_{+}^{2}.$$
So by induction on $m$,
$$A_{+} \subset A' + A_{+}^{m}$$
for all positive integer $m$. Then $A=A'$ since $A$ is graded and $A_{+}$ is generated by 
elements of positive degree. Moreover, $\poi x1{,\ldots,}{n}{}{}{}$ are algebraically 
independent over $\k$ since $A$ has dimension $n$. Hence $A$ is a polynomial algebra 
generated by $n$ homogenous elements.

(ii) The implications $(4)\!\Rightarrow\!(5)$, $(6)\!\Rightarrow\!(5)$ are 
straightforward and $(5)\!\Rightarrow\!(6)$ is a consequence of Lemma~\ref{l2gf1}.   

$(5)\!\Rightarrow\!(4)$: 
Recall that $x_{0}=A_{+}$. 
Let us suppose that $S$ is a flat $A$-module. Then $\pi_{A,S}$ is an open morphism 
whose image contains $x_{0}$. Moreover, $\pi (X_S)$ is stable under the action of 
${\mathrm {G}}_{{\mathrm {m}}}$. So $\pi_{A,S} $ is surjective. Hence, by 
\cite[Ch.\,3, Theorem 7.2]{Mat}, $S$ is a faithfully flat extension of $A$.

$(4)\!\Rightarrow\!(2)$:  
Since $S$ is regular and since $S$ is a faithfully flat extension of $A$, all finitely 
generated $A$-module has finite projective dimension. 
So by \cite[Ch.\,7, \S19, Lemma~2]{Mat}, the global dimension of $A$ is finite. Hence
by \cite[Ch.\,7, Theorem 19.2]{Mat}, $A$ is regular. 

If Condition (4) holds, by \cite[Ch.\,5, Theorem 15.1]{Mat}, the fibers of $\pi_{A,S} $ 
are equidimensional of dimension $N$. So ${\cal N}_{0}$ is equidimensional of dimension 
$N$.

(iii) Suppose that ${\cal N}_{0}$ is equidimensional of dimension $N$. By 
(i) and (ii), it suffices to prove that $(2)\!\Rightarrow\!(5)$. 
By~Lemma \ref{lgf1},(ii), the fibers of $\pi_{A,S}$ are equidimensional of dimension $N$.
Hence by \cite[Ch.\,8, Theorem 23.1]{Mat}, $S$ is a flat extension of $A$ since 
$S$ and $A$ are regular.
\end{proof}

\subsection{} \label{gf2}
We present in this paragraph some results about algebraic extensions, 
that are independent of Subsection \ref{gf1}. These results are used only in the proof 
of Proposition \ref{pgf3}. Our main reference is \cite{Mat}. For $A$ an algebra and 
${\goth p}$ a prime ideal of $A$, $A_{{\goth p}}$ denotes the localization of $A$ at 
${\goth p}$. 

\smallskip

Let $t$ be an indeterminate, and let $L$ be a field containing $\k$. Let $B$, $L_{1}$, 
$B_{1}$ satisfying the following conditions:

\smallskip

\begin{tabular}{ll}
{\rm (I)} & $L_{1}$ is an algebraic extension of $L(t)$ of finite degree,\\
{\rm (II)} & $L$ is algebraically closed in $L_{1}$,\\
{\rm (III)} & $B$ is a finitely generated subalgebra of $L$, $L$ is the
fraction field of $B$ and $B$ is integrally closed in $L$, \\
{\rm (IV)} & $B_{1}$ is the integral closure of $B[t]$ in $L_{1}$, \\
{\rm (V)} & $tB_{1}$ is a prime ideal of $B_{1}$.
\end{tabular}

\medskip

For $C$ a subalgebra of $L$, containing $B$, we set: 
$$R(C):=\tk {B}{C}B_{1},$$
and we denote by $\mu _{C}$ the canonical morphism 
$R(C)\rightarrow CB_{1}$. 
Since $C$ and $B_{1}$ are integral algebras, the morphisms $c\mapsto c\tens 1$ and 
$b\mapsto 1\tens b$ from $C$ and $B_{1}$ to $R(C)$ respectively are embeddings. So, 
$C$ and $B_{1}$ are identified to subalgebras of $R(C)$ by these embeddings.
We now investigate some properties of the algebras $R(C)$.

\begin{lemma}\label{lgf2}
Let $\mu_L$ be the canonical morphism $R(L)\rightarrow LB_{1}$.

{\rm (i)} The algebra $R(L)$ is reduced and $\mu_L$ is an isomorphism.

{\rm (ii)} The ideal $tLB_{1}$ of $LB_{1}$ is maximal. Furthermore $B_{1}$ is a finite 
extension of $B[t]$.

{\rm (iii)} The algebra $LB_{1}$ is the direct sum of $L$ and $tLB_{1}$.

{\rm (iv)} The ring $LB_{1}$ is integrally closed in $L_{1}$.
\end{lemma}

\begin{proof}
(i) Let $a$ be in the kernel of $\mu_L $. Since $L$ is the fraction field of $B$, for some
$b$ in $B$, $ba=1\tens \mu_L (ba)$ so that $ba=0$ and $a=0$. As a result, $\mu_L $
is an isomorphism and $R(L)$ is reduced since $LB_{1}$ is integral.

(ii) Since $t$ is not algebraic over $L$ and since $LB_{1}$ is integral over $L[t]$ by
Condition (IV), $tLB_{1}$ is strictly contained in $LB_{1}$. Let $a$ and $b$ be in 
$LB_{1}$ such that $ab$ is in $tLB_{1}$. By Condition (III), for some $c$ in 
$B\setminus \{0\}$, $ca$ and $cb$ are in $B_{1}$. So, by Condition (V), $ca$ or $cb$ is 
in $tB_{1}$. Hence $a$ or $b$ is in $tLB_{1}$. As a result, $tLB_{1}$ is a prime ideal 
and the quotient $Q$ of $LB_{1}$ by $tLB_{1}$ is an integral domain. Denote by $\iota $ 
the quotient morphism. Since $L$ is a field, the restriction of $\iota $ to $L$ is an 
embedding of $L$ into $Q$. According to Conditions (I) and (IV) and 
\cite[\S33, Lemma 1]{Mat}, $B_{1}$ is a finite extension of $B[t]$. Then $Q$ is a finite 
extension of $L$ and $tLB_{1}$ is a maximal ideal of $LB_{1}$. 

(iii) Since $L$ is algebraically closed in $L_{1}$, $Q$ and $L_{1}$ are 
linearly disjoint over $L$. So, $\tk {L}QL_{1}$ is isomorphic to the extension of $L_{1}$
generated by $Q$. Denoting this extension by $QL_{1}$, $QB_{1}$ is a subalgbera of 
$QL_{1}$ and we have the exact sequences
$$\xymatrix{ 0 \ar[r] & tLB_{1} \ar[r] & LB_{1} \ar[r] & Q \ar[r] & 0 \\
0 \ar[r] & tQB_{1} \ar[r] & QB_{1} \ar[r] & \tk LQQ \ar[r] & 0 \\
0 \ar[r] & tQB_{1} \ar[r] & tQB_{1}+LB_{1} \ar[r] & Q\tens 1 \ar[r] & 0 }$$
As a result, 
$$ Q \subset LB_{1} + tQB_{1}.$$
By (ii), $QB_{1}$ is a finite $L[t]$-module. So, by Nakayama's Lemma, for some $a$ in 
$L[t]$, $(1+ta)QB_{1}$ is contained in $LB_{1}$. As a result, $Q$ is contained in 
$L_{1}$, whence $Q=L$ since $L$ is algebraically closed in $L_{1}$. The assertion follows
since $Q$ is the quotient of $LB_1$ by $t LB_1$.

(iv) Let $a$ be in the integral closure of $LB_{1}$ in $L_{1}$ and let
$$ a^{m} + a_{m-1}a^{m-1} + \cdots + a_{0} = 0$$
an integral dependence equation of $a$ over $LB_{1}$. For some $b$ in $L\setminus \{0\}$,
$ba_{i}$ is in $B_{1}$ for $i=0,\ldots,m-1$. Then, by Condition (IV), $ba$ is in $B_{1}$ 
since it satisfies an integral dependence equation over $B_{1}$. As a result, $LB_{1}$
is integrally closed in $L_{1}$.
\end{proof}

Let $L_{2}$ be the Galois extension of $L(t)$ generated by $L_{1}$, and let $\Gamma $ be
the Galois group of the extension $L_{2}$ of $L(t)$. Denote by $B_{2}$ the integral 
closure of $B[t]$ in $L_{2}$. For $C$ subalgebra of $L$, containing $B$, set
$$ R_{2}(C) := \tk BCB_{2},$$
and denote by $\mu _{C,2}$ the canonical morphism $R_{2}(C)\rightarrow CB_{2}$. The 
action of $\Gamma $ in $B_{2}$ induces an action of $\Gamma $ in $R_{2}(C)$ given 
by $g.(c\tens b) = c\tens g(b)$.

\begin{lemma}\label{l2gf2}
Let $x$ be a primitive element of $L_{1}$, and let $\Gamma_x$ be the stabilizer 
of $x$ in $\Gamma$. 

{\rm (i)} The subfield $L_{1}$ of $L_{2}$ is the set of fixed points under the action of 
$\Gamma _{x}$ in $L_{2}$, and $B_{1}$ is the set of fixed points under the action of 
$\Gamma _{x}$ in $B_{2}$.

{\rm (ii)} For $C$ subalgebra of $L$ containing $B$, the canonical morphism 
$R(C)\rightarrow R_{2}(C)$ is an embedding and its image is the set of fixed points 
under the action of $\Gamma _{x}$ in $R_{2}(C)$.

{\rm (iii)} For $C$ subalgebra of $L$, containing $B$, $C[t]$ is embedded in $R(C)$ and 
$R_{2}(C)$. Moreover, $C[t]$ is the set of fixed points under the action of $\Gamma $
in $R_{2}(C)$.
\end{lemma}

\begin{proof}
(i) Let $L'_{1}$ be the
set of fixed points under the action of $\Gamma _{x}$ in $L_{2}$, 
$$L'_{1}=\{y \in L_2 \; \vert \; \Gamma_{x}.y=y\}.$$
Then $L_{1}$ is 
contained in $L'_{1}$, and $L_{2}$ is an extension of degree $\vert \Gamma _{x} \vert$ of 
$L'_{1}$. Since $x$ is a primitive element of $L_{1}$, the degree
of this extension is equal to $\vert \Gamma .x \vert$ so that $L_{2}$ is an extension of 
degree $\vert \Gamma _{x} \vert$. Hence $L'_{1}=L_{1}$. 

Since $B_{2}$ is the integral closure of $B[t]$ in $L_{2}$, $B_{2}$ is invariant under 
$\Gamma $. Moreover, the intersection of $B_{2}$ and $L_{1}$ is equal to $B_{1}$ by
Condition (IV). Hence $B_{1}$ is the set of fixed points under the action of 
$\Gamma _{x}$ in $B_{2}$.

(ii) For $a$ in $B_{2}$ and $b$ in $R_{2}(C)$, set:
$$ a^{\#} := \frac{1}{\vert \Gamma _{x} \vert} \sum_{g\in \Gamma _{x}} g(a), \quad
\overline{b} := \frac{1}{\vert \Gamma _{x} \vert} \sum_{g\in \Gamma _{x}} g.b .$$
Then $a\mapsto a^{\#}$ is a projection of $B_{2}$ onto $B_{1}$. Moreover, it is a 
morphism of $B_{1}$-module. Denote by $\iota $ the canonical morphism 
$R(C)\rightarrow R_{2}(C)$, and by $\varphi $ the morphism 
$$ R_{2}(C) \longrightarrow R(C), \quad c\tens a \longmapsto c\tens a^{\#} .$$
For $b$ in $R_{2}(C)$,
$$ \varphi (b) = \varphi (\overline{b}) \quad  \text{and} \quad 
\iota \rond \varphi (b)  = \overline{b}$$
Then $\varphi $ is a surjective morphism and the image of $\iota \rond \varphi $ is
the set of fixed points under the action of $\Gamma _{x}$ in $R_{2}(C)$. Moreover
$\iota $ is injective, whence the assertion. 

(iii) From the equalities
$$ R(C) = \tk {B[t]}{(\tk BCB[t])}B_{1} \quad  \text{and} \quad C[t] = \tk BCB[t]$$
we deduce that $R(C) = \tk {B[t]}{C[t]}B_{1}$. In the same way, 
$R_{2}(C) = \tk {B[t]}{C[t]}B_{2}$. Then, since $C[t]$ is an integral algebra, the 
morphism $c\mapsto c\tens 1$ is an embedding of $C[t]$ in $R(C)$ and $R_{2}(C)$. 
Moreover, $C[t]$ is invariant under the action of $\Gamma $ in $R_{2}(C)$. 

Let $a$ be in $R_{2}(C)$ invariant under $\Gamma $. Then $a$ has an expansion
$$ a = \sum_{i=1}^{k} c_{i}\tens b_{i}$$
with $\poi c1{,\ldots,}{k}{}{}{}$ in $C[t]$ and $\poi b1{,\ldots,}{k}{}{}{}$ in $B_{2}$.
Since $a$ is invariant under $\Gamma $,
$$ a = \frac{1}{\vert \Gamma  \vert} \sum_{g\in \Gamma } g.a = 
\frac{1}{\vert \Gamma  \vert} \sum_{g\in \Gamma } \sum_{i=1}^{k} c_{i}\tens g.b_{i} .$$
For $i=1,\ldots,k$, set: 
$$ b'_{i} := \frac{1}{\vert \Gamma  \vert} \sum_{g\in \Gamma } g.b_{i}$$
The elements $\poie b1{,\ldots,}{k}{}{}{}{\prime}{\prime}$ are in $B[t]$, and
$$ a = (\sum_{i=1}^{k} c_{i}b'_{i})\tens 1 \in C[t] ,$$
whence the assertion.
\end{proof}

From now on, we fix a finitely generated subalgebra $C$ of $L$ containing $B$. Denote by
${\goth n}$ the nilradical of $R(C)$.

\begin{lemma}\label{l3gf2}
Let ${\goth k}$ be the kernel of $\mu _{C,2}$ and let ${\goth n}_{2}$ be the nilradical
of $R_{2}(C)$.

{\rm (i)} The algebras $R(C)$ and $R_{2}(C)$ are finitely generated. Furthermore, they 
are finite extensions of $C[t]$. 

{\rm (ii)} For $a$ in ${\goth k}$, $ba=0$ for some $b$ in $B\setminus \{0\}$.

{\rm (iii)} The ideal ${\goth k}$ is the minimal prime ideal of $R_{2}(C)$ such that 
${\goth k}\cap B=\{0\}$. Moreover, ${\goth k}\cap B[t]=\{0\}$.

{\rm (iv)} The ideal ${\goth n}$ is the kernel of $\mu _{C}$. Moreover, 
${\goth n}_{2}={\goth k}$ and ${\goth n}$ is a prime ideal.

{\rm (v)} The local algebra $R(C)_{{\goth n}}$ is isomorphic to $L_{1}$.
\end{lemma}

\begin{proof}
(i) According to Lemma~\ref{l2gf2},(iii), $R(C)$ is an extension of $C[t]$ and 
$R(C)=\tk {B[t]}{C[t]}B_{1}$. Then, by Lemma~\ref{lgf2},(ii), $R(C)$ is 
a finite extension of $C[t]$. In particular, $R(C)$ is a finitely generated algebra since 
so is $C$. In the same way, $R_{2}(C)$ is a finite extension of $C[t]$ and it is 
finitely generated.

(ii) Let $a$ be in ${\goth k}$. Then $a$ has an expansion 
$$ a = \sum_{i=1}^{k} c_{i}\tens b_{i}$$
with $\poi c1{,\ldots,}{k}{}{}{}$ in $C$ and $\poi b1{,\ldots,}{k}{}{}{}$ in $B_{2}$.
Since $C$ and $B$ have the same fraction field, for some $b$ in $B\setminus \{0\}$, 
$bc_{i}$ is in $B$, whence
$$ ba = 1\tens (\sum_{i=1}^{k}bc_{i}b_{i}) .$$
So $ba=0$ since ${\goth k}$ is the kernel of $\mu _{C,2}$. 

(iii) By (i) there are finitely many minimal prime ideals
of $R_{2}(C)$. Denote them by $\poi {{\goth p}}1{,\ldots,}{k}{}{}{}$. Since $C[t]$ is an 
integral algebra, ${\goth n}_{2}\cap C[t] = \{0\}$ so that ${\goth p}_{i}\cap C = \{0\}$ 
for some $i$. Let $i$ be such that ${\goth p}_{i}\cap B=\{0\}$ and let $a$ be in 
${\goth k}$. By (ii), for some $b$ in $B\setminus \{0\}$, $ba$ is in ${\goth p}_{i}$. 
Hence ${\goth k}$ is contained in ${\goth p}_{i}$. Since $CB_{2}$ is an integral algebra,
${\goth k}$ is a prime ideal. Then ${\goth p}_{i}={\goth k}$ since ${\goth p}_{i}$ is a 
minimal prime ideal, whence the assertion since for some $j$, 
${\goth p}_{j}\cap C[t]=\{0\}$.

(iv) By (iii), there is only one minimal prime ideal of $R_{2}(C)$ whose intersection 
with $B$ is equal to $\{0\}$. So, it is invariant under $\Gamma $. Hence ${\goth k}$
is invariant under $\Gamma $. As a result, for $a$ in ${\goth k}$,
$$ 0 = \prod_{g\in \Gamma } (a - g.a) = a^{m} + a_{m-1}a^{m-1} + \cdots + a_{0}$$
with $m=\vert \Gamma  \vert$ and $\poi a0{,\ldots,}{m-1}{}{}{}$ in ${\goth k}$. Moreover,
by Lemma~\ref{l2gf2},(iii), $\poi a0{,\ldots,}{m-1}{}{}{}$ are in $C[t]$. So, by (iii),
they are all equal to zero so that $a$ is a nilpotent element. Hence ${\goth k}$
is contained in ${\goth n}_{2}$. Then ${\goth n}_{2}={\goth k}$ by (iii).

By Lemma~\ref{l2gf2},(ii), $R(C)$ identifies with a subalgebra of $R_{2}(C)$ so that
${\goth n}={\goth n}_{2}\cap R(C)$, and $\mu _{C}$ is the restriction of 
$\mu _{C,2}$ to $R(C)$. Hence ${\goth n}$ is the kernel of $\mu _{C}$ and 
${\goth n}$ is a prime ideal of $R(C)$.

(v) By (iii), ${\goth n}\cap C = \{0\}$. So, by (ii), ${\goth n}R(C)_{{\goth n}}=\{0\}$.
As a result, $R(C)_{{\goth n}}$ is a field since ${\goth n}R(C)_{{\goth n}}$ is a maximal
ideal of $R(C)_{{\goth n}}$. Moreover, by (iii), it is isomorphic to a subfield of 
$L_{1}$, containing $B_{1}$. So, $R(C)_{{\goth n}}$ is isomorphic to $L_{1}$.
\end{proof}

For $c$ in $L[t]$, denote by $c(0)$ the constant term of $c$ as a polynomial in $t$ with 
coefficients in $L$.

\begin{lemma}\label{l4gf2}
Assume that $C$ is integrally closed in $L$. Denote
by $\overline{CB_{1}}$ the integral closure of $CB_{1}$ in $L_{1}$. 

{\rm (i)} Let $i\in\{1,2\}$. For all positive integer $j$, the intersection of $C[t]$ and 
$t^{j}LB_{i}$ equals $t^{j}C[t]$.

{\rm (ii)} The intersection of $tLB_{1}$ and $\overline{CB_{1}}$ equals 
$t\overline{CB_{1}}$.

{\rm (iii)} The algebra $\overline{CB_{1}}$ is contained in $C+t\overline{CB_{1}}$.

{\rm (iv)} The algebra $B_{1}$ is the direct sum of $B$ and $tB_{1}$.
\end{lemma}

\begin{proof}
First of all, $CB_{1}$ and $\overline{CB_{1}}$ are finite extensions of $C[t]$ by 
Lemma~\ref{l2gf2},(i), and ~\cite[\S 33, Lemma 1]{Mat}. So $\overline{CB_{1}}$ is the
integral closure of $C[t]$ in $L_{1}$ by Condition (IV). Denote by $\overline{CB_{2}}$ 
the integral closure of $C[t]$ in $L_{2}$. Since $C$ is integrally closed in $L$, 
$C[t]$ is integally closed in $L[t]$. Hence $C[t]$ is the set of fixed points under
the action of $\Gamma $ in $\overline{CB_{2}}$. Let $a$ be in $\overline{CB_{2}}$. Then
$$ 0 = \prod_{g\in \Gamma } (a-g(a)) = a^{m} + a_{m-1} a^{m-1} + \cdots + a_{0}$$ 
with $\poi a0{,\ldots,}{m-1}{}{}{}$ in $C[t]$. 

(i) Since $t^{j}LB_{1}$ is contained in $t^{j}LB_{2}$ and contains $t^{j}C[t]$, it 
suffices to prove the assertion for $i=2$. Let us prove it by induction on $j$. Let $c$ 
be in $C[t]$. Then $c-c(0)$ is in $tLB_{2}$. By Lemma~\ref{lgf2},(ii), 
$L\cap tLB_{2}=\{0\}$ since $L$ is a field, whence $C\cap tLB_{2}=\{0\}$ since $C$ is 
contained in $L$. As a result, if $c$ is in $tLB_{2}$, $c(0)=0$ and $c$ is in $tC[t]$, 
whence the assertion for $j=1$. Suppose the assertion true for $j-1$. Let $c$ be in 
$C[t]\cap t^{j}LB_{2}$. By induction hypothesis, $c=t^{j-1}c'$ with $c'$ in $C[t]$. Then 
$c'$ is in $C[t]\cap tLB_{2}$, whence $c$ is in $t^{j}C[t]$ by the assertion for $j=1$. 

(ii) Suppose that $a$ is in $tLB_{1}$. Since $tLB_{2}$ is invariant under $\Gamma $,
for $i=0,\ldots,m-1$, $a_{i}$ is in $t^{m-i}LB_{2}$. Set for $i=0,\ldots,m-1$, 
$$ a'_{i} := \frac{a_{i}}{t^{m-i}}.$$
Then by (i), $\poie a0{,\ldots,}{{m-1}}{}{}{}{\prime}{\prime}$ are in 
$C[t]$. Moreover, 
$$ (\frac{a}{t})^{m} + a'_{m-1}(\frac{a}{t^{m-1}})^{m-1} + \cdots +
a'_{0} = 0,$$
so that $a/t$ is in $\overline{CB_{1}}$, whence the assertion.

(iii) Suppose that $a$ is in $\overline{CB_{1}}$. By Lemma~\ref{lgf2},(iii), $L$ is the 
quotient of $LB_{1}$ by $tLB_{1}$. So, denoting by $\overline{a}$ the image of $a$ by the
quotient morphism,
$$ \overline{a}^{m} + a_{m-1}(0) \overline{a}^{m-1} + \cdots + a_{0}(0) = 0.$$
Then $\overline{a}$ is in $C$ since $C$ is integrally closed. Hence $a$ is in 
$C+tLB_{1}$. As a result, by (ii), $\overline{CB_{1}}$ is contained in 
$C+t\overline{CB_{1}}$.

(iv) By Condition (III), $B$ is integrally closed in $L$. So the assertion results 
from (iii) and Condition (IV) for $C=B$.
\end{proof}

\begin{coro}\label{cgf2}
The ideal $R(C)t$ of $R(C)$ is prime and $t$ is not a zero divisor in $R(C)$. 
\end{coro}

\begin{proof}
According to Lemma~\ref{l4gf2},(iv), $R(C) = C+R(C)t$. Furthermore, this sum is
direct since $C\cap tCB_{1}=\{0\}$ by Lemma~\ref{lgf2},(ii) and since the restriction of 
$\mu _{C}$ to $C$ is injective. Then $R(C)t$ is a prime ideal of $R(C)$ since $C$ is an 
integral algebra.

Since $R(C)t$ is a prime ideal, ${\goth n}$ is contained in $R(C)t$. According to 
Lemma~\ref{l3gf2},(iv), ${\goth n}$ is the kernel of $\mu _{C}$. Let $a$ be in 
${\goth n}$. Then $a=a't$ for some $a'$ in $R(C)$. Since $0=\mu _{C}(a't)=\mu _{C}(a')t$,
$a'$ is in ${\goth n}$. As a result, by induction on $m$, for all positive integer $m$, 
$a=a_{m}t^{m}$ for some $a_{m}$ in ${\goth n}$.

For $k$ positive integer, denote by $J_{k}$ the subset of elements $a$ of $R(C)$ such 
that $at^{k}=0$. Then $(\poi J1{,}{2}{}{}{},\ldots)$ is an increasing sequence of ideals 
of $R(C)$. For $a$ in $J_{k}$, $0=\mu _{C}(at^{k})=\mu _{C}(a)t^{k}$. Hence $J_{k}$ is 
contained in ${\goth n}$. According to Lemma~\ref{l3gf2},(i), the $\k$-algebra $R(C)$ is 
finitely generated. So for some positive integer $k_{0}$, $J_{k}=J_{k_{0}}$ for all $k$ 
bigger than $k_{0}$. Let $a$ be in $J_{1}$. Then $a = a_{k_{0}}t^{k_{0}}$ for some 
$a_{k_{0}}$ in ${\goth k}$. Since $a_{k_{0}}t^{k_{0}+1}=0$, $a_{k_{0}}$ is in $J_{k_{0}}$
so that $a=0$. Hence $t$ is not a zero divisor in $R(C)$. 
\end{proof}

\begin{prop}\label{pgf2}
Suppose that $C$ is integrally closed and Cohen-Macaulay. Let ${\goth p}$ be 
a prime ideal of $CB_{1}$, containing $t$ and let $\tilde{{\goth p}}$ be its inverse 
image by $\mu _{C}$.

{\rm (i)} The local algebra $(CB_{1})_{{\goth p}}$ is normal.

{\rm (ii)} The local algebra $R(C)_{\tilde{{\goth p}}}$ is Cohen-Macaulay and reduced.
In particular, the canonical morphism 
$R(C)_{\tilde{{\goth p}}}\rightarrow (CB_{1})_{{\goth p}}$ is an isomorphism.

{\rm (iii)} The local algebra $(CB_{1})_{{\goth p}}$ is Cohen-Macaulay.
\end{prop}

\begin{proof}
(i) Let $\overline{CB_{1}}$ be the integral closure of $CB_{1}$ in $L_{1}$. Setting 
$S:=CB_{1}\setminus {\goth p}$, $(CB_{1})_{{\goth p}}$ is the localization of 
$CB_{1}$ with respect to $S$. Denote by $(\overline{CB_{1}})_{{\goth p}}$ the localization
of $\overline{CB_{1}}$ with respect to $S$. Then $(\overline{CB_{1}})_{{\goth p}}$ is a 
finite $(CB_{1})_{{\goth p}}$-module since $\overline{CB_{1}}$ is a finite extension 
of $CB_{1}$. According to Lemma~\ref{l3gf2},(iii),
$$ \overline{CB_{1}} \subset CB_{1} + t\overline{CB_{1}} .$$
Then since $t$ is in ${\goth p}$, 
$$ (\overline{CB_{1}})_{{\goth p}}/(CB_{1})_{{\goth p}} = 
{\goth p}(\overline{CB_{1}})_{{\goth p}}/(CB_{1})_{{\goth p}}.$$
So, by Nakayama's Lemma, $(\overline{CB_{1}})_{{\goth p}}=(CB_{1})_{{\goth p}}$, whence
the assertion.

(ii) According to Corollary~\ref{cgf2}, $R(C)t$ is a prime ideal containing ${\goth n}$. 
Denote by $\overline{{\goth p}}$ the intersection of ${\goth p}$ and $C$. Since 
$\tilde{{\goth p}}$ is the inverse image of ${\goth p}$ by $\mu _{C}$, 
$C_{\overline{{\goth p}}}$ is the quotient of $R(C)_{\tilde{{\goth p}}}$ by 
$R(C)_{\tilde{{\goth p}}}t$. Since $C$ is Cohen-Macaulay, so is 
$C_{\overline{{\goth p}}}$. As a result, $R(C)_{\tilde{{\goth p}}}$ is 
Cohen-Macaulay since $t$ is not a zero divisor in $R(C)$ by Corollary~\ref{cgf2} and
since $R(C)_{\tilde{{\goth p}}}t$ is a prime ideal of height $1$.

Denote by $\mu _{C,\tilde{{\goth p}}}$ the canonical extension of $\mu _{C}$ to
$R(C)_{\tilde{{\goth p}}}$. Then $(CB_{1})_{{\goth p}}$ is the image of 
$\mu _{C,\tilde{{\goth p}}}$. According to Lemma~\ref{l3gf2},(iv), the nilradical 
${\goth n}R(C)_{\tilde{{\goth p}}}$ of $R(C)_{\tilde{{\goth p}}}$ is the minimal
prime ideal of $R(C)_{\tilde{{\goth p}}}$ and it is the kernel of 
$\mu _{C,\tilde{{\goth p}}}$. By Lemma~\ref{l3gf2},(v), the localization of 
$R(C)_{\tilde{{\goth p}}}$ at ${\goth n}R(C)_{\tilde{{\goth p}}}$ is a field. In 
particular, it is regular. Then, by \cite[\S 1, Proposition 15]{Bou}, 
$R(C)_{\tilde{{\goth p}}}$ is a reduced algebra since it is Cohen-Macaulay. As a result, 
$\mu _{C,\tilde{{\goth p}}}$ is an isomorphism onto $(CB_{1})_{{\goth p}}$.

(iii) results from (ii).
\end{proof}

\subsection{} \label{gf3}
Return to the situation of Subsection \ref{gf1}, and keep its notations. 
From now on, and until the end of the section, we assume that 
$S={\rm S}(E)$ for some finite dimensional $\k$-vector space $E$. 
As a rule, if $B$ is a subalgebra of ${\rm S}(E)$, we denote by $K(B)$ its fraction 
field, and we set for simplicity 
$$K:=K({\rm S}(E)).$$ 
Furthermore we assume until the end of the section that the following conditions hold:
\begin{itemize}
\item[(a)] $\dim {\cal N}_{0}=N$, 
\item[(b)] $A$ is a polynomial algebra, 
\item[(c)] $K(\overline{A})$ is algebraically closed in $K$.
\end{itemize} 
We aim to prove Theorem~\ref{tgf4} (see Subsection \ref{gf4}). 
Let $(\poi v1{,\ldots,}{N}{}{}{})$ be a sequence of elements of $E$ such that 
its nullvariety in ${\cal N}_{0}$ equals $\{0\}$. Such a sequence does exist by 
Lemma~\ref{lgf1},(iii). Set 
$$C :=\overline{A}[\poi v1{,\ldots,}{N}{}{}{}].$$
By Proposition~\ref{pgf1},(ii), 
$C$ is a polynomial algebra if and only if so is $\overline{A}$ since $C$ is a 
faithfully flat extension of $\overline{A}$. Therefore, in order to prove 
Theorem~\ref{tgf4}, it suffices to prove that S$(E)$ is a free extension of $C$, again by
Proposition~\ref{pgf1},(ii). This is now our goal.

\smallskip

Condition (c) is actually not useful for the following lemma: 

\begin{lemma}\label{lgf3}
The algebra $C$ is integrally closed and $\es SE$ is the integral closure of $C$ in $K$.
\end{lemma}

\begin{proof}
Since $\overline{A}$ has dimension $\dim E-N$ and since the nullvariety
of $\poi v1{,\ldots,}{N}{}{}{}$ in ${\cal N}_{0}$ is $\{0\}$, 
$\poi v1{,\ldots,}{N}{}{}{}$ are algebraically independent over $A$ and 
$\overline{A}$. By Serre's normality criterion \cite[\S 1, \no 10, Th\'eor\`eme 4]{Bou}, 
any polynomial algebra over a normal ring is normal. So $C$ is integrally closed since 
so is $\overline{A}$ by definition. Moreover, $C$ is a homogenous finitely generated  
subalgebra of $\es SE$ since so is $\overline{A}$ by Corollary~\ref{cgf1}. Since $C$ has
dimension $\dim E$, $\es SE$ is algebraic over $C$. Then, by Corollary~\ref{cgf1}, 
$\es SE$ is the integral closure of $C$ in $K$. Indeed, $\es SE$ is integrally closed as 
a polynomial algebra and $\{0\}$ is the nullvariety of $C_{+}$ in $E^{*}$.
\end{proof}
\smallskip

Set $Z_{0} := {\mathrm {Specm}}(\overline{A})$ and $Z:=Z_{0}\times \k^{N}$. 
Then $Z$ is equal to ${\mathrm {Specm}}(C)$. Let ${X}_{0}$ be a desingularization of 
$Z_{0}$ and let ${\pi} _{0}$ be the morphism of desingularization. Such a 
desingularization does exist by \cite{Hi}. Set $X := X_{0}\times \k^{N}$ and 
denote by $\pi $ the morphism 
$$X\longrightarrow Z , \qquad (x,v) \longmapsto (\pi _{0}(x),v) .$$   
Then $(X,\pi )$ is a desingularization of $Z$.

Fix $x_{0}$ in $\pi _{0}^{-1}(C_{+})$. For $i=0,\ldots,N$, set 
$X_{i} := X_{0}\times \k^{i}$ and let $x_{i} := (x_{0},0_{ \k^{i}})$. Define $K_{i}$, 
$C'_{i}$, $C_{i}$ by the induction relations:
\begin{list}{}{}
\item {\rm (1)} $C'_{0} := C_{0} := \overline{A}$ and $K_{0}$ is the fraction field of 
$\overline{A}$,
\item {\rm (2)} $C'_{i}:=C'_{i-1}[v_{i}]$,
\item {\rm (3)} $K_{i}$ is the algebraic closure of $K_{i-1}(v_{i})$ in $K$ and $C_{i}$ 
is the integral closure of $C_{i-1}[v_{i}]$ in $K_{i}$.
\end{list}

\begin{lemma}\label{l2gf3}
Let $i=1,\ldots,N$.

{\rm (i)} The field $K_{i}$ is a finite extension of $K_{i-1}(v_{i})$ and $K_{i-1}$ is 
algebraically closed in $K_{i}$.

{\rm (ii)} The algebra $C_{i}$ is finitely generated and integrally closed in $K$. 
Moreover, $K_{i}$ is the fraction field of $C_{i}$. 

{\rm (iii)} The algebra $C_{i}$ is contained in $\es SE$ and $C_{N}=\es SE$. Moreover, 
$K_{N}=K$.

{\rm (iv)} The algebra $C_{i}$ is a finite extension of $C'_{i}$.

{\rm (v)} The algebra $C_{i}$ is the intersection of $\es SE$ and $K_{i}$. Moreover, 
$v_{i}C_{i}$ is a prime ideal of $C_{i}$.
\end{lemma}

\begin{proof}
(i) By Condition (c), $K_{0}$ is algebraically closed in $K$. So $K_{0}$ is 
algebraically closed in $K_{1}$. By definition, for $i>1$, $K_{i-1}$ is algebraically
closed in $K$. So it is in $K_{i}$. Since the nullvariety of $\poi v1{,\ldots,}{N}{}{}{}$
in ${\cal N}_{0}$ equals $\{0\}$, $\poi v1{,\ldots,}{N}{}{}{}$ are algebraically 
independent over $K_{0}$. Hence $K_{i-1}(\poi vi{,\ldots,}{N}{}{}{})$ is a field of 
rational fractions over $K_{i-1}$. Moreover, $K$ is an algebraic extension of 
$K_{i-1}(\poi vi{,\ldots,}{N}{}{}{})$ by Lemma~\ref{lgf3}. Since $\es SE$ is a finitely 
generated $\k$-algebra, $K$ is a finite extension of 
$K_{i-1}(\poi vi{,\ldots,}{N}{}{}{})$. By definition, $K_{i}$ is the algebraic closure of
$K_{i-1}(v_{i})$ in $K$. Hence $K_{i}$ is a finite extension of $K_{i-1}(v_{i})$.

(ii) Prove the assertion by induction on $i$. By definition, it is true for $i=0$ and 
$C_{i}$ is the integral closure of $C_{i-1}[v_{i}]$ in $K_{i}$ for $i=1,\ldots,N$, 
whence the assertion by (i) and \cite[\S33, Lemma 1]{Mat}.

(iii) Since $\es SE$ is integrally closed in $K$, $C_{i}$ is contained in $\es SE$
by induction on $i$. By definition, the field $K_{N}$ is algebraically closed in 
$K$ and it contains $C$. So $K_{N}=K$ by Lemma~\ref{lgf3}. Since $C_{N}$ is integrally 
closed in $K_{N}$ and it contains $C$, $C_{N}=\es SE$ by Lemma~\ref{lgf3}.

(iv) Prove the assertion by induction on $i$. By definition, it is true for $i=0$. Suppose
that it is true for $i-1$. Then $C_{i}$ is a finite extension of $C'_{i-1}[v_{i}]=C'_{i}$.

(v) Prove by induction on $i$ that $C_{N-i}$ is the intersection of $\es SE$ and 
$K_{N-i}$ for $i=0,\ldots,N$. By (iii), it is true for $i=0$. Suppose that it is true
for $i-1$. By induction hypothesis, it suffices to prove that $C_{N-i}$ is the 
intersection of $C_{N-i+1}$ and $K_{N-i}$. Let $a$ be in this intersection. Then 
$a$ satisfies an integral dependence equation over $C_{N-i}[v_{N-i+1}]$:
$$ a^{m}+a_{m-1}a^{m-1}+\cdots + a_{0} = 0 .$$
Denoting by $a_{j}(0)$ the constant term of $a_{j}$ as a polynomial in $v_{N-i+1}$ with
coefficients in $C_{N-i}$,
$$ a^{m}+a_{m-1}(0)a^{m-1}+\cdots + a_{0}(0) = 0 $$
since $a$ is in $K_{N-i}$ and $v_{N-i+1}$ is algebraically independent over $K_{N-i}$. 
Hence $a$ is in $C_{N-i}$ since $C_{N-i}$ is integrally closed in $K_{N-i}$ by (ii).

Let $a$ and $b$ be in $C_{i}$ such that $ab$ is in $v_{i}C_{i}$. Since $v_{i}$ is in 
$E$, $v_{i}\es SE$ is a prime ideal of $\es SE$. So $a$ or $b$ is in $v_{i}\es SE$ since
$C_{i}$ is contained in $\es SE$. Hence $a/v_{i}$ or $b/v_{i}$ are in the intersection of
$\es SE$ and $K_{i}$. So $a$ or $b$ is in $v_{i}C_{i}$.
\end{proof}

\begin{rema}\label{rgf3}
According to Lemma~\ref{l2gf3},(i),(ii),(iv), for $i=1,\ldots,N$, $K_{i-1}$, $v_{i}$, 
$C_{i-1}$, $K_{i}$, $C_{i}$ satisfy Conditions (I), (II), (III), (V) satisfed by 
$L$, $t$, $B$, $L_{1}$, $B_{1}$ in Subsection {\ref{gf2}}. Moreover, Condition (IV) is 
satisfied by construction (cf.~Lemma \ref{l2gf3},(v)). 
\end{rema}

\begin{prop}\label{pgf3}
Let $i=1,\ldots,N$. 

{\rm (i)} The semi-local algebra $\an {X_{i}}{x_{i}}C_{i}$ is normal and 
Cohen-Macaulay. 

{\rm (ii)} The canonical morphism 
$\tk {C'_{i}}{\an {X_{i}}{x_{i}}}C_{i}\rightarrow \an {X_{i}}{x_{i}}C_{i}$ is an 
isomorphism.
\end{prop}

\begin{proof}
(i)  The local ring $\an {X_{i}}{x_{i}}$ is an extension of $C'_{i}$ and $C_{i}$ is a 
finite extension of $C'_{i}$ by Lemma~\ref{l2gf3},(iv). So $\an {X_{i}}{x_{i}}C_{i}$ is a
semi-local ring as a finite extension of the local ring $\an {X_{i}}{x_{i}}$. Prove
the assertion by induction on $i$. For $i=0$, 
$\an {X_{0}}{x_{0}}C_{0}=\an {X_{0}}{x_{0}}$ and $\an {X_{0}}{x_{0}}$ is a regular local
algebra. Suppose that it is true for $i-1$ and set 
${\goth A}_{i-1}:= \an {X_{i-1}}{x_{i-1}}C_{i-1}$. Then ${\goth A}_{i-1}$ is a subalgebra 
of $K_{i-1}$ since $\an {X_{i-1}}{x_{i-1}}$ is contained in the fraction field of 
$C'_{i-1}$. Let ${\goth m}$ be a maximal ideal of
$\an {X_{i}}{x_{i}}C_{i}$.  The local ring $\an {X_{i}}{x_{i}}$ is the localization 
of $\an {X_{i-1}}{x_{i-1}}[v_{i}]$ at ${\goth m}\cap \an {X_{i-1}}{x_{i-1}}[v_{i}]$.
Hence $v_{i}$ is in ${\goth m}$, and 
${\goth m}\cap {\goth A}_{i-1}C_{i}$ is a prime ideal of ${\goth A}_{i-1}C_{i}$ such that
the localization of ${\goth A}_{i-1}C_{i}$ at this prime ideal is the localization of 
$\an {X_{i}}{x_{i}}C_{i}$ at ${\goth m}$. 
By the induction hypothesis, ${\goth A}_{i-1}$ is normal and Cohen-Macaulay. According to
Remark~\ref{rgf3} and Proposition~\ref{pgf2},(i) and (iii), the localization of 
${\goth A}_{i-1}C_{i}$ at ${\goth m}\cap {\goth A}_{i-1}C_{i}$ is normal and 
Cohen-Macaulay, whence the assertion.

(ii) Prove the assertion by induction on $i$. For $i=0$, $C_{0}$ is contained in 
$\an {X_{0}}{x_{0}}$. Suppose that it is true for $i-1$. For $j\in\{i-1,i\}$, denote by 
$\nu _{j}$ the canonical morphism 
$$\tk {C'_{j}}{\an {X_{j}}{x_{j}}}C_{j} \longrightarrow \an {X_{j}}{x_{j}}C_{j} .$$ 
Recall that ${\goth A}_{i-1} := \an {X_{i-1}}{x_{i-1}}C_{i-1}$. By induction hypothesis, 
the morphism 
$\nu _{i-1}\tens {\mathrm {id}}_{C_{i}}$,
$$ \tk {C_{i-1}}{(\tk {C'_{i-1}}{\an {X_{i-1}}{x_{i-1}}}C_{i-1})}C_{i} 
\longrightarrow \tk {C_{i-1}}{{\goth A}_{i-1}}C_{i}$$
is an isomorphism. Since $C'_{i-1}$ is contained in $\an {X_{i-1}}{x_{i-1}}$,
$$\tk {C'_{i-1}}{\an {X_{i-1}}{x_{i-1}}}C'_{i-1}[v_{i}]=\an {X_{i-1}}{x_{i-1}}[v_{i}].$$
Furthermore,
$$ \tk {C_{i-1}}{(\tk {C'_{i-1}}{\an {X_{i-1}}{x_{i-1}}}C_{i-1})}C_{i} =
\tk {C'_{i-1}}{\an {X_{i-1}}{x_{i-1}}}C_{i} = \tk {C'_{i-1}[v_{i}]}
{(\tk {C'_{i-1}}{\an {X_{i-1}}{x_{i-1}}}C'_{i-1}[v_{i}])}C_{i} ,$$
whence an isomorphism 
$$ \tk {C'_{i-1}[v_{i}]}{\an {X_{i-1}}{x_{i-1}}[v_{i}]}C_{i} \longrightarrow 
\tk {C_{i-1}}{{\goth A}_{i-1}}C_{i}.$$
Let ${\goth m}$ be as in (i). Set 
$${\goth p}:={\goth m}\cap {\goth A}_{i-1}C_{i}, \qquad 
\tilde{{\goth m}} := \nu _{i}^{-1}({\goth m}), $$ 
and denote by $\tilde{{\goth p}}$ the inverse image of ${\goth p}$ by the canonical 
morphism 
$$ \tk {C_{i-1}}{{\goth A}_{i-1}}C_{i} \longrightarrow {\goth A}_{i-1}C_{i} .$$
According to Proposition~\ref{pgf2},(ii), the canonical morphism 
$$ (\tk {C_{i-1}}{\an {X_{i-1}}{x_{i-1}}C_{i-1}}C_{i})_{\tilde{{\goth p}}} 
\longrightarrow  (\an {X_{i-1}}{x_{i-1}}C_{i})_{{\goth p}}$$
is an isomorphism since $\an {X_{i-1}}{x_{i-1}}C_{i-1}$ is a finitely generated 
subalgebra of $K_{i-1}$, containing $C_{i-1}$, which is Cohen-Macaulay and integrally 
closed. Let ${\goth p}^{\#}$ be the inverse image of $\tilde{{\goth p}}$ by the 
isomorphism 
$$ \tk {C'_{i-1}[v_{i}]}{\an {X_{i-1}}{x_{i-1}}[v_{i}]}C_{i} \longrightarrow 
\tk {C_{i-1}}{\an {X_{i-1}}{x_{i-1}}C_{i-1}}C_{i}.$$
Then the canonical morphism 
$$ (\tk {C'_{i}}{\an {X_{i-1}}{x_{i-1}}[v_{i}]}C_{i})_{{\goth p}^{\#}} \longrightarrow 
(\an {X_{i-1}}{x_{i-1}}C_{i})_{{\goth p}}$$
is an isomorphism. From the equalities
$$ (\tk {C'_{i}}{\an {X_{i-1}}{x_{i-1}}[v_{i}]}C_{i})_{{\goth p}^{\#}} =
(\tk {C'_{i}}{\an {X_{i}}{x_{i}}}C_{i})_{\tilde{{\goth m}}}, \qquad
(\an {X_{i-1}}{x_{i-1}}C_{i})_{{\goth p}} = (\an {X_{i}}{x_{i}}C_{i})_{{\goth m}}$$
we deduce that the support of the kernel of $\nu _{i}$ in 
${\mathrm {Spec}}(\tk {C'_{i}}{\an {X_{i}}{x_{i}}}C_{i})$ does not contain 
$\tilde{{\goth m}}$. As a result, denoting by ${\cal S}_{i}$ this support, 
${\cal S}_{i}$ does not contain the inverse images by $\nu _{i}$ of the maximal 
ideals of $\an {X_{i}}{x_{i}}C_{i}$. 

According to Lemma~\ref{l3gf2},(iv), the kernel of the canonical morphism 
$$ \tk {C_{i-1}}{{\goth A}_{i-1}}C_{i} \longrightarrow 
\an {X_{i-1}}{x_{i-1}}C_{i} $$
is the nilradical of $\tk {C_{i-1}}{{\goth A}_{i-1}}C_{i}$. Hence, the kernel of 
the canonical morphism  
$$ \tk {C'_{i}}{\an {X_{i-1}}{x_{i-1}}[v_{i}]}C_{i} \rightarrow 
\an {X_{i-1}}{x_{i-1}}[v_{i}]C_{i}$$
is the nilradical of $\tk {C'_{i}}{\an {X_{i-1}}{x_{i-1}}[v_{i}]}C_{i}$ since
the canonical map
$$ \tk {C'_{i-1}[v_{i}]}{\an {X_{i-1}}{x_{i-1}}[v_{i}]}C_{i} \longrightarrow 
\tk {C_{i-1}}{{\goth A}_{i-1}}C_{i}$$
is an isomorphism by induction hypothesis. As a result, all element of ${\cal S}_{i}$ is 
the inverse image of a prime ideal in $\an {X_{i}}{x_{i}}C_{i}$. Hence ${\cal S}_{i}$ is 
empty, and $\nu _{i}$ is an isomorphism.
\end{proof}

The following Corollary results from Proposition~\ref{pgf3} and Lemma~\ref{l2gf3},(iii)
since $\pi ^{-1}(C_{+})=\pi _{0}^{-1}(C_{+})\times \{0\}$.

\begin{coro} \label{cgf3}
Let $x$ be in $\pi ^{-1}(C_{+})$.

{\rm (i)} The semi-local algebra $\an {{X}}x\es SE$ is normal and Cohen-Macaulay.

{\rm (ii)} The canonical morphism $\tk {C}{\an Xx}\es SE\rightarrow \an Xx\es SE$ is an
isomorphism.
\end{coro}

Let $d$ be the degree of the extension $K$ of $K(C)$. Let $x$ be in $\pi ^{-1}(C_{+})$, 
and denote by $Q_{x}$ the quotient of $\an {{X}}x \es SE$ by ${\goth m}_{x}\es SE$, 
with ${\goth m}_{x}$ the maximal ideal of $\an Xx$.

\begin{lemma}\label{l3gf3}
Let $V$ be a homogenous complement to $\es SEC_{+}$ in ${\rm S}(E)$. 

{\rm (i)} The $\k$-space $V$ has finite dimension, $\es SE=CV$ and $K = K(C)V$.

{\rm (ii)} The $\k$-space $Q_{x}$ has dimension $d$. Furthermore, for all subspace $V'$ 
of dimension $d$ of $V$ such that $Q_{x}$ is the image of $V'$ by the quotient map, the 
canonical map 
$$ \tk {\k}{\an Xx}V' \longrightarrow \an Xx \es SE$$
is bijective.
\end{lemma}

\begin{proof}
(i) According to Lemma~\ref{lgf3}, $\es SE$ is a finite extension of $C$. Hence, the 
$\k$-space $V$ is finite dimensional. On the other hand, we have  
$\es SE = V +  \es SEC_{+}$. Hence, by induction on $m$, 
$\es SE = CV +  \es SE C_{+}^m$ for any $m$, whence 
$\es SE =CV$ since $C_+$ is generated by elements of positive degree. 
As a result, $K=K(C)V$ since the $\k$-space $V$ is finite dimensional.

(ii) Let $d'$ be the dimension of $Q_{x}$. By (i), since $C_{+}$ is contained in 
${\goth m}_{x}$,
$$ \an Xx \es SE = V + {\goth m}_{x}\es SE .$$
As a result, for some subspace $V'$ of dimension $d'$ of $V$, $Q_{x}$ is the image of 
$V'$ by the quotient map. Then,
$$ \an Xx \es SE = \an Xx V' + {\goth m}_{x}\es SE ,$$ 
and by Nakayama's Lemma, $\an Xx \es SE = \an Xx V'$. Let $(\poi v1{,\ldots,}{d'}{}{}{})$ 
be a basis of $V'$. Suppose that the elements $\poi v1{,\ldots,}{d'}{}{}{}$ are not 
linearly independent over $\an Xx$. A contradiction is expected. Let $l$ be the smallest 
integer such that $$ a_{1} v_{1} + \cdots + a_{d'}v_{d'} = 0 $$
for some sequence $(\poi a1{,\ldots,}{d'}{}{}{})$ in ${\goth m}_{x}^{l}$, not contained 
in ${\goth m}_{x}^{l+1}$. According to Corollary~\ref{cgf3},(i) and 
\cite[Ch.\,8, Theorem 23.1]{Mat}, $\an Xx \es SE$ is a flat extension of $\an Xx$ since
$\an Xx \es SE$ is a finite extension of $\an Xx$. So, for some 
$\poi w1{,\ldots,}{m}{}{}{}$ in $\es SE$ and for some sequences 
$(\poi b{i,1}{,\ldots,}{i,m}{}{}{}$, $i=1,\ldots,d')$ in $\an Xx$,
$$ v_{i} = \sum_{j=1}^{m} b_{i,j}w_{j} \quad  \text{and} \quad
\sum_{j=1}^{d'} a_{j}b_{j,k} = 0$$
for all $i=1,\ldots,d'$ and for $k=1,\ldots,m$. Since $\an Xx \es SE = \an Xx V'$, 
$$ w_{j} = \sum_{k=1}^{d'} c_{j,k}v_{k}$$
for some sequence $(c_{j,k}, j=1,\ldots,m, i=1,\ldots,d')$ in $\an Xx$. Setting
$$ u_{i,k} = \sum_{j=1}^{m} b_{i,j}c_{j,k}$$
for $i,k=1,\ldots,d'$, we have
$$ v_{i} = \sum_{k\in I} u_{i,k} v_{k} \quad  \text{and} \quad
\sum_{j\in I} a_{j}u_{j,i} =0$$
for all $i=1,\ldots,d'$. Since $\poi v1{,\ldots,}{d'}{}{}{}$ are linearly independent 
modulo ${\goth m}_{x}\es SE$, 
$$u_{i,k} - \delta _{i,k} \in {\goth m}_{x}$$
for all $(i,k)$, with $\delta _{i,k}$ the Kronecker symbol. As a result, $a_{i}$ is
in ${\goth m}_{x}^{l+1}$ for all $i$, whence a contradiction. Then the canonical map
$$ \tk {\k}{\an Xx}V' \longrightarrow \an Xx \es SE$$
is bijective. Since $K=K(C)\es SE$ 
and since $K(C)$ is the fraction field of $\an Xx$, 
$\poi v1{,\ldots,}{d'}{}{}{}$ is a basis of $K$ over $K(C)$. Hence, $d'=d$ and the 
assertion follows.
\end{proof}

Recall that $K_{0}$ is the fraction field of $\overline{A}$. 
Let $\poi v{N+1}{,\ldots,}{N+r}{}{}{}$ be elements of $E$ such that
$\poi v1{,\ldots,}{N+r}{}{}{}$ is a basis of $E$. 
Denoting by 
$\poi t1{,\ldots,}{r}{}{}{}$ some indeterminates, let $\vartheta $ be the morphism of 
$C$-algebras
$$ C[\poi t1{,\ldots,}{r}{}{}{}] \longrightarrow \es SE, \qquad t_{i} \longmapsto 
v_{N+i},$$ 
and let $\tilde{\vartheta }$ be the morphism of 
$K_{0}[\poi v1{,\ldots,}{N}{}{}{}]$-algebras
$$ K_{0}[\poi v1{,\ldots,}{N}{}{}{},\poi t1{,\ldots,}{r}{}{}{}] \longrightarrow 
\tk {\overline{A}}{K_{0}}\es SE, \qquad t_{i} \longmapsto v_{N+i} .$$
For ${\bf i}=(\poi i1{,\ldots,}{N}{}{}{})$ in ${\Bbb N}^{N}$ and for 
${\bf j}=(\poi j1{,\ldots,}{r}{}{}{})$ in ${\Bbb N}^{r}$, set:
$$ v^{{\bf i}} := \poie v1{\cdots }{N}{}{}{}{i_{1}}{i_{N}}, \qquad
t^{{\bf j}} := \poie t1{\cdots }{r}{}{}{}{j_{1}}{j_{r}} .$$
For $a$ in $\overline{A}$, denote by $\overline{a}$ the polynomial in 
$\k[\poi v1{,\ldots,}{N}{}{}{},\poi t1{,\ldots,}{r}{}{}{}]$ such that 
$\vartheta (\overline{a})=a$. 

\begin{lemma}\label{l4gf3}
Let $I$ be the ideal of $C[\poi t1{,\ldots,}{r}{}{}{}]$ generated by the 
elements $a-\overline{a}$ with $a$ in $\overline{A}$. 

{\rm (i)} For all homogenous generating family $(\poi a1{,\ldots,}{m}{}{}{})$ of 
$\overline{A}_{+}$, $I$ is the ideal generated by the sequence 
$(a_{i}-\overline{a_{i}}\,, \,i=1,\ldots,m)$.

{\rm (ii)} The ideal $I$ is the kernel of $\vartheta $. 
\end{lemma}

\begin{proof}
(i) Let $I'$ be the ideal of $C[\poi t1{,\ldots,}{r}{}{}{}]$ generated by the sequence
$(a_{i}-\overline{a_{i}},i=1,\ldots,m)$. Since the map $a\mapsto \overline{a}$ is linear, 
it suffices to prove that $a-\overline{a}$ is in $I'$ for all homogenous element $a$
of $\overline{A}_{+}$. Prove it by induction on the degree of $a$. For some homogenous
sequence $(\poi b1{,\ldots,}{m}{}{}{})$ in $\overline{A}$, 
$$ a = b_{1}a_{1} + \cdots + b_{m}a_{m}$$
so that
$$ a - \overline{a} = \sum_{i=1}^{m} b_{i}(a_{i}-\overline{a_{i}}) + 
\sum_{i=1}^{m} \overline{a_{i}}(b_{i}-\overline{b_{i}}) .$$
If $a$ has minimal degree, $\poi b1{,\ldots,}{m}{}{}{}$ are in $\k$ and 
$b_{i}=\overline{b_{i}}$ for $i=1,\ldots,m$. Otherwise, for $i=1,\ldots,m$, if $b_{i}$ 
is not in $\k$, $b_{i}$ has degree smaller than $a$, whence the assertion by induction 
hypothesis. 

(ii) By definition, $I$ is contained in the kernel of $\vartheta $. Let $a$ be in
$C[\poi t1{,\ldots,}{r}{}{}{}]$. Then $a$ has an expansion
$$ a = \sum_{({\bf i},{\bf j}) \in {\Bbb N}^{N}\times {\Bbb N}^{r}}
a_{{\bf i},{\bf j}} v^{{\bf i}}t^{{\bf j}} $$
with the $a_{{\bf i},{\bf j}}$'s in $\overline{A}$, whence
$$ a = \sum_{({\bf i},{\bf j}) \in {\Bbb N}^{N}\times {\Bbb N}^{r}}
(a_{{\bf i},{\bf j}} - \overline{a_{{\bf i},{\bf j}})}v^{{\bf i}}t^{{\bf j}} 
+ \sum_{({\bf i},{\bf j})\in {\Bbb N}^{N}\times {\Bbb N}^{r}}
\overline{a_{{\bf i},{\bf j}}} v^{{\bf i}} t^{{\bf j}} .$$
If $\vartheta (a)=0$, then 
$$ \sum_{({\bf i},{\bf j})\in {\Bbb N}^{N}\times {\Bbb N}^{r}}
\overline{a_{{\bf i},{\bf j}}} v^{{\bf i}} t^{{\bf j}} = 0$$
since the restriction of $\vartheta $ to 
$\k[\poi v1{,\ldots,}{N}{}{}{},\poi t1{,\ldots,}{r}{}{}{}]$ is injective, whence the 
assertion.
\end{proof}

For $x$ in ${{\pi}} ^{-1}(C_{+})$, denote by $\vartheta _{x}$ the morphism 
$$ \an Xx[\poi t1{,\ldots,}{r}{}{}{}] \longrightarrow K, \qquad 
at^{{\bf j}} \longmapsto av_{N+1}^{j_{1}}\cdots v_{N+r}^{j_{r}} .$$

\begin{prop}\label{p2gf3}
Let $x$ be in $\pi ^{-1}(C_{+})$.

{\rm (i)} The kernel of $\vartheta _{x}$ is the ideal of 
$\an Xx[\poi t1{,\ldots,}{r}{}{}{}]$ generated by $I$. Furthermore, the image of 
$\vartheta _{x}$ is the subalgebra $\an Xx\es SE$ of $K$.

{\rm (ii)} The intersection of ${\goth m}_{x}\es SE$ and $\es SE$ is equal to  
$C_{+}\es SE$.
\end{prop}

\begin{proof}
(i) From the short exact sequence
$$ 0 \longrightarrow I \longrightarrow C[\poi t1{,\ldots,}{r}{}{}{}] \longrightarrow 
\es SE \longrightarrow 0$$
we deduce the exact sequence
$$ \tk C{\an Xx}I \longrightarrow \tk C{\an Xx}C[\poi t1{,\ldots,}{r}{}{}{}] 
\longrightarrow \tk C{\an Xx}\es SE \longrightarrow 0 .$$
Moreover, we have a commutative diagram
$$\xymatrix{  && &&  0 \ar[d] & \\
\tk C{\an Xx}I \ar[rr]^{\hspace{-1em}\dd} \ar[d]^{\delta } && 
\tk C{\an Xx}C[\poi t1{,\ldots,}{r}{}{}{}] \ar[d]^{\delta } \ar[rr]^{\dd} &&
\tk C{\an Xx}\es SE \ar[d]^{\delta } \ar[r] & 0 \\
\an XxI \ar[d] \ar[rr]^{\hspace{-1em}\dd} && \an Xx[\poi t1{,\ldots,}{r}{}{}{}] 
\ar[d] \ar[rr]^{\dd} &&
\an Xx \es SE \ar[d] \ar[r] & 0 \\ 0 && 0 && 0 & }$$
with exact columns by Corollary~\ref{cgf3},(ii). For $a$ in 
$\an Xx[\poi t1{,\ldots,}{}{}{}{}]$ such that $\dd a = 0$,
$$ a = \delta b, \quad b = \dd c \quad  \text{with} \quad 
b \in \tk C{\an Xx}C[\poi t1{,\ldots,}{r}{}{}{}], \quad 
c \in \tk C{\an Xx}I ,$$
so that $a = \dd \rond \delta c$. Hence $\an Xx I$ is the kernel of $\vartheta _{x}$.

(ii) Let $\poi a1{,\ldots,}{m}{}{}{}$ be a homogenous generating family of 
$\overline{A}_{+}$. For $i=1,\ldots,m$,
$$ \overline{a_{i}} = \sum_{({\bf j},{\bf k})\in {\Bbb N}^{N}\times {\Bbb N}^{r}}
a_{i,{\bf j},{\bf k}} v^{{\bf j}} t^{{\bf k}} ,$$
with the $a_{i,{\bf j},{\bf k}}$'s in $\k$. Set:
$$ a'_{i} := \sum_{{\bf k}\in {\Bbb N}^{N}} a_{i,0,{\bf k}} t^{{\bf k}} .$$
For $i=1,\ldots,m$,
$$ a'_{i} \in \overline{a_{i}} - a_{i} + C_{+}[\poi t1{,\ldots,}{r}{}{}{}]$$ 
since $a_{i}$ is in $\overline{A}_{+}$ so that $\vartheta (a'_{i})$ is in 
$C_{+}\es SE$.

Since $C_{+}$ is contained in ${\goth m}_{x}$, $C_{+}\es SE$ is contained in 
${\goth m}_{x}\es SE \cap \es SE$. Let $a$ be in 
${\goth m}_{x}[\poi t1{,\ldots,}{r}{}{}{}]$ such that $\vartheta _{x}(a)$ is 
in $\es SE$. According to (i),
$$ a \in C[\poi t1{,\ldots,}{r}{}{}{}] + \an Xx I .$$
So, by Lemma~\ref{l3gf3},(i), 
$$ a = b + b_{1}(a_{1}-\overline{a_{1}}) + \cdots + b_{m} (a_{m}-\overline{a_{m}}) ,$$
with $b$ in $C[\poi t1{,\ldots,}{r}{}{}{}]$ and $\poi b1{,\ldots,}{m}{}{}{}$ in 
$\an Xx$. Then,
$$ b = b_{0} + b_{+}, \quad  \text{with} \quad b_{0} \in \k[t_{1},\ldots,t_{r}] 
\quad  \text{and} \quad b_{+} \in C_{+}[\poi t1{,\ldots,}{r}{}{}{}]$$
$$ b_{i} = b_{i,0} + b_{i,+}, \quad  \text{with} \quad b_{i,0} \in \k 
\quad  \text{and} \quad b_{i,+} \in {\goth m}_{x}$$
for $i=1,\ldots,m$. Since $a$ is in ${\goth m}_{x}[\poi t1{,\ldots,}{r}{}{}{}]$ and
$\poi a1{,\ldots,}{m}{}{}{}$ are in $C_{+}$,
$$ b_{0} - b_{1,0}\overline{a_{1}} - \cdots - b_{m,0}\overline{a_{m}}
\in {\goth m}_{x}[\poi t1{,\ldots,}{r}{}{}{}] $$
Moreover, for $i=1,\ldots,m$,
$$\overline{a_{i}} - a'_{i} \in C_{+}[\poi t1{,\ldots,}{r}{}{}{}] .$$
Hence 
$$ b_{0} - b_{1,0}a'_{1} - \cdots - b_{m,0}a'_{m} = 0 \quad  \text{since} \quad
{\goth m}_{x}[\poi t1{,\ldots,}{r}{}{}{}] \cap \k[\poi t1{,\ldots,}{r}{}{}{}] = 0 .$$
As a result, $\vartheta _{x}(a)$ is in $C_{+}\es SE$ since 
$\vartheta _{x}(a)= \vartheta _{x}(b_{0}) + \vartheta _{x}(b_{+})$.
\end{proof}

\subsection{}\label{gf4}
We are now in a position to prove the main result of the section. 
Recall the main notations: 
$E$ is a finite dimensional vector space over $\k$, 
$A$ is a homogenous subalgebra of S$(E)$, different from S$(E)$ and 
such that $A=\k+A_{+}$, ${\cal N}_{0}$ is the nullvariety of $A_{+}$ in 
$E^{*}$, $K$ is the fraction field of S$(E)$ and $K(\overline{A})$ 
that one of $\overline{A}$, the algebraic closure of $A$ in S$(E)$. 

\begin{theorem}\label{tgf4}
Suppose that the following conditions are satisfied:
\begin{itemize}
\item [{\rm (a)}] ${\cal N}_{0}$ has dimension $N$, 
\item [{\rm (b)}] $A$ is a polynomial algebra,
\item [{\rm (c)}] $K(\overline{A})$ is algebraically closed in $K$.
\end{itemize}
Then $\overline{A}$ is a polynomial algebra. Moreover, $\es SE$ is a free extension of 
$\overline{A}$.
\end{theorem}

\begin{proof}
Use the notations of Subsection \ref{gf3}. In particular, set 
$$C=\overline{A}[\poi v1{,\ldots,}{N}{}{}{}],$$
with $(\poi v1{,\ldots,}{N}{}{}{})$ a sequence of elements of $E$ such that 
its nullvariety in ${\cal N}_{0}$ is equal to $\{0\}$ (cf.~Lemma~\ref{lgf1},(iii)), 
and let $K(C)$ be the fraction field of $C$. As already explained, according to 
Proposition~\ref{pgf1},(ii), it suffices to prove that $\es SE$ is a free extension of 
$C$. Let $V$ be as in Lemma~\ref{l4gf3}, a homogenous complement to $\es SEC_+$ in 
$\es SE$. Recall that $X$ is a desingularization of $Z={\rm Specm}(C)$ and that $\pi$ is 
the morphism of desingularization. Let $x$ be in ${\pi} ^{-1}(C_{+})$. According to 
Proposition~\ref{p2gf3},(ii), for some subspace $V'$ of $V$, $V'$ is a complement 
to ${\goth m}_{x}\es SE$ in $\an Xx \es SE$. Then, by Lemma~\ref{l3gf3},(ii), $V'$
has dimension the degree of the extension $K$ of $K(C)$ and the canonical map
$$ \tk {\k}{\an Xx}V' \longrightarrow \an Xx \es SE $$
is bijective. 
Moreover, 
$$V' \oplus \es SE C_{+} = \es SE \quad  \text{and} \quad V' = V .$$
Indeed, for $a\in \es SE$, write $a = b+c$ with $b\in V'$ and 
$c \in {\goth m}_{x}\es SE$. Since $V'$ is contained in $\es SE$, $c$ is in 
$\es SE$, whence $c$ in $\es SE C_{+}$ by Proposition~\ref{p2gf3},(ii). 
In addition, $\es SE = CV$ as it has been observed in the proof of Lemma \ref{l3gf3},(i). 
As a result, the canonical map 
$$ \tk {\k}{C}V \longrightarrow \es SE $$ 
is bijective. This concludes the proof of the theorem.
\end{proof}

\section{Good elements and good orbits} \label{ge} 
Recall that $\k$ is an algebraically closed field of characteristic zero. 
As in the introduction, $\g$ is a simple Lie algebra over $\k$ of rank $\rg$, 
$\dv ..$ denotes the Killing form of $\g$, and $G$ denotes the adjoint group 
of $\g$. 

\subsection{}\label{ge1}
The notions of good element and good orbit in ${\goth g}$ are introduced in this 
paragraph. 

For $x$ in ${\goth g}$, denote by $\g^{x}$ its centralizer in $\g$, by $G^{x}$ 
its stabilizer in $G$, by $G_0^{x}$ the identity component of $G^{x}$ and 
by $K_{x}$ the fraction field of the symmetric algebra $\es S{{\goth g}^{x}}$. 
Then $\ai gx{}$ and $K_x^{\g^{x}}$ denote the sets of $G_0^{x}$-invariant elements of 
$\es S{{\goth g}^{x}}$ and $K_x$ respectively. 

\begin{lemma}\label{lge1}
Let $x$ be in ${\goth g}$. Then $K_{x}^{{\goth g}^{x}}$ is the fraction field of 
$\ai gx{}$ and $K_{x}^{{\goth g}^{x}}$ is algebraically closed in $K_{x}$ of 
transcendental degree $\rg$ over $\k$.
\end{lemma}

\begin{proof}
Let $a$ be in $K_{x}$, algebraic over $K_{x}^{{\goth g}^{x}}$. For all $g$ in 
$G^{x}_{0}$, $g.a$ satisfies the same equation of algebraic dependence over 
$K_{x}^{{\goth g}^{x}}$ as $a$. Since a polynomial in one indeterminate has a finite 
number of roots, the $G_0^{x}$-orbit of $a$ is finite. But this orbit is then reduced to 
$\{a\}$, $G_0^{x}$ being connected. Hence $a$ is in $K_{x}^{{\goth g}^{x}}$. This shows 
that $K_{x}^{{\goth g}^{x}}$ is algebraically closed in $K_{x}$. According to 
\cite[Theorem 1.2]{CM} (see also Theorem \ref{ti1}), the index of 
${\goth g}^{x}$ is equal to $\rg$. So, by \cite{Ro}, the transcendental degree of 
$K_{x}^{{\goth g}^{x}}$ over $\k$ is equal to $\rg$. It remains to prove that 
$K_{x}^{{\goth g}^{x}}$ is the fraction field of $\ai gx{}$.

Since ${\goth g}^{x}$ is the centralizer of $x_{\n}$ in the reductive Lie algebra
${\goth g}^{x_{\s}}$, we can suppose $x$ nilpotent. Any rational invariant is a quotient
of two semi-invariant polynomials, because of the prime factor decomposition. Each
semi-invariant has a central character $\lambda$, a character of the center of a Levi
subalgebra in ${\goth g}^{x}$.  By~\cite[Lemma 4.6,(i)]{JS}, there is also a
semi-invariant with the character $-\lambda$.
 Multiplying both numerator and denominator by this invariant, we get
 the same invariant as a quotient of invariants, whence the lemma.
\end{proof}

\begin{defi}\label{dge1}
An element $x\in{\goth g}$ is called a {\em good element of ${\goth g}$} if for some 
homogenous elements $\poi p1{,\ldots,}{\rg}{}{}{}$ of $\ai g{x}{}$, the nullvariety of 
$\poi p1{,\ldots,}{\rg}{}{}{}$ in $({\goth g}^{x})^{*}$ has codimension $\rg$ in
$({\goth g}^{x})^{*}$. A $G$-orbit in ${\goth g}$ is called {\em good} if it is the orbit
of a good element.
\end{defi}

Since the nullvariety of $\ai g{}{}_{+}$ in ${\goth g}$ is the nilpotent cone of 
${\goth g}$, $0$ is a good element of ${\goth g}$. For $(g,x)$ in $G\times {\goth g}$ and
for $a$ in $\ai g{x}{}$, $g(a)$ is in $\ai g{g(x)}{}$. So, an orbit is good if and only 
if all its elements are good. 

\begin{theorem}\label{tge1}
Let $x$ be a good element of ${\goth g}$. Then $\ai gx{}$ is a polynomial algebra and 
$\es S{{\goth g}^{x}}$ is a free extension of $\ai gx{}$.  
\end{theorem}

\begin{proof}
Let $\poi p1{,\ldots,}{\rg}{}{}{}$ be homogenous elements of $\ai gx{}$ such that the 
nullvariety of $\poi p1{,\ldots,}{\rg}{}{}{}$ in $({\goth g}^{x})^{*}$ has codimension 
$\rg$. Denote by $A$ the subalgebra of $\ai gx{}$ generated by 
$\poi p1{,\ldots,}{\rg}{}{}{}$. Then $A$ is a homogenous subalgebra of ${\rm S}(\g^{x})$ 
and the nullvariety of $A_{+}$ in $({\goth g}^{x})^{*}$ has codimension $\rg$. So, by 
Lemma~\ref{lgf1},(ii), $A$ has dimension $\rg$. Hence, $\poi p1{,\ldots,}{\rg}{}{}{}$ are 
algebraically independent and $A$ is a polynomial algebra. Denote by $\overline{A}$ the
algebraic closure of $A$ in $\es S{{\goth g}^{x}}$. By Lemma~\ref{lge1}, $\overline{A}$
is contained in $\ai gx{}$ and the fraction field of $\ai gx{}$ is algebraically closed
in $K_{x}$. As a matter of fact, $\overline{A}=\ai gx{}$ since the fraction fields of $A$
and $\ai gx{}$ have the same transcendental degree. Hence, by Theorem~\ref{tgf4}, 
$\ai gx{}$ is a polynomial algebra and $\es S{{\goth g}^{x}}$ is free extension of 
$\ai gx{}$. 
\end{proof}

\begin{rema} \label{rge1}
The algebra $\ai gx{}$ may be polynomial even though $x$ is not good. Indeed, let us 
consider a nilpotent element $e$ of $\g=\mathfrak{so}(\k^{10})$ in the nilpotent orbit 
associated with the partition $(3,3,2,2)$. Then the algebra $\ai ge{}$ is polynomial, 
generated by elements of degrees $1,1,2,2,5$. But the nullcone 
has an irreducible component of codimension at most $4$. So, $e$ is not good. 
We refer the reader to Example \ref{erc2} for more details. 
\end{rema}

For $x\in\g$, denote by $x_{\s}$ and $x_{\n}$ the semisimple and the nilpotent 
components of $x$ respectively. 

\begin{prop}\label{pge1}
Let $x$ be in ${\goth g}$. Then $x$ is good if and only if $x_{\n}$ is a good 
element of the derived algebra of ${\goth g}^{x_{\s}}$.
\end{prop}

\begin{proof}
Let ${\goth z}$ be the center of ${\goth g}^{x_{\s}}$ and let ${\goth a}$ be the 
derived algebra of ${\goth g}^{x_{\s}}$. Then 
$$\begin{array}{ccc}
{\goth g}^{x} = {\goth z}\oplus {\goth a}^{x_{\n}}, &&
\ai gx{} = \tk {\k}{\e Sz}\ai a{x_{\n}}{} .\end{array}$$
By the first equality, $({\goth a}^{x_{\n}})^{*}$ identifies with the 
orthogonal complement to ${\goth z}$ in $({\goth g}^{x})^{*}$. Set 
$d:=\dim {\goth z}$. Suppose that $x_{\n}$ is a good element of 
${\goth a}$ and let $\poi p1{,\ldots,}{\rg -d}{}{}{}$ be homogenous elements of 
$\ai a{x_{\n}}{}$ whose nullvariety in $({\goth a}^{x_{\n}})^{*}$ has codimension 
$\rg -d$. Denoting by $\poi v1{,\ldots,}{d}{}{}{}$ a basis of ${\goth z}$, 
the nullvariety of $\poi v1{,\ldots,}{d}{}{}{},\poi p1{,\ldots,}{\rg -d}{}{}{}$ in 
$({\goth g}^{x})^{*}$ is the nullvariety of $\poi p1{,\ldots,}{\rg -d}{}{}{}$ in 
$({\goth a}^{x_{\n}})^{*}$. Hence, $x$ is a good element of ${\goth g}$.

Conversely, let us suppose that $x$ is a good element of ${\goth g}$. By 
Theorem \ref{tge1}, $\ai gx{}{}$ is a polynomial algebra generated by 
homogenous polynomials $\poi p1{,\ldots,}{\rg}{}{}{}$. Since ${\goth z}$ is contained
in $\ai gx{}$, $\poi p1{,\ldots,}{\rg}{}{}{}$ can be chosen so that 
$\poi p1{,\ldots,}{d}{}{}{}$ are in ${\goth z}$ and $\poi p{d+1}{,\ldots,}{\rg}{}{}{}$ 
are in $\ai a{x_{\n}}{}$. Then the nullvariety of $\poi p{d+1}{,\ldots,}{\rg}{}{}{}$ in 
$({\goth a}^{x_{\n}})^{*}$ has codimension $\rg -d$. Hence, $x_{\n}$ is a good element of 
${\goth a}$.
\end{proof}

\subsection{} \label{ge2}
In view of Theorem \ref{tge1}, we wish to find a sufficient condition for that an element 
$x\in\g$ is good. According to Proposition \ref{pge1}, it is enough to consider the case 
where $x$ is nilpotent. 

\smallskip

Let $e$ be a nilpotent element of ${\goth g}$, embedded into an ${\goth {sl}}_{2}$-triple 
$(e,h,f)$ of $\g$. Identify the dual of $\g$ with $\g$, and the dual of ${\goth g}^{e}$ 
with ${\goth g}^{f}$ through the Killing form $\dv ..$ of $\g$. For $p$ in 
$\e Sg\simeq \k[\g]$, denote by $\kappa (p)$ the restriction to $\g^{f}$ of the 
polynomial function $x\mapsto p(e+x)$ and denote by $\ie{p}$ its initial homogenous 
component. According to~\cite[Proposition 0.1]{PPY}, for $p$ in $\e Sg^{{\goth g}}$, 
$\ie{p}$ is in $\ai ge{}$.

\smallskip

The proof of the following theorem will be achieved in Subsection \ref{sg4}.  

\begin{theorem}\label{tge2}
Suppose that for some homogenous generators $\poi q1{,\ldots,}{\rg}{}{}{}$ of 
$\ai g{}{}$, the polynomial functions $\poi {\ie q}1{,\ldots,}{\rg}{}{}{}$ are 
algebraically independent. Then $e$ is a good element of ${\goth g}$. In particular, 
$\ai ge{}$ is a polynomial algebra and $\es S{{\goth g}^{e}}$ is a free extension of 
$\ai ge{}$. Moreover, $(\poi {\ie q}1{,\ldots,}{\rg}{}{}{})$ is a regular sequence in 
$\es S{{\goth g}^{e}}$.
\end{theorem}

The overall idea of the proof is the following. 

\smallskip
 
According to Theorem~\ref{tge1}, it suffices to prove that $e$ is good, and 
more accurately that the nullvariety of $\poi {\ie q}1{,\ldots,}{\rg}{}{}{}$ in 
${\goth g}^{f}$ has codimension $\rg$ since $\poi {\ie q}1{,\ldots,}{\rg}{}{}{}$ 
are invariant homogenous polynomials. As explained in the introduction, we will use the 
Slodowy grading on $\es S{\g^{e}}[[t]]$ and $\es S{\g^{e}}((t))$, induced from that on 
$\es S{\g^{e}}$, to deal with this problem. This is the main purpose of Section \ref{sg}. 

\section{Slodowy grading and proof of Theorem \ref{ti4}} \label{sg}
This section is devoted to the proof of Theorem \ref{tge2} (or Theorem \ref{ti4}). 
The proof is achieved in Subsection~\ref{sg5}.  
As in the previous section, ${\goth g}$ is a simple Lie algebra over $\k$ 
and $(e,h,f)$ is an ${\goth {sl}}_{2}$-triple of ${\goth g}$. Let us simply denote by $S$
the algebra $\es S{{\goth g}^{e}}$.

Let $\poi q1{,\ldots,}{\rg}{}{}{}$ be homogenous generators of $\ai g{}{}$ of degrees
$\poi d1{,\ldots,}{\rg}{}{}{}$ respectively. The sequence 
$(\poi q1{,\ldots,}{\rg}{}{}{})$ is ordered so that 
$\poi d1{\leq \cdots \leq }{\rg}{}{}{}$. We assume in the whole 
section that the polynomial functions $\poi {\ie q}1{,\ldots,}{\rg}{}{}{}$ are 
algebraically independent. 
The aim is to show that $e$ is good (cf.\ Definition \ref{dge1}). 

\subsection{} \label{sg1}
Let $\poi x1{,\ldots,}{r}{}{}{}$ be a basis of ${\goth g}^{e}$ such that for 
$i=1,\ldots,r$, $[h,x_{i}]=n_{i}x_{i}$ for some nonnegative integer $n_{i}$. For 
${\bf j}=(\poi j1{,\ldots,}{r}{}{}{})$ in ${\Bbb N}^{r}$, set:
$$ \vert {\bf j} \vert := \poi j1{+\cdots +}{r}{}{}{}, \qquad 
\vert {\bf j} \vert_{e} := j_{1}n_{1}+\cdots +j_{r}n_{r}+2\vert {\bf j} \vert, \qquad
x^{{\bf j}} = \poie x1{\cdots }{r}{}{}{}{j_{1}}{j_{r}}.$$
The algebra $S$ has two gradings: the standard one and the 
{\em Slodowy grading}. For all ${\bf j}$ in ${\Bbb N}^{r}$, $x^{{\bf j}}$ is 
homogenous with respect to these two gradings. It has standard degree 
$\vert {\bf j} \vert$ and Slodowy degree $\vert {\bf j} \vert_{e}$. In this 
section, we only consider the Slodowy grading. So, by grading we will always 
mean Slodowy grading. For $m$ nonnegative integer, denote by $S^{[m]}$ 
the subspace of $S$ of degree $m$.

Let $t$ be an indeterminate. For all subspace $V$ of $S$, set:
$$ V[t] := \tk {\k}{\k[t]}V, \qquad V[t,t^{-1}] := \tk {\k}{\k[t,t^{-1}]}V, \qquad
V[[t]] := \tk {\k}{\k[[t]]}V, \qquad V((t)) := \tk {\k}{\k((t))}V ,$$
with $\k((t))$ the fraction field of $\k[[t]]$. For $V$ a subspace of $S[[t]]$, denote by
$V(0)$ the image of $V$ by the quotient morphism 
$$S[t]\longrightarrow S, \qquad a(t) \longmapsto a(0) .$$
  
The grading of $S$ induces a grading of the algebra $S((t))$ with $t$ having 
degree $0$. For $V$ a homogenous subspace of $S((t))$ and for $m$ a nonnegative integer, 
let $V^{[m]}$ be its component of degree $m$. In particular, for $V$ a homogenous 
subspace of $S$, $V((t))$ is a homogenous subspace of $S((t))$ and 
$$V((t))^{[m]}=V^{[m]}((t)).$$
Let $\tau $ be the morphism of algebras,  
$$\tau \colon  S \longrightarrow S[t], \quad x_{i} \mapsto tx_{i} \quad\text{for}\quad 
i=1,\ldots,r .$$
The morphism $\tau$ is a morphism of homogenous algebras. 
Denote by $\poi {\delta }1{,\ldots,}{\rg}{}{}{}$ the standard degrees of 
$\poi {\ie q}1{,\ldots,}{\rg}{}{}{}$ respectively, and set for $i=1,\ldots,\rg$ 
$$Q_{i} := t^{-\delta _{i}} \tau (\kappa (q_{i})).$$
Let $A$ be the subalgebra of $S[t]$ generated by 
$\poi Q1{,\ldots,}{\rg}{}{}{}$. Then observe that $A(0)$ is the subalgebra of $S$ 
generated by $\poi {\ie q}1{,\ldots,}{\rg}{}{}{}$. 
For ${\bf j}=(\poi j1{,\ldots,}{\rg}{}{}{})$ in ${\Bbb N}^{\rg}$, set 
$$ q^{{\bf j}} := \poie q1{\cdots }{\rg}{}{}{}{j_{1}}{j_{\rg}}, \qquad
\kappa (q)^{{\bf j}} := \poie q1{\cdots }{\rg}{\kappa }{}{}{j_{1}}{j_{\rg}}, 
\qquad \ie q^{{\bf j}} := \poie {\ie q}1{\cdots }{\rg}{}{}{}{j_{1}}{j_{\rg}}, 
\qquad Q^{{\bf j}} := \poie Q1{\cdots }{\rg}{}{}{}{j_{1}}{j_{\rg}}.$$

\begin{prop}\label{psg1}
{\rm (i)} For ${\bf j}$ in ${\Bbb N}^{\rg}$, $\kappa (q)^{{\bf j}}$ and $\ie q^{{\bf j}}$
are homogenous of degree $2d_{1}j_{1}+\cdots +2d_{\rg}j_{\rg}$.

{\rm (ii)} The map $Q\mapsto Q(0)$ is an isomorphism of homogenous algebras from $A$ 
onto $A(0)$.
\end{prop}

\begin{proof} 
(i) follows from \cite[Section 5]{Pr} or \cite[Section 2]{PPY}. 

(ii) The set $(Q^{{\bf j}},{\bf j}\in {\Bbb N}^{\rg})$ is a basis of the $\k$-space $A$ 
and the image of $Q^{{\bf j}}$ by the map $Q\mapsto Q(0)$ is equal to $\ie q^{{\bf j}}$. 
Moreover, by (i), $Q^{{\bf j}}$ and $\ie q^{{\bf j}}$ are homogenous of degree 
$2d_{1}j_{1}+\cdots +2d_{\rg}j_{\rg} $ so that $Q\mapsto Q(0)$ is a morphism of graded 
algebras. By definition, its image is $A(0)$. Since $\poi {\ie q}1{,\ldots,}{\rg}{}{}{}$ 
are algebraically independent, it is injective.
\end{proof}

By Proposition~\ref{psg1},(ii), $A$ and $A(0)$ are isomorphic homogenous subalgberas of 
$S[t]$ and $S$ respectively. In particular, $A$ is a polynomial algebra since   
$A(0)$ is polynomial by our hypothesis. 

Denote by $A_+$ and $A(0)_{+}$ the ideals of $A$ and $A(0)$ generated by the homogenous
elements of positive degree respectively, and denote by $\tilde{A}$ the subalgebra of 
$S[[t]]$ generated by $\k[[t]]$ and $A$, i.e.,  
$$\tilde{A}:=\k[[t]] A.$$

\begin{lemma}\label{lsg1}
{\rm (i)} The algebra $\tilde{A}$ is isomorphic to $\tk {\k}{\k[[t]]}A$. In particular,
it is regular. 

{\rm (ii)} The element $t$ of $\tilde{A}$ is prime.

{\rm (iii)} Each prime element of $A$ is a prime element of $\tilde{A}$.
\end{lemma}

\begin{proof}
(i) Let $a_{m},m\in {\Bbb N},$ be in $A$ such that
$$ \sum_{m\in {\Bbb N}} t^{m}a_{m} = 0 .$$
If $a_{m}\neq 0$ for some $m$, then $a_{p}(0)=0$ if $p$ is the smallest one such that
$a_{p}\neq 0$. By Proposition~\ref{psg1},(ii), it is not possible. Hence, the 
canonical map
$$ \tk {\k}{\k[[t]]}A \longrightarrow \tilde{A}$$
is an isomorphism. As observed just above, $A$ is a polynomial 
algebra. Then $\tilde{A}$ is a regular algebra by~\cite[Ch.\ 7, Theorem 19.5]{Mat}. 

(ii) By (i), $A$ is the quotient of $\tilde{A}$ by $t\tilde{A}$ so that $t$ is a prime 
element of $\tilde{A}$.

(iii) By (i), for $a$ in $A$, the quotient $\tilde{A}/\tilde{A}a$ is isomorphic to 
$\tk {\k}{\k[[t]]}A/Aa$. Hence $a$ is a prime element of $\tilde{A}$ if it is a prime 
element of $A$.
\end{proof}

As it has been explained in Subsection \ref{ge2}, in order to prove Theorem \ref{tge2}, 
we aim to prove that $S$ is a free extension that $A(0)$. As a first step, we describe in
Subsections \ref{sg2}, \ref{sg3} and \ref{sg4} some properties of the algebra $A$. We 
show in Subsection \ref{sg3} that $S((t))$ is a free extension of $A$ 
(cf.\ Proposition \ref{psg3},(iii)), and we show in Subsection \ref{sg4} 
that $S[[t]]$ is a free extension of $A$ (cf.\ Corollary \ref{c2sg4}). 
In Subsection~\ref{sg5}, we consider the algebra $\tilde{A}$ and prove that $S[[t]]$ 
is a free extension of $\tilde{A}$ (cf.\ Theorem \ref{tsg5},(i)). The expected result 
will follow from this (cf.\ Theorem \ref{tsg5},(iii)). 

\subsection{} \label{sg2} 
Let $\theta _{e}$ be the map
$$ G\times (e+{\goth g}^{f}) \longrightarrow {\goth g}, \qquad  (g,x) \mapsto g(x),$$
and let ${\cal J}_{e}$ be the ideal of $\es S{{\goth g}^{e}}$ generated by the elements 
$\kappa(q_1),\ldots,\kappa(q_\rg)$. The following lemma is known by 
\cite[Theorem 5.4]{Pr} and the proof of \cite[Theorem 2.1]{PPY}.

\begin{lemma}\label{lsg2}
{\rm (i)} The map $\theta_e$ is a smooth morphism onto a dense open subset of ${\goth g}$,
containing $G.e$.

{\rm (ii)} The nullvariety of ${\cal J}_{e}$ in ${\goth g}^{f}$ is equidimensional of 
dimension $r-\rg$.

{\rm (iii)} The ideal ${\cal J}_{e}$ of $\es S{{\goth g}^{e}}$ is radical.
\end{lemma}

Denote by $\cal{V}$ the nullvariety of $A_+$ in $\g^{f}\times \k$, and by $\cal{V}_0$ 
the nullvariety of $A(0)_+$ in $\g^{f}$. 
Then denote by ${\cal V}_{*}$ the union of the irreducible components of ${\cal V}$ 
which are not contained in ${\goth g}^{f}\times \{0\}$. 
Note that ${\cal V}_0\times\{0\}$ is the nullvariety of $t$ in ${\cal V}$, and that 
$${\cal V} = {\cal V}_{*} \cup {\cal V}_0 \times\{0\}.$$

\begin{coro}\label{csg2}
{\rm (i)} The variety ${\cal V}_{*}$ is equidimensional of dimension $r+1-\rg$. Moreover,
for $X$ an irreducible component of ${\cal V}_{*}$ and for $z$ in $\k$, the 
nullvariety of $t-z$ in $X$ has dimension $r-\rg$.

{\rm (ii)} The algebra $S[t,t^{-1}]$ is a free extension of $A$.

{\rm (iii)} The ideal $S[t,t^{-1}]A_{+}$ of $S[t,t^{-1}]$ is radical.
\end{coro}

\begin{proof}
(i) Let ${\cal V}'_{*}$ be the intersection of ${\cal V}_{*}$ and 
${\goth g}^{f}\times \k^{*}$ and let $X$ be an irreducible component of ${\cal V}'_{*}$. 
Then ${\cal V}'_{*}$ is the nullvariety of $\poi Q1{,\ldots,}{\rg}{}{}{}$ in 
${\goth g}^{f}\times \k^{*}$ since $A_{+}$ is the ideal of $A$ generated by 
$\poi Q1{,\ldots,}{\rg}{}{}{}$. In particular, $X$ has dimension at least $r+1-\rg$.
For $z$ in $\k^{*}$, denote by $X_{z}$ the subvariety of ${\goth g}^{f}$ such that 
$X_{z}\times \{z\}=X \cap {\goth g}^{f}\times \{z\}$. By definition, for $i=1,\ldots,\rg$,
$Q_{i}=t^{-\delta _{i}}\tau \rond \kappa (q_{i})$. Hence ${\cal V}'_{*}$ is the 
nullvariety of $\poi q1{,\ldots,}{\rg}{\tau \rond \kappa }{}{}$ in 
${\goth g}^{f}\times \k^{*}$ and $X_{z}$ is the image of $X_{1}$ by the homothety
$v\mapsto z^{-1}v$. By Lemma~\ref{lsg2},(ii), $X_{1}$ has dimension $r-\rg$. Hence 
$X_{z}$ has dimension $r-\rg$ and $X$ has dimension at most $r+1-\rg$. As a result, 
$X$ has dimension $r+1-\rg$ and $X_{z}$ is strictly contained in $X$, whence the 
assertion since $X$ is not contained in ${\goth g}^{f}\times \{0\}$ by definition.

(ii) The algebra $S[t,t^{-1}]$ is graded and $A$ is a homogenous polynomial subalgebra of 
$S[t,t^{-1}]$. According to (i), the fiber at $A_{+}$ of the extension $S[t,t^{-1}]$
of $A$ is equidimensional of dimension $r+1-\rg$. Hence, by Proposition~\ref{pgf1}, 
$S[t,t^{-1}]$ is a free extension of $A$.

(iii) Let ${\cal I}_{e}$ be the ideal of $S[t,t^{-1}]$ generated by 
$\poi q1{,\ldots,}{\rg}{\tau \rond \kappa }{}{}$. Since 
$t^{\delta _{i}}Q_{i}=\tau \rond \kappa (q_{i})$ for $i=1,\ldots,\rg$, 
we get ${\cal I}_{e}=S[t,t^{-1}]A_{+} $. Denote by $\overline{\tau }$
the endomorphism of the algebra $S[t,t^{-1}]$ defined by
$$ \overline{\tau }(t)=t, \quad \overline{\tau }(x_{1}) = tx_{1},\ldots,
\overline{\tau }(x_{r}) = tx_{r}.$$
Then $\overline{\tau }$ is an automorphism and ${\cal I}_{e}=
\overline{\tau }(S[t,t^{-1}] {\cal J}_{e})$.
So, it suffices to prove that the ideal $S[t,t^{-1}] {\cal J}_{e}$ is radical. 

Let ${\cal J}'_{e}$ be the radical of $S[t,t^{-1}]{\cal J}_{e}$. 
For $a$ in $S[t,t^{-1}]$, $a$ has a unique expansion
$$ a = \sum_{m\in {\Bbb Z}} t^{m} a_{m}$$
with $(a_{m},m\in {\Bbb Z})$ a sequence of finite support in $S$. Denote by $\nu (a)$ the
cardinality of this finite support. Moreover, $a$ is in $S[t,t^{-1}]{\cal J}_{e}$ if and 
only if $a_{m}$ is in ${\cal J}_{e}$ for all $m$. Suppose that $S[t,t^{-1}]{\cal J}_{e}$ 
is strictly contained in ${\cal J}'_{e}$. A contradiction is expected. Let $a$ be in 
${\cal J}'_{e} \setminus S[t,t^{-1}] {\cal J}_{e}$ such that 
$\nu (a)$ is minimal. Denote by $m_{0}$ the smallest integer such that $a_{m_{0}}\neq 0$.
For some positive integer, $a^{k}$ and $(t^{-m_{0}}a)^{k}$ are in 
$S[t,t^{-1}] {\cal J}_{e}$ and we have
$$ (t^{-m_{0}}a)^{k} = a_{m_{0}}^{k} + \sum_{m>0} t^{m}b_{m} $$ 
with the $b_{m}$'s in ${\cal J}_{e}$. Then $a_{m_{0}}^{k}$ is in ${\cal J}_{e}$ and by 
Lemma~\ref{lsg2},(iii), $a_{m_{0}}$ is in ${\cal J}_{e}$. As a result 
$a':=a-t^{m_{0}}a_{m_{0}}$ is an element of ${\cal J}'_{e}$ such that 
$\nu (a') < \nu (a)$. By the minimality of $\nu (a)$, $a'$ is in 
$S[t,t^{-1}] {\cal J}_{e}$ and so is $a$, whence the contradiction. 
\end{proof}

Let ${\cal I}_{*}$ be the ideal of definition of ${\cal V}_{*}$ in $S[t]$. 
Then ${\cal I}_{*}$ is an ideal of $S[t]$ containing the radical of $S[t]A_+$.
It will be shown that ${\cal V}_*={\cal V}$ and that $S[t]A_+$ is radical 
(cf.\ Theorem \ref{tsg5}). Thus, ${\cal I}_*$ will be at the end equal to $S[t]A_+$. 

Let $\poi {{\goth p}}1{,\ldots,}{m}{}{}{}$ be the minimal prime ideals containing
$S[t]A_{+}$ and let $\poi {{\goth q}}1{,\ldots,}{m}{}{}{}$ be the primary decomposition 
of $S[t]A_{+}$ such that ${\goth p}_{i}$ is the radical of ${\goth q}_{i}$ for 
$i=1,\ldots,m$.

\begin{lemma}\label{l2sg2}
{\rm (i)} For $a$ in $S[t]$, $a$ is in ${\cal I}_{*}$ if and only if $t^{m}a$ is in 
$S[t]A_{+}$ for some positive integer $m$. Moreover, for some nonnegative integer $l$,
$t^{l}{\cal I}_{*}$ is contained in $S[t]A_{+}$. 

{\rm (ii)} The ideal ${\cal I}_{*}$ is the intersection of the prime ideals 
${\goth p}_{i}$ which do not contain $t$. Furthermore, for such $i$, 
${\goth p}_{i}={\goth q}_{i}$, i.e. ${\goth q}_i$ is radical. 
\end{lemma}

\begin{proof}
(i) Let $a$ be in $S[t]$. If $t^{l}a$ is in $S[t]A_{+}$ for some positive integer $l$, 
then $a$ is equal to $0$ on ${\cal V}_{*}$ so that $a$ is in ${\cal I}_{*}$. Conversely, 
if $a$ is in ${\cal I}_{*}$, then $ta$ is in the radical of $S[t]A_{+}$ since ${\cal V}$ 
is contained in the union of ${\cal V}_{*}$ and ${\goth g}^{f}\times \{0\}$. According to 
Corollary~\ref{csg2},(iii), for some positive integer $m$, $t^{m}(ta)$ is in 
$S[t]A_{+}$. Since ${\cal I}_{*}$ is finitely generated as an ideal of $S[t]$, we deduce 
that for some nonnegative integer $l$, $t^{l}{\cal I}_{*}$ is contained in $S[t]A_{+}$, 
whence the assertion.

(ii) Let $i \in\{ 1,\ldots,m\}$. Then ${\goth p}_i$ does not contain $t$ if and only if 
the nullvariety of ${\goth p}_{i}$ in ${\goth g}^{f}\times \k$ is an irreducible 
component of ${\cal V}_{*}$, whence the first part of the statement. 

By (i), for some nonnegative integer $l$, $t^{l}{\cal I}_{*}$ is contained in 
$S[t]A_{*}$. Let $l$ be the minimal nonnegative integer satisfying this condition. If 
$l=0$, ${\cal I}_{*}=S[t]A_{+}$, whence the assertion. Suppose $l$ positive. Denote by 
${\cal I}'_{*}$ the ideal of $S[t]$ generated by $t^{l}$ and $S[t]A_{+}$. It suffices to 
prove that $S[t]A_{+}$ is the intersection of ${\cal I}_{*}$ and 
${\cal I}'_{*}$. As a matter of fact, if so, the primary 
decomposition of $S[t]A_{+}$ is the union of the primary decompositions of ${\cal I}_{*}$ 
and ${\cal I}'_{*}$ since the minimal prime ideals containing ${\cal I}_{*}$ 
do not contain $t$.

Let $a$ be in the intersection of ${\cal I}_{*}$ and ${\cal I}'_{*}$. Then 
$$ a = t^{l}b + \sum_{i=1}^{\rg} a_{i}Q_{i}$$
with $b,\poi a1{,\ldots,}{l}{}{}{}$ in $S[t]$. Since $S[t]A_{+}$ is contained in 
${\cal I}_{*}$, $t^{l}b$ is in ${\cal I}_{*}$ and $b$ is in ${\cal I}_{*}$ by (i). Hence 
$t^{l}b$ and $a$ are in $S[t]A_{+}$. As a result, $S[t]A_{+}$ is the intersection of 
${\cal I}_{*}$ and ${\cal I}'_{*}$ since $S[t]A_{+}$ is contained in this intersection.
\end{proof}

\subsection{} \label{sg3}
Let $V_0$ be a homogenous complement to $SA(0)_{+}$ in $S$. We will show that the linear 
map 
$$ \tk {\k}{V_0}A(0) \longrightarrow S, \qquad v\tens a \longmapsto va$$
is a linear isomorphism (cf.\ Theorem \ref{tsg5}).

\begin{lemma}\label{lsg3} 
We have 
$S[[t]] = V_{0}[[t]]+S[[t]]A_{+}$ and $S((t)) = V_{0}((t))+S((t))A_{+}$.
\end{lemma}

\begin{proof} 
The equality $S((t)) = V_{0}((t))+S((t))A_{+}$ will follow from 
the equality $S[[t]] = V_{0}[[t]]+S[[t]]A_{+}$. 
Since $S[[t]]$, $V_{0}[[t]]$ and $S[[t]]A_{+}$ are 
homogenous, it suffices to show that for $d$ a positive integer, 
$$S[[t]]^{[d]} \subset V_{0}[[t]]^{[d]} + (S[[t]]A_{+})^{[d]},$$
the inclusion $V_{0}[[t]]+S[[t]]A_{+} \subset S[[t]]$ being obvious. 

Let $d$ be a positive integer and let $a$ be in $S[[t]]^{[d]}$. Let 
$(\poi {\varphi }1{,\ldots,}{m}{}{}{})$ be a basis of the $\k[[t]]$-module 
$(S[[t]]A_{+})^{[d]}$. Such a basis does exist since $\k[[t]]$ is a principal ring and 
$S[[t]]^{[d]}$ is a finite free $\k[[t]]$-module. Then 
$\poi 0{}{,\ldots,}{}{\varphi }{1}{m}$ generate $(SA(0)_{+})^{[d]}$. 
Since $S^{[d]}=V_{0}^{[d]}\oplus (SA(0)_{+})^{[d]}$, 
$$a - a_{0} - \sum_{j=1}^{m} a_{0,j} \varphi _{j} = t \psi_0,$$
with $a_{0}$ in $V_{0}^{[d]}$, $\poie a{{0,1}}{,\ldots,}{{0,m}}{}{}{}{}{}$ 
in $\k$ and $\psi_0 \in S[[t]]^{[d]}$. Suppose that there are sequences 
$(\poi a0{,\ldots,}{n}{}{}{})$ and 
$(\poie a{{i,1}}{,\ldots,}{{i,m}}{}{}{}{}{})$, for $i=0,\ldots,n$, in $V_{0}^{[d]}$ and 
$\k$ respectively such that 
$$a-\sum_{i=0}^{n} a_{i} t^{i}- \sum_{i=0}^{n} \sum_{j=1}^{m} t^{i}a_{i,j}\varphi _{j}
= t^{n+1} \psi_n $$ 
for some $\psi_n $ in $S[[t]]^{[d]}$. Then for some $a_{n+1}$ in $V_{0}^{[d]}$ and 
$\poie a{{n+1,1}}{,\ldots,}{{n+1,m}}{}{}{}{}{}$ in $\k$,
$$ \psi_n - a_{n+1} - \sum_{j=1}^{m} a_{n+1,j} \varphi _{j} \in tS[[t]]$$
so that
$$a-\sum_{i=0}^{n+1} a_{i}t^{i} - \sum_{i=0}^{n+1} 
\sum_{j=1}^{m}a_{i,j}\varphi _{j} t^{i} \in t^{n+2} S[[t]] .$$ 
As a result, 
$$ a \in V_{0}[[t]]^{[d]} + (S[[t]]A_{+})^{[d]}$$
since $S[[t]]^{[d]}$ is a finite $\k[[t]]$-module.
\end{proof}

Recall that $\poi {{\goth p}}1{,\ldots,}{m}{}{}{}$ are the minimal prime ideals of 
$S[t]$ containing $S[t]A_{+}$. Since $A_{+}$ is a homogenous subspace of $S[t]$, 
$S[t]A_{+}$ is a homogenous ideal of $S[t]$, and so are 
$\poi {{\goth p}}1{,\ldots,}{m}{}{}{}$. According
to Lemma~\ref{l2sg2},(ii), ${\cal I}_{*}$ is the intersection of the ${\goth p}_{i}$'s 
which do not contain $t$. Hence, ${\cal I}_{*}$ is homogenous. Thereby, 
${\cal I}_{*}\cap V_{0}[t]$ has a homogenous complement in $V_{0}[t]$. Set 
$$W:={\cal I}_{*}\cap V_{0}[t].$$
Then $W(0)$ is a homogenous subspace of $V_{0}$. Denote by $V'_{0}$ a homogenous 
complement to $W(0)$ in $V_{0}$. Then set 
$$V''_0:=W(0)$$ 
so that $V_0 = V'_{0} \oplus V''_{0}.$

\begin{lemma}\label{l2sg3}
Let $(v_{i},i\in J)$ be a homogenous basis of $V'_{0}$.

{\rm (i)} The elements $v_{i},i\in J,$ are linearly independent over $\k[t]$.

{\rm (ii)} The sum of $W$ and of $V'_{0}[t]$ is direct. 
\end{lemma}

\begin{proof}
We prove (i) and (ii) all together. 

Let $(c_{i},i\in J)$ be a sequence in $\k[t]$, with finite support $J_{c}$, such that
$$ \sum_{i\in J} c_{i}v_{i} = w$$
for some $w$ in $W$. Suppose that $J_{c}$ is not empty. A contradiction is expected. 
Since $V'_{0}$ is a complement to $V''_{0}$, $c_{i}(0)=0$ for all $i$ in $J$. Then, for 
$i$ in $J_{c}$, $c_{i}=t^{m_{i}}c'_{i}$  with $m_{i}>0$ and $c'_{i}(0)\neq 0$. Denote by 
$m$ the smallest of the integers $m_{i}$, for $i\in J_{c}$. Then $w=t^{m}w'$ for some 
$w'$ in $V_{0}[t]$, and 
$$ \sum_{i\in J_{c}} t^{m_{i}-m} c'_{i}v_{i} = w' .$$
According to Lemma~\ref{l2sg2},(i), $w'$ is in ${\cal I}_{*}$. So, $c'_{i}(0)=0$ 
for $i$ such that $m_{i}=m$, whence the contradiction.
\end{proof}

As a rule, for $M$ a $\k[t]$-submodule of $S[t]$, we denote by $\widehat{M}$ 
the $\k[[t]]$-module generated by $M$, i.e., 
$$\hat{M} = \k[[t]] M.$$

\begin{lemma}\label{l3sg3}
Let $M$ be a $\k[t]$-submodule of $S[t]$. 

{\rm (i)} Let $a$ be in the intersection of $S[t]$ and $\widehat{M}$. For
some $q$ in $\k[t]$ such that $q(0)\neq 0$, $qa$ is in $M$.

{\rm (ii)} For $N$ a $\k[t]$-submodule of $S[t]$, the intersection of 
$\widehat{M}$ and $\widehat{N}$ is the $\k[[t]]$-module generated by $M\cap N$.
\end{lemma}

\begin{proof}
(i) Denote by $\overline{a}$ the image of $a$ in 
$S[t]/M$ by the quotient map and by $J$ its annihilator in $\k[t]$. Then we have 
a commutative diagram with exact lines and columns:
$$\xymatrix{
0 \ar[r] & M \ar[r]^{\dd} & S[t] \ar[rr]^{\dd} && S[t]/M \ar[r] & 0 \\ 
0 \ar[r] & J \ar[r]^{\dd}  & \k[t] \ar[rr]^{\dd} \ar[u]^{\delta } && 
\k[t]\overline{a} \ar[r] \ar[u]^{\delta } & 0 \\ 
& & 0 \ar[u] && 0 \ar[u] & }$$
Since $\k[[t]]$ is a flat extension of $\k[t]$, tensoring this diagram by $\k[[t]]$ gives
the following diagram with exact lines and columns:
$$\xymatrix{
0 \ar[r] & \widehat{M} \ar[r]^{\dd} & S[[t]] \ar[rr]^{\dd} && 
\tk {\k[t]}{\k[[t]]}S[t]/M \ar[r] & 0 \\ 
0 \ar[r] & \k[[t]]J \ar[r]^{\dd} & \k[[t]] \ar[rr]^{\dd} \ar[u]^{\delta }
&& \k[[t]]\overline{a} \ar[r] \ar[u]^{\delta } & 0 \\ 
& & 0 \ar[u] && 0 \ar[u] & }$$
For $b$ in $\k[[t]]$, $(\delta \rond \dd) b = (\dd \rond \delta) b = 0$ since $a$ is in 
$\widehat{M}$, whence $\dd b = 0$. As a result, $\k[[t]]J=\k[[t]]$. So $qa$ is in 
$M$ for some $q$ in $\k[t]$ such that $q(0)\neq 0$. 

(ii) Since $\k[[t]]$ is a flat extension of $\k[t]$, the canonical morphism 
$$ \tk {\k[t]}{\k[[t]]}M \longrightarrow \widehat{M}.$$
is an isomorphism and from the short exact sequence
$$ 0 \longrightarrow M\cap N \longrightarrow M\oplus N \longrightarrow M+N
\longrightarrow 0$$
we deduce the short exact sequence
$$ 0 \longrightarrow \tk {\k[t]}{\k[[t]]} M\cap N \longrightarrow 
\tk {\k[t]}{\k[[t]]} (M\oplus N) \longrightarrow \tk {\k[t]}{\k[[t]]} (M+N)
\longrightarrow 0 ,$$
whence the short exact sequence
$$ 0 \longrightarrow \widehat{M\cap N} \longrightarrow 
\widehat{M}\oplus \widehat{N} \longrightarrow \widehat{M+N}
\longrightarrow 0 ,$$
and whence the assertion.
\end{proof}

\begin{prop}\label{psg3}
{\rm (i)} The space $V_{0}[[t]]$ is the direct sum of $V'_{0}[[t]]$ and $\widehat{W}$.

{\rm (ii)} The space $S[[t]]$ is the direct sum of $V'_{0}[[t]]$ and of $W+S[[t]]A_{+}$.

{\rm (iii)} The linear map
$$ \tk {\k}{V'_{0}((t))}A \longrightarrow S((t)), \qquad v\tens a \longmapsto v\tens a$$
is a homogenous isomorphism onto $S((t))$.

{\rm (iv)} For all nonnegative integer $d$,
$$ \dim S^{[d]} = \sum_{i=0}^{d} \dim {V'_{0}}^{[d-i]} \mul \dim A^{[i]} .$$
\end{prop}

\begin{proof}
(i) According to Lemma~\ref{l3sg3},(ii), the intersection of $V'_{0}[[t]]$ and 
$\widehat{W}$ is the $\k[[t]]$-submodule generated by the intersection of $V'_{0}[t]$ and
$W$. So, by Lemma~\ref{l2sg3},(iii), the sum of $V'_{0}[[t]]$ and $\widehat{W}$ is 
direct. 

Let $(v_{i},i \in J)$ be a homogenous basis of $V'_{0}$. Let $d$ be a positive integer 
and let $v$ be in $V_{0}^{[d]}$. Denote by $J_{d}$ the set of indices $i$ such that 
$v_{i}$ has degree $d$. Since $V_{0}$ is the direct sum of $V'_{0}$ and $V''_{0}$, for 
some $w$ in $W^{[d]}$ and for some $c_{i},i\in J_{d},$ in $\k$,
$$ v - \sum_{i\in J} c_{i}v_{i} = w(0).$$
Since $w-w(0)$ is in $tV_{0}[t]^{[d]}$, 
$$ v - \sum_{i\in J_{d}} c_{i} v_{i} - w \in tV_{0}[t]^{[d]}.$$
As a result,
$$ V_{0}^{[d]}[[t]] \subset {V'_{0}}^{[d]}[[t]] + \widehat{W}^{[d]} + 
tV_{0}^{[d]}[[t]] .$$
Then by induction on $m$,
$$ V_{0}^{[d]}[[t]] \subset {V'_{0}}^{[d]}[[t]] + \widehat{W}^{[d]} + 
t^{m}V_{0}^{[d]}[[t]] .$$
So, since $V_{0}^{[d]}[[t]]$ is a finitely generated $\k[[t]]$-module,
$$ V_{0}^{[d]}[[t]] = {V'_{0}}^{[d]}[[t]] + \widehat{W}^{[d]}  ,$$
whence the assertion.

(ii) According to Lemma~\ref{l2sg2},(i), for some nonnegative integer $l$, 
$t^{l}{\cal I}_{*}$ is contained in $S[t]A_{+}$. Hence $\widehat{W}+S[[t]]A_{+}$ is equal
to $W+S[[t]]A_{+}$. So, by (i) and Lemma~\ref{lsg3},
$$ S[[t]] = V'_{0}[[t]] + W + S[[t]]A_{+} .$$
According to Lemma~\ref{l2sg3},(ii), the intersection of $V'_{0}[t]$ and $S[t]A_{+}$
is equal to $\{0\}$ since $S[t]A_{+}$ is contained in ${\cal I}_{*}$. As a result, by 
Lemma~\ref{l3sg3},(ii), the intersection of $V'_{0}[[t]]$ and $S[[t]]A_{+}$ is 
equal to $\{0\}$. If $a$ is in the intersection of $V'_{0}[[t]]$ and $W+S[[t]]A_{+}$,
$t^{l}a$ is in the intersection of $V'_{0}[[t]]$ and $S[[t]]A_{+}$. So the sum of 
$V'_{0}[[t]]$ and $W+S[[t]]A_{+}$ is direct.

(iii) According to Lemma~\ref{l2sg2},(i), $W$ is contained in $S((t))A_{+}$. So, by (ii),
$$ S((t)) = V'_{0}((t)) \oplus S((t))A_{+} .$$
Since $\k[[t]]$ is a flat extension of $\k[t]$, and since 
$$S((t))=\tk {\k[t]}{\k[[t]]}S[t,t^{-1}],$$
we deduce that $S((t))$ is a flat extension of $A$ by Corollary~\ref{csg2},(ii). So, by 
Lemma~\ref{l2gf1}, all basis of $V'_{0}[[t]]$ over $\k$ consists of linearly independent 
elements over $A$. The assertion follows.

(iv) First of all, the canonical map 
$$ \tk {\k}{\k((t))} A \longrightarrow \k((t))A$$
is an isomorphism by Lemma~\ref{lsg1},(i). As a result, we have the canonical isomorphism
$$ \tk {\k((t))}{V'_{0}((t))}{\k((t))A} \longrightarrow  
\tk {\k((t))}{V'_{0}((t))} (\tk {\k}{\k((t))}A), $$
and for all nonnegative integer $i$, 
$$ \dim A^{[i]} = \dim _{\k((t))} (\k((t))A)^{[i]} .$$
From the above isomorphism, it results that the canonical morphism 
$$ \tk {\k((t))}{V'_{0}((t))}{\k((t))A} \longrightarrow \tk {\k}{V'_{0}((t))}A $$
is an isomorphism of graded spaces since 
$\tk {\k((t))}{V'_{0}((t))}\k((t)) = {V'_{0}((t))}$. As a result, by (iii), the canonical 
morphism 
$$ \tk {\k((t))}{V'_{0}((t))}{\k((t))A} \longrightarrow S((t))$$
is a homogenous isomorphism. So, for all nonnegative integer $d$, 
$$ \dim _{\k((t))}S((t))^{[d]} = 
\sum_{i=0}^{d} \dim _{\k((t))} {V'_{0}((t))}^{[d-i]} \mul 
\dim _{\k((t))} (\k((t))A)^{[i]} ,$$
whence the assertion since $\dim S^{[i]} = \dim _{\k((t))}S((t))^{[i]}$ and 
$\dim {V'_{0}}^{[i]} = \dim _{\k((t))}V'_{0}((t))^{[i]}$ for all $i$.
\end{proof}

\subsection{} \label{sg4}
Let $(w_{k}, k\in K)$ be a homogenous sequence in $W$ such that $(w_{k}(0),k\in K)$ is a 
basis of $V''_0=W(0)$. For $k$ in $K$, denote by $m_{k}$ the smallest integer such that 
$t^{m_{k}}w_{k}$ is in $S[t]A_{+}$. According to Lemma~\ref{l2sg2},(i),
$m_{k}$ is finite for all $k$. Moreover, $m_{k}$ is 
positive since $W(0)\cap SA(0)_+=\{0\}$. Set
$$ \Theta := \{(k,i) \; \vert \; k \in K, \; i \in \{0,\ldots,m_{k}-1\}\},$$
and set for all $(k,i)$ in $\Theta $, 
$$w_{k,i} := t^{i}w_{k}.$$
Let $E_{\Theta }$ be the $\k$-subspace of $V_{0}[t]$ generated 
by the elements $w_{k,i}, (k,i) \in \Theta$.

Set 
$$\widehat{{\cal I}_{*}} : =\k[[t]]{\cal I}_*.$$ 
It is an ideal of $S[[t]]$.

\begin{lemma}\label{lsg4} 
{\rm (i)} For some $q$ in $\k[t]$ such that $q(0)\neq 0$, $q{\cal I}_{*}$ is contained in
$W+S[t]A_{+}$.

{\rm (ii)} The space $W$ is contained in $E_{\Theta }+S[t]A_{+}$. Moreover, 
$\widehat{{\cal I}_{*}}$ is the sum of $E_{\Theta }$ and $S[[t]]A_{+}$.

{\rm (iii)} The sequence $(w_{k,i}, (k,i) \in \Theta)$ is a homogenous basis of 
$E_{\Theta }$.

{\rm (iv)} For all nonnegative integer $i$, $E_{\Theta }^{[i]}$ has finite dimension. 

{\rm (v)} For $i$ a nonnegative integer, there exists a nonnegative integer $l_{i}$ such 
that $t^{l_{i}}E_{\Theta }^{[i]}$ is contained in $V'_{0}[[t]]A_{+}$.
\end{lemma}

\begin{proof}
(i) Let $a$ be in ${\cal I}_{*}$. According to  Lemma~\ref{lsg3} and 
Lemma~\ref{l3sg3},(i), for some $q$ in 
$\k[t]$ such that $q(0)\neq 0$, $qa \in {\cal I}_{*}$ and $qa=a_{1}+a_{2}$ with $a_{1}$ 
in $V_{0}[t]$ and $a_{2}$ in $S[t]A_{+}$. Then $a_{1}$ is in ${\cal I}_{*}$ since so are 
$a_{2}$ and $qa$. So $a_1\in {\cal I}_{*} \cap V_0[t]=W$. The assertion follows because 
${\cal I}_{*}$ is finitely generated. 

(ii) Let us prove the first assertion. It suffices to prove
$$ W \subset E_{\Theta } + S[t]A_{+} + t^{m}S[t]$$   
for all $m$. Indeed, $W$, $E_{\Theta }$, $S[t]A_{+}$ are contained in ${\cal I}_{*}$. 
So, if $w= e+a+t^{m}b$, with $w\in W$, $e\in E_\Theta$ and $b \in S[t]$, then 
$b$ is in ${\cal I}_*$ and so, for $m$ big enough, it is in $S[t]A_+$ by 
Lemma~\ref{l2sg2},(i). 

Prove now the inclusion by induction on $m$. The inclusion is tautological for $m=0$, 
and it is true $m=1$ because $E_{\Theta }(0)=V''_{0}$.
Suppose that it is true for $m>0$. Let $w$ be in $W$. By induction hypothesis, 
$$w=a+b+t^{m}c,\quad  \text{with} \quad 
a \in E_{\Theta }, \; b \in S[t]A_{+}, \; c \in S[t] .$$
Since $E_{\Theta }$ and $S[t]A_{+}$ are contained in ${\cal I}_{*}$, $c$ is in 
${\cal I}_{*}$ by Lemma~\ref{l2sg2},(i). According to (i), for some $q$ in $\k[t]$ 
such that $q(0)\neq 0$, $qc=a'+b'$ with $a'$ in $W$ and $b'$ in $S[t]A_{+}$. Since the 
inclusion is true for $m=1$,
$$ t^{m}(a'+b') \in t^{m}E_{\Theta } + S[t]A_{+} + t^{m+1}S[[t]], $$
and by definition, $t^{m}E_{\Theta }$ is contained in $E_{\Theta }+S[t]A_{+}$. Moreover, 
$q(0)c$ is in $qc+tS[t]$. Then
$$ t^{m}c \in E_{\Theta }+S[t]A_{+}+t^{m+1}S[t] \quad  \text{and} \quad 
w \in E_{\Theta }+S[t]A_{+}+t^{m+1}S[t] ,$$
whence the statement. 

Turn to the second assertion. By (i), $\widehat{{\cal I}_{*}}$ is the sum of 
$\widehat{W}$ and $S[[t]]A_{+}$. An element of $\widehat{W}$ is the sum of terms 
$t^m w_m$, with $m\in\N$ and $w_n\in W$. For $m$ big enough, $t^m w_m \in S[t]A_+$ by 
Lemma~\ref{l2sg2},(i). So $\widehat{{\cal I}_{*}}$ is the sum of $W$ and 
$S[[t]]A_{+}$, whence the assertion by the previous inclusion.

(iii) By definition, the elements $w_{k,i}, (k,i) \in \Theta,$ are homogenous. So it 
suffices to prove that they are linearly independent over $\k$. Let 
$(c_{k,i}, (k,i)\in \Theta)$ be a sequence in $\k$, with finite support, such that 
$$ \sum_{k\in K} \sum_{i=0}^{m_{k}-1} c_{k,i} w_{k,i} = 0.$$ 
Let us prove that $c_{k,i}=0$ for all $(k,i)$. Suppose $c_{k,i}\neq 0$ for some $(k,i)$. 
A contradiction is expected. Let $K'$ be the set of $k$ such that $c_{k,i}\neq 0$ for
some $i$. Denote by $i_{0}$ the smallest integer such that $c_{k,i_{0}}\neq 0$ for some 
$k$ in $K'$ and set:
$$ K'_{0} := \{k \in K' \; \vert \; c_{k,i_{0}} \neq 0 \}.$$
Then 
$$ \sum_{k\in K'_{0}} c_{k,i_{0}} w_{k}(0) = 0 ,$$
whence the contradiction since the elements $(w_{k}(0),k \in K)$ are linearly independent.

(iv) Let $K_{i}$ be the set of $k$ such that $w_{k}$ is in $S[t]^{[i]}$. For such $k$,
$w_{k}(0)$ is in $S^{[i]}$. Hence $K_{i}$ is finite since $S^{[i]}$ has finite dimension
and since the elements $(w_{k}(0),k\in K)$ are linearly independent. For $k$ in $K$, 
$\k[t]w_{k}\cap E_{\Theta }$ has dimension $m_{k}$ by (iii). Hence $E_{\Theta }^{[i]}$ 
has finite dimension. 

(v) Let $k$ be in $K_{i}$. Set 
$$\Theta ^{[i]} := \Theta \cap (K_{i}\times {\Bbb N}).$$
By Proposition~\ref{psg3},(iii), $t^{l+m_{k}}w_{k}$ is in $V'_{0}[[t]]A_{+}$ since 
$t^{m_{k}}w_{k}$ is in $S[t]A_{+}$ by definition, whence the assertion since 
$E_{\Theta }^{[i]}$ is generated by the finite sequence 
$(w_{k,j}, (k,j) \in \Theta ^{[i]})$.
\end{proof}

\begin{defi} 
We say that a subset $T$ of $\Theta $ is {\it complete} if 
$$ (k,i) \in T \Longrightarrow (k,j) \in T , \; \forall j \in \{0,\ldots,i\} .$$
\end{defi}

For $T$ subset of $\Theta $, denote by $K_{T}$ the image of $T$ by the projection
$(k,i)\mapsto k$, and by $E_{T}$ the subspace of $E_{\Theta }$ generated by 
the elements $w_{k,i}$, $(k,i) \in T$. In particular, $K_\Theta=K$. 

\begin{lemma}\label{l2sg4}
For some complete subset $T$ of $\Theta $ such that $K_{T}=K$, the subspace 
$E_{T}$ is a complement to $S[t]A_{+}$ in $E_{\Theta }+S[t]A_+$. In particular, the sum 
of $E_{T}$ and $S[t]A_{+}$ is direct. 
\end{lemma}

\begin{proof}
Since $V''_{0}\cap SA(0)_{+}=\{0\}$, the sum of $E_{K\times \{0\}}$ and $S[t]A_{+}$ is 
direct. Let ${\cal T}$ be the set of complete subsets $T$ of $\Theta $ satisfying the 
following conditions:
\begin{itemize}
\item [{\rm (1)}] for all $k$ in $K$, $(k,0)$ is in $T$, 
\item [{\rm (2)}] the sum of $E_{T}$ and $S[t]A_{+}$ is direct.
\end{itemize}
Since the sum of $E_{K\times \{0\}}$ and $S[t]A_{+}$ is direct, ${\cal T}$ is not empty. 
If $(T_{j},j\in {\goth J})$ is an increasing sequence of elements of ${\cal T}$,
with respect to the inclusion, its union is in ${\cal T}$. Then, by Zorn's Lemma, 
${\cal T}$ has a maximal element. Denote it by $T_{*}$. It remains to prove
that $w_{k,i}$ is in $E_{T_{*}}+S[t]A_{+}$ for all $(k,i)$ in $\Theta $. 

Let $k$ be in $K$. Denote by $i$ the biggest integer such that $(k,i)$ is in 
$T_{*}$. Prove by induction on $i'$ that for $m_{k}>i'>i$, $w_{k,i'}$ is 
in $E_{T_{*}}+S[t]A_{+}$. By maximality of $T_{*}$ and $i$,
$w_{k,i+1}$ is in $E_{T_{*}}+S[t]A_{+}$. Suppose that $w_{k,i'}$ is 
in $E_{T_{*}}+S[t]A_{+}$. Then, for some $a$ in $S[t]A_{+}$ and 
$c_{m,j}, (m,j) \in T_{*}$ in $\k$, 
$$w_{k,i'} =  \sum_{(m,j) \in T_{*}} c_{m,j} w_{m,j} + a,$$
whence
$$w_{k,i'+1} =  \sum_{(m,j) \in T_{*}} c_{m,j} t^{j+1}w_{m} + ta.$$ 
By maximality of $T_{*}$, $t^{j+1}w_{m}$ is in $E_{T_{*}}+S[t]A_{+}$
for all $(m,j)$ such that $t^{j}w_{m}$ is in $T_{*}$. Hence 
$w_{k,i'+1}$ is in $E_{T_{*}}+S[t]A_{+}$. The lemma follows.
\end{proof}

Fix a complete subset $T_{*}$ of $\Theta $ such that 
$$K_{T_{*}}=K \qquad  \text{and} \qquad 
E_{\Theta }+S[t]A_{+} = E_{T_{*}} \oplus S[t]A_{+} ,$$
and set 
$$E := E_{T_{*}}.$$
Such a set $T_*$ does exist by Lemma \ref{l2sg4}. 

\begin{coro}\label{csg4}
{\rm (i)} The space $S[[t]]$ is the direct sum of $V'_{0}[[t]]$, $E$ and $S[[t]]A_{+}$.

{\rm (ii)} The space $S[[t]]$ is the sum of $EA$ and $V'_{0}[[t]]A$.
\end{coro}

\begin{proof}
(i) According to Proposition~\ref{psg3},(ii), $S[[t]]$ is the direct sum of $V'_{0}[[t]]$ 
and $W+S[[t]]A_{+}$. By Lemma~\ref{lsg4},(ii) (and its proof), $W+S[[t]]A_{+}$ is equal 
to $E_{\Theta }+S[[t]]A_{+}$. Since $E_{\Theta }+S[t]A_{+}$ is the direct sum of $E$ and 
$S[t]A_{+}$, we deduce that $W+S[[t]]A_{+}$ is the direct sum of $E$ and $S[[t]]A_{+}$. 
Hence, $S[[t]]$ is the direct sum of $V'_{0}[[t]]$, $E$ and $S[[t]]A_{+}$.

(ii) By (i) and by induction on $m$,
$$ S[[t]] \subset V'_{0}[[t]]A + EA + S[[t]]A_{+}^{m} .$$
Hence $S[[t]]$ is the sum of $V'_{0}[[t]]A$ and $EA$ since $S[[t]]$ is graded and $A_{+}$
is generated by elements of positive degree.
\end{proof}

\begin{defi} 
For $k$ in $K$, denote by $\nu_{k}$ the degree of $w_{k}$. For $T$ and $T'$ 
subsets of $\Theta $, we say that $T$ {\em is smaller than} $T'$, and we denote 
$T \prec T'$, if the following conditions are satisfied:
\begin{itemize}
\item [{\rm (1)}] $T$ is contained in $T'$ 
\item [{\rm (2)}] if for $k$ in $K_{T}$ and $k'$ in $K_{T'}$, we have 
$\nu _{k'} < \nu_{k}$, then $k'$ is in $K_{T}$.
\end{itemize}
\end{defi} 

Let $\mu $ be the linear map
$$ \tk {\k}EA \oplus \tk {\k}{V'_{0}[[t]]}A \longrightarrow S[[t]], \qquad
w\tens a + v\tens b \longmapsto wa + vb .$$
For $T$ a subset of $T_{*}$, denote by $\mu _{T}$ the restriction of $\mu $ to the 
subspace
$$ \tk {\k}{E_{T}}A \oplus \tk {\k}{V'_{0}[[t]]}A .$$
 
\begin{lemma}\label{l3sg4}
Let ${\cal T}_*$ be the set of subsets $T$ of $T_{*}$ such that $\mu _{T}$ is injective.

{\rm (i)} The set ${\cal T}_*$ is not empty.

{\rm (ii)} The set ${\cal T}_*$ has a maximal element with respect to the order $\prec $.

{\rm (iii)} The set $T_{*}$ is in ${\cal T}_*$.
\end{lemma}

\begin{proof}
(i) For $k$ in $K$, set $T_{k} := \{(k,0)\}$. Suppose that $T_{k}$ is not in 
${\cal T}_*$. A contradiction is expected. Then for some $a$ in $A\setminus \{0\}$, 
$w_{k}a$ is in $V'_{0}[[t]]A_{+}$, whence
$$ w_{k}a = \sum_{i\in J} v_{i}b_{i}$$ 
with $(b_{i},i\in J)$ in $\k[[t]]A_{+}$ with finite support. By Lemma~\ref{lsg4},(v), 
for some positive integer, $t^{l}w_{k}$ is in $V'_{0}[[t]]A_{+}$. Then 
$$ t^{l}w_{k} = \sum_{i\in J} v_{i}c_{i}$$ 
with $(c_{i},i\in J)$ in $\k[[t]]A_{+}$ with finite support. Hence
$$ \sum_{i\in J} v_{i}t^{l}b_{i} = \sum_{i\in J} v_{i}c_{i}a .$$
According to Proposition~\ref{psg3},(iii), $t^{l}b_{i}=c_{i}a$ for all $i$. Since
$a\neq 0$, $a(0)\neq 0$ by Proposition~\ref{psg1},(ii). Then, by Lemma \ref{lsg1},(ii), 
$c_{i}=t^{l}c'_{i}$ for some $c'_{i}$ in $\tilde{A}=\k[[t]]A$. As a result, 
$$ w_{k} = \sum_{i\in J} v_{i}c'_{i} ,$$
whence the contradiction by Corollary~\ref{csg4},(i).

(ii) Let $(T_{l},l\in L)$ be a net in ${\cal T}_*$ with respect to 
$\prec$. Let $T$ be the 
union of the sets $T_{l}$, $l\in L$. Since $E_{T}$ is the space generated by the 
subspaces $E_{T_{l}}, l\in L$, the map  $\mu _{T}$ is injective. Let $l_{0}$ be in $L$ 
and $k$ in $K_{T}$ such that $\nu_{k} < \nu _{k'}$ for some $k'$ in $K_{T_{l_{0}}}$. 
Since $K_{T}$ is the union of the sets $K_{T_{l}}$, $l\in L$, we deduce that $k$ is in 
$K_{T_{l}}$ for some $l$ in $L$. By properties of the nets, for some $l'$ in $L$, 
$T_{l}\prec T_{l'}$ and $T_{l_{0}}\prec T_{l'}$ so that $k$ is in $K_{T_{l'}}$. Hence, 
$k$ is in $K_{T_{l_{0}}}$, whence $T_{l_{0}}\prec T$. As a result, $\prec $ is an 
inductive order in ${\cal T}_{*}$, and by Zorn's Theorem, it has a maximal element.

(iii) Let $T$ be a maximal element of ${\cal T}_*$ with respect to $\prec$. Suppose $T$
strictly contained in $T_{*}$. A contradiction is expected. Let $k$ be in $K$ such that 
$(k,i)$ is not in $T$ and $(k,i)$ is in $T_{*}$ for some $i$. We can suppose 
that $\nu_{k}$ is minimal under this condition. Let $i_{*}$ be the smallest integer 
such that $(k,i_{*})$ is not in $T$ and $(k,i_{*})$ is in $T_{*}$. Then 
$T \prec T\cup \{(k,i_{*})\}$. So, by the maximality of $T$, for some $a$ in 
$A\setminus \{0\}$, 
$$ w_{k,i_{*}}a \in E_{T}A + V'_{0}[[t]]A .$$
Since $E_{T}$, $V'_{0}[[t]]$, $A$, $w_{k,i_{*}}$ are homogenous, we can suppose that $a$ 
is homogenous. Then $a$ has positive degree.
Otherwise, 
$ w_{k,i_{*}}\in E_{T}A + V'_{0}[[t]]A \subset   E_{T} + V'_{0}[[t]] +S[[t]]A_+$, 
and we deduce from Corollary~\ref{csg4},(i), that $w_{k,i_{*}} \in E_{T}$ since 
$w_{k,i_{*}} \in E_{T_*}$. This is impossible by the choice of $(k,i_*)$. 
Thus, by Corollary~\ref{csg4},(ii),
$$ w_{k,i_{*}}a \in E_{T}A_{+} + V'_{0}[[t]]A_{+}.$$
Hence
$$w_{k,i_{*}}a = \sum_{(n,j)\in T} w_{n,j}a_{n,j} + \sum_{i\in J} v_{i}b_{i}$$
with $(a_{n,j},(n,j)\in T)$ in $A_{+}$  and $(b_{i},i\in J)$ in 
$\tilde{A}_{+}$ with finite support. 

By Corollary~\ref{csg4},(ii), 
$$ t^{m_{k}}w_{k} = \sum_{(l,s)\in T_{*}} w_{l,s} a_{l,s,k} + 
\sum_{i\in J} v_{i}b_{i,k} $$
with $(a_{l,s,k},(l,s) \in T_{*})$ in $A_{+}$ and $(b_{i,k},i\in J)$ in $\tilde{A}_{+}$
with finite support. Moreover these two sequences are homogenous, so that $a_{l,s,k}=0$
if $\nu _{l}\geq \nu_{k}$. By minimality of $\nu_{k}$, $(l,s)$ is in $T$ if 
$a_{l,s,k}\neq 0$. For $(n,j)$ in $T$ such that $m_{k}-i_{*}+j\geq m_{n}$, 
$$ t^{m_{k}-i_{*}}w_{n,j} = \sum_{(l,s)\in T_{*}} w_{l,s} a_{l,s,n,j} + 
\sum_{i\in J} v_{i}b_{i,n,j} $$
with $(a_{l,s,n,j},(l,s) \in T_{*})$ in $A_{+}$ and $(b_{i,n,j},i\in J)$ in 
$\tilde{A}_{+}$ with finite support. Moreover these two sequences are homogenous, so that 
$a_{l,s,n,j}=0$ if $\nu _{l}\geq \nu _{n}$. So, by minimality of $\nu_{k}$, 
$(l,s)$ is in $T$ if $a_{l,s,n,j}\neq 0$ and $\nu _{n}\leq \nu_{k}$. As a result, 
\begin{eqnarray*}
\sum\limits_{(l,s)\in T} w_{l,s} a_{l,s,k}a + \sum\limits_{i\in J} v_{i}b_{i,k}a  &= & 
 \sum\limits_{(n,j)\in T} w_{n,j}t^{m_{k}-i_{*}}a_{n,j} 
+ \sum\limits_{i\in J} v_{i} t^{m_{k}-i_{*}}b_{i} \\
\smallskip\\
&= & \sum_{\genfrac {}{}{0pt}{}{(n,j)\in T}
{ m_{k}-i_{*}+j < m_{n}}} w_{n,m_{k}-i_{*}+j}a_{n,j}
+ \sum_{\genfrac {}{}{0pt}{}{(n,j)\in T}{ m_{k}-i_{*}+j \geq  m_{n}}}
w_{l,s}a_{l,s,n,j}a_{n,j} \\ \smallskip\\
& & + \sum_{i\in J} v_{i}t^{m_{k}-i^{*}}b_{i}
+ \sum_{\genfrac {}{}{0pt}{}{(n,j)\in T}{ m_{k}-i_{*}+j \geq  m_{n}}} 
\sum_{i\in J} v_{i}b_{i,n,j}a_{n,j}
\end{eqnarray*}
whence
\begin{eqnarray*}
\sum_{(l,s)\in T} w_{l,s} a_{l,s,k}a + \sum_{i\in J} v_{i}b_{i,k}a& = &
\sum_{\genfrac {}{}{0pt}{}{(n,j)\in T}{m_{k}-i_{*}+j<m_{n}}} w_{n,m_{k}-i_{*}+j}a_{n,j}
+ \sum_{\genfrac {}{}{0pt}{}{(n,j)\in T}{m_{k}-i_{*}+j \geq m_{n}}} 
\sum_{(l,s)\in T} w_{l,s} a_{l,s,n,j}a_{n,j} \\
\smallskip\\
& &+ \sum_{i\in J} v_{i}(t^{m_{k}-i_{*}}b_{i} 
+ \sum_{\genfrac {}{}{0pt}{}{(n,j)\in T}{m_{k}-i_{*}+j \geq m_{n}}} b_{i,n,j}a_{n,j}) .
\end{eqnarray*}
Since $\mu _{T}$ is injective, for all $i$ in $J$, 
\begin{eqnarray}\label{eqsg4}
t^{m_{k}-i_{*}} b_{i} + 
\sum_{\genfrac {}{}{0pt}{}{(n,j)\in T}{m_{k}-i_{*}+j \geq m_{n}}} 
b_{i,n,j}a_{n,j} - b_{i,k}a = 0,  
\end{eqnarray}
and for all $(l,s)$ in $T$, 
\begin{eqnarray}\label{eq2sg4}
a_{l,s+i_{*}-m_{k}} + \sum\limits_{\genfrac {}{}{0pt}{}{(n,j)\in T}
{m_{k}-i_{*}+j \geq m_{n}}} a_{n,j}a_{l,s,n,j} - a_{l,s,k}a = 0 .
\end{eqnarray}
with $a_{l,s}=0$ if $s<0$. 

\begin{claim}\label{clsg4}
For all $(l,s)$ in $T$, $a$ divides $a_{l,s}$ in $A$.
\end{claim}

\begin{proof}[Proof of Claim~\ref{clsg4}]
Prove the claim by induction on $\nu _{l}$. Let $l$ be in $K_{T}$ such that 
$$ \nu _{l'} > \nu _{l} \quad  \text{and} \quad (l',s') \in T \Longrightarrow 
a_{l',s'} = 0 .$$ 
Then by Equality (\ref{eq2sg4}), $a_{l,s+i_{*}-m_{k}}=a_{l,s,k}a$, whence the satement 
for $l$. Suppose that $a$ divides $a_{l',s'}$ in $A$ for all $(l',s')$ in $T$ such that
$\nu _{l'}>\nu _{l}$. By Equality (\ref{eq2sg4}) and the induction hypothesis, $a$ 
divides $a_{l,s+i_{*}-m_{k}}$ in $A$ since $a_{l,s,n,j}=0$ for 
$\nu _{n}\leq \nu _{l}$, whence the claim.
\end{proof}

By Claim~\ref{clsg4} and Equality (\ref{eqsg4}), for all $i$ in $J$, $a$ divides 
$t^{m_{k}-i_{*}}b_{i}$ in $\k[[t]]A$. Since $a$ has positive degree, all prime divisor
of $a$ in $A$ has positive degree and does not divide $t$ since $t$ has degree $0$. Then, 
by Lemma~\ref{lsg1},(iii), $a$ divides $b_{i}$ in $\k[[t]]A$. As a result, 
$$ w_{k,i_{*}} \in E_{T}A + V'_{0}[[t]]A$$
whence
$$w_{k,i_{*}} \in V'_{0}[[t]] + E_{T} + S[[t]]A_{+} .$$
Since $w_{k,i_{*}}$ is in $E$, $w_{k,i_{*}}$ is in $E_{T}$ by Corollary~\ref{csg4},(i). 
We get a contradiction because $(k,i_{*})$ is not in $T$.
\end{proof}

\begin{coro}\label{c2sg4}
The canonical map
$$ \tk {\k}EA \oplus \tk {\k}{V'_{0}[[t]]}A \longrightarrow S[[t]]$$
is an isomorphism. In particular, $S[[t]]$ is a free extension of $A$.
\end{coro}

\begin{proof}
By Lemma~\ref{l3sg4}, $T_{*}$ is the biggest element of ${\cal T}_*$. Hence $\mu $
is injective. Then, by Corollary~\ref{csg4},(ii), $\mu $ is bijective. As a matter of
fact, $\mu $ is an isomorphism of $A$-modules, whence the corollary.
\end{proof}

\subsection{} \label{sg5}
Recall that $\tilde{A}$ is the subalgebra of $S[[t]]$ generated by $\k[[t]]$ and $A$. 
Our next aim is to show that $S[[t]]$ is a free extension of $\tilde{A}$ 
(cf. Theorem \ref{tsg5}). Theorem~\ref{tge2} will then follows. 

For $I$ an ideal of $\tilde{A}$, 
denote by $\sigma _{I}$ and $\nu _{I}$ the canonical morphisms
$$ \xymatrix{ \tk {A}{S[[t]]}I \ar[r]^{\sigma _{I}} & \tk {A}{S[[t]]}{\tilde{A}}
&& \tk {\tilde{A}}{S[[t]]}I \ar[r]^{\nu _{I}} & S[[t]]I } .$$
Consider on $\tk {A}{S[[t]]}I$ and 
$\tk {\tilde{A}}{S[[t]]}I$ the linear topologies such that 
$\{t^{n}(\tk {A}{S[[t]]}I)\}_{n\in {\Bbb N}}$ and 
$\{t^{n}(\tk {\tilde{A}}{S[[t]]}I)\}_{n\in {\Bbb N}}$ are systems of neighborhood of
$0$ in these $S[[t]]$-modules. Denote by $\varphi _{I}$ the canonical morphism 
$$ \xymatrix{ \tk {A}{S[[t]]}I \ar[r]^{\varphi _{I}} & \tk {\tilde{A}}{S[[t]]}I }$$
and by ${\cal K}_{I}$ its kernel. Then $\varphi _{I}$ is continuous with respect to the 
above topologies.

\begin{lemma}\label{lsg5}
Let $I$ be an ideal of $\tilde{A}$.

{\rm (i)} The morphism $\sigma _{I}$ is injective.

{\rm (ii)} The module ${\cal K}_{I}$ is the $S[[t]]$-submodule of $\tk {A}{S[[t]]}I$ 
generated by the elements $r\tens a - 1\tens ra$ with $r$ in $\k[[t]]$ and $a$ in $I$. 
\end{lemma}

\begin{proof}
(i) According to Corollary~\ref{c2sg4}, $S[[t]]$ is a flat extension of $A$. 
The assertion follows since $I$ is contained in $\tilde{A}$.

(ii) Let ${\cal K}'_{I}$ be the $S[[t]]$-submodule of $\tk {A}{S[[t]]}I$ generated by the 
elements $r\tens a - 1\tens ra$  with $r$ in $\k[[t]]$ and $a$ in $I$. Then 
${\cal K}'_{I}$ is contained in ${\cal K}_{I}$. Let us prove the opposite inclusion. 

Let $(x,y)$ be in $S[[t]]\times I$ and let $a$ be in $\tilde{A}$. According to (i),
$a$ has an expansion
$$ a = \sum_{i=1} r_{i}a_{i}$$ 
with $\poi r1{,\ldots,}{m}{}{}{}$ in $\k[[t]]$ and $\poi a1{,\ldots,}{m}{}{}{}$ in 
$A$. Then, in $\tk {A}{S[[t]]}I$,
$$x\tens ay - ax\tens y = \sum_{i=1}^{m} x\tens r_{i}a_{i}y - r_{i}x\tens a_{i}y =
\sum_{i=1}^{m} x(1\tens r_{i}a_{i}y - r_{i}\tens a_{i}y) \in {\cal K}'_{I} .$$
As a result, ${\cal K}_{I}={\cal K}'_{I}$ since ${\cal K}_{I}$ is the 
S[[t]]-submodule of $\tk A{S[[t]]}I$ generated by the $xa\tens y-x\tens ay$'s.  
\end{proof}

\begin{coro}\label{csg5}
Let $I$ be an ideal of $\tilde{A}$. The module ${\cal K}_{I}$ is the closure of the
$S[[t]]$-submodule of $\tk {A}{S[[t]]}I$ 
generated by the set $\{t\tens a-1\tens ta\}_{a\in I}$.
\end{coro}

\begin{proof}
Let ${\cal L}_{I}$ be the $S[[t]]$-submodule generated by the set 
$\{t\tens a-1\tens ta\}_{a\in I}$. Prove by induction on $n$ that 
$t^{n}\tens a - 1\tens t^{n}a$ is in ${\cal L}_{I}$ for all $a$ in $I$. The statement 
is straightforward for $n=0,1$. Suppose $n\geq 2$ and the statement true for $n-1$. 
For $a$ in $I$, 
$$ t^{n}a-1\tens t^{n}a = t^{n-1}(t\tens a- 1\tens ta) +
t^{n-1}\tens ta - 1\tens t^{n-1}ta .$$
By induction hypothesis, $t^{n-1}\tens ta - 1\tens t^{n-1}ta$ is in ${\cal L}_{I}$, 
whence $t^{n}\tens a - 1\tens t^{n}a$ is in ${\cal L}_{I}$. As a result, for $r$ in
$\k[t]$, $r\tens a - 1\tens ra$ is in ${\cal L}_{I}$. So, for $r$ in $\k[[t]]$, 
$r\tens a - 1 \tens ra$ is in the closure of ${\cal L}_{I}$ in $\tk {A}{S[[t]]}I$.
Since $\varphi _{I}$ is continuous, ${\cal K}_{I}$ is a closed submodule of 
$\tk {A}{S[[t]]}I$, whence the corollary by Lemma~\ref{lsg5},(iii).
\end{proof}

\begin{prop}\label{psg5}
Let $I$ be an ideal of $\tilde{A}$. 

{\rm (i)} The canonical morphism 
$$ \tk {\tilde{A}}{V'_{0}\tilde{A}}I \longrightarrow \tk {\tilde{A}}{S[[t]]}I$$
is an embedding.

{\rm (ii)} For the structure of $S[[t]]$-module on $\tk {\tilde{A}}{S[[t]]}I$,
$t$ is not a divisor of $0$ in $\tk {\tilde{A}}{S[[t]]}I$.
\end{prop}

\begin{proof}
(i) We have the commutative diagram
$$\xymatrix{ \tk {\tilde{A}}{V'_{0}\tilde{A}}I \ar[r]^{\dd} \ar[d]^{\delta} & 
\tk {\tilde{A}}{S[[t]]}I \ar[d]^{\delta} \\
V'_{0}I \ar[r]^{\dd} & S[[t]]I } $$
with canonical arrows $\dd$ and $\delta$. 
According to Proposition~\ref{psg3},(iii),
the left down arrow $\delta$ is an isomorphism. Let $a$ be in 
$\tk {\tilde{A}}{V'_{0}\tilde{A}}I$ such that $\dd a = 0$. Then 
$\dd \rond \delta a = 0$, whence $\delta a = 0$ since the bottom horizontal arrow
$\dd$ is an embedding so that $a = 0$.

(ii) Let $a$ be in $\tk {A}{S[[t]]}I$ such that $t\varphi _{I}(a) = 0$. According to 
Corollary~\ref{csg5}, for $l$ in ${\Bbb N}$ such that $l\geq 2$, 
$$ ta - \sum_{i=1}^{m} b_{i}(t\tens a_{i}-1\tens ta_{i}) \in 
t^{l}\tk {A}{S[[t]]}I$$
for some $\poi b1{,\ldots,}{m}{}{}{}$ in $S[[t]]$ and for some
$\poi a1{,\ldots,}{m}{}{}{}$ in $I$. For $i=1,\ldots,m$, 
$$ b_{i} = b_{i,0} + tb'_{i}$$
with $b_{i,0}$ in $S$ and $b'_{i}$ in $S[[t]]$, whence
$$ t(a-\sum_{i=1}^{m} b'_{i}(t\tens a_{i}-1\tens ta_{i})) - 
\sum_{i=1}^{m} b_{i,0}(t\tens a_{i}-1-\tens ta_{i})  \in 
t^{l}\tk {A}{S[[t]]}I .$$
Set:
$$ a' := a-\sum_{i=1}^{m} b'_{i}(t\tens a_{i}-1\tens ta_{i}) \quad  \text{and} \quad
a'' = \sum_{i=1}^{m} b_{i,0}(t\tens a_{i}-1\tens ta_{i}) .$$
Then $\varphi _{I}(a)=\varphi _{I}(a')$ and $\sigma _{I}(a'')$ is in 
$t\tk {\k}{S[[t]]}\k[[t]]$. Moreover, for $i=1,\ldots,m$, $a_{i}$ has a unique expansion
$$ a_{i} = \sum_{n\in {\Bbb N}} t^{n}a_{i,n}$$
with $a_{i,n},n\in {\Bbb N},$ in $A$. Then
\begin{eqnarray*}
\sigma _{I}(a'')& =&  
\sum_{i=1}^{m} b_{i,0} (\sum_{n\in {\Bbb N}}ta_{i,n}\tens t^{n} - a_{i,n}\tens t^{n+1})\\
&=&  t\sum_{i=1}^{m} a_{i,0}b_{i,0}\tens 1 + 
\sum_{n\in {\Bbb N}^{*}} \sum_{i=1}^{m} b_{i,0}(ta_{i,n}-a_{i,n-1})\tens t^{n} .
\end{eqnarray*}
Since the right hand side is divisible by $t$ in $\tk {\k}{S[[t]]}\k[[t]]$, for all 
positive integer $n$, 
$$ \sum_{i=1}^{m} b_{i,0}a_{i,n-1} = 0$$
since $b_{i,0}$ and $a_{i,n-1}$ are in $S$ for all $i$. Hence $\sigma _{I}(a'')=0$
and $a''=0$ by Lemma~\ref{lsg5},(i). Thus, 
$$ a' \in t^{l-1}\tk {A}{S[[t]]}I .$$
As a result, $\varphi _{I}(a)$ is in $t^{l}\tk {\tilde{A}}{S[[t]]}I$ for all positive 
integer $l$. Since the $S[[t]]$-module $\tk {\tilde{A}}{S[[t]]}I$ is finitely 
generated, by a Krull's theorem~\cite[Ch. 3, Theoreom 8.9]{Mat}, for some $b$ in $S[[t]]$,
$(1+tb)\varphi _{I}(a)=0$, whence $\varphi _{I}(a)=0$ since $t\varphi _{I}(a)=0$.
\end{proof}

Remind that ${\cal V}_{0}$ is the nullvariety of $A(0)_{+}$ in ${\goth g}^{f}$, 
and that $\tilde{A}=\k[[t]]A$.

\begin{theorem}\label{tsg5}
{\rm (i)} The algebra $S[[t]]$ is a free extension of $\tilde{A}$.

{\rm (ii)} The varieties ${\cal V}$ and ${\cal V}_{*}$ are equal. Moreover, 
${\cal V}_{0}$ is equidimensional of dimension $r-\rg$.

{\rm (iii)} The $A(0)$-module $S$ is free and $V_{0}=V'_{0}$. 
In particular, the canonical morphism 
$$V_0 \otimes_\k A(0) \longrightarrow S,\quad v\tens a \longmapsto va$$ 
is an isomorphism. 
\end{theorem}

\begin{proof}
(i) First of all, prove that $S[[t]]$ is a flat extension of $\tilde{A}$. Then the 
freeness of the extension will result from the equality $V_{0}=V'_{0}$, Lemma~\ref{lsg3}
and Proposition~\ref{psg3},(iii).

By the criterion of flatness \cite[Ch. 3, Theorem 7.7]{Mat}, it is equivalent to say that
for all ideal $I$ of $\tilde{A}$, the canonical morphism $\nu _{I}$,
$$ \tk {\tilde{A}}{S[[t]]}I \longrightarrow S[[t]]I$$
is injective. Let $a$ be in the kernel of $\nu _{I}$. Consider the commutative 
diagram
$$\xymatrix{ \tk {\tilde{A}}{V'_{0}\tilde{A}}I \ar[r]^{\dd} \ar[d]^{\delta } & 
\tk {\tilde{A}}{S[[t]]}I \ar[d]^{\delta } \\
V'_{0}I \ar[r]^{\dd} & S[[t]]I } $$
of the proof of Proposition~\ref{psg5},(i). According to Lemma~\ref{lsg4},(v), for $l$ 
sufficiently big, $t^{l}a=\dd b $ for some $b$ in 
$\tk {\tilde{A}}{V'_{0}\tilde{A}}I$. Then $\delta b = 0$ since 
$\nu _{I}(t^{l}a) = 0$. By Proposition~\ref{psg3},(iii), $\delta $ is an isomorphism. 
Hence $b=0$ and $t^{l}a=0$. Then, by Proposition~\ref{psg5},(ii), $a=0$, whence the 
the flatness.

(ii) Denote by $\k[t]_{0}$ the localization of $\k[t]$ at $t\k[t]$. Then $\k[[t]]$ is 
a faithfully flat extension of $\k[t]_{0}$. Hence, $S[[t]]$ is a faithfully flat 
extension of 
$$S[t]_{0}:=\tk {\k[t]}{\k[t]_{0}}S.$$ Set   
$$\tilde{A}_0 := \tk {\k}{\k[t]_{0}}A.$$
Then  
$$\tilde{A} = \tk {\k[t]_{0}}{\k[[t]]}\tilde{A}_0$$ 
so that $\tilde{A}$ is 
faithfully flat extension of $\tilde{A}_0$. For $M$ a $\tilde{A}_0$-module, we have 
$$ \tk {\k[t]_{0}}{\k[[t]]}(\tk {\tilde{A}_0}{S[t]_{0}}M) =
\tk {\tilde{A}_0}{(\tk {\k[t]_{0}}{\k[[t]]}S[t]_{0})}M = 
\tk {\tilde{A}}{S[[t]]}(\tk {\tilde{A}_0}{\tilde{A}}M) .$$ 
Hence, $S[t]_{0}$ is a flat extension of $\tilde{A}_0$ since so is the extension
$S[[t]]$ of $\tilde{A}$. 

The variety ${\cal V}$ is the union of ${\cal V}_{*}$ and ${\cal V}_{0}\times \{0\}$. 
Moreover ${\cal V}_{0}\times \{0\}$ is the nullvariety in ${\goth g}^{f}\times \k$ of 
the ideal of $\k[t]A$ generated by $t$ and $A_{+}$. Then, by 
\cite[Ch. 5, Theorem 15.1]{Mat}, ${\cal V}_{0}$ is equidimensional of dimension $r-\rg$ 
since $S[t]_{0}$ is a flat extension of $\tilde{A}_0$ by (i) and since $\tilde{A}_{0}$ 
has dimension $\rg+1$. Since ${\cal V}$ is the nullvariety of $\rg$ functions, all 
irreducible component of ${\cal V}$ has dimension at least $r+1-\rg$ by 
\cite[Ch. 5, Theorem 13.5]{Mat}. Hence any irreducible component of 
${\cal V}_{0}\times \{0\}$ is not an irreducible component of 
${\cal V}$. As a result, ${\cal V}_{0}\times \{0\}$ is contained
in ${\cal V}_{*}$ and so ${\cal V}={\cal V}_{*}$. 

(iii)  Since $A(0)$ is a polynomial algebra, $S$ is a free extension of $A(0)$ by (ii) 
and Proposition~\ref{pgf1}. Moreover, by Lemma~\ref{l2gf1}, the linear map
$$ \tk {\k}{V_{0}}A(0) \longrightarrow S, \qquad v\tens a \longmapsto va$$
is a homogenous isomorphism with respect to the grading of $\tk {\k}{V_{0}}A(0)$
induced by those of $V_{0}$ and $A(0)$. As a result, for all nonnegative integer $i$, 
$$ \dim S^{[i]} = \sum_{j=0}^{i} \dim V_{0}^{[i-j]} \mul \dim A(0)^{[j]},$$
whence $\dim V_{0}^{[i]}=\dim {V'_{0}}^{[i]}$ for all $i$ by Proposition~\ref{psg3},(iv)
since $\dim A^{[i]}=\dim A(0)^{[i]}$ for all $i$ by Proposition~\ref{psg1},(ii). Then
$V_{0}=V'_{0}$.
\end{proof}

As explained in Subsection \ref{ge2}, by Theorem~\ref{tge1} 
and Proposition~\ref{pgf1},(ii), Theorem~\ref{tge2} results from Theorem~\ref{tsg5},(ii).

\begin{rema} 
According to the part (ii) of Theorem \ref{tsg5}, ${\cal I}_{*}$ is the radical of 
$S[t]A_{+}$. Hence $S[t]A_{+}$ is radical by Lemma~\ref{l2sg2},(ii), and then 
${\cal I}_{*}=S[t]A_{+}$. 
\end{rema}

\section{Consequences of Theorem \ref{ti4} for the simple classical Lie algebras} 
\label{ca} 
This section concerns some applications of Theorem \ref{ti4} to the simple classical Lie 
algebras. 

\subsection{}\label{ca1}
The first consequence of Theorem~\ref{tge2} is the following.

\begin{theorem}  \label{tca1} 
Assume that $\g$ is simple of type {\bf A} or {\bf C}. Then all the elements of $\g$ are 
good. 
\end{theorem}

\begin{proof}
This follows from~\cite[Theorems 4.2 and 4.4]{PPY}, Theorem \ref{tge2} and 
Proposition~\ref{pge1}. Further, in type {\bf A}, the result is given 
by~\cite[Theorem 5.4]{PPY}.
\end{proof}

\subsection{}\label{ca2}
In this subsection and the next one, $\g$ is assumed to be simple of type ${\bf B}$ or 
${\bf D}$. More precisely, we assume that ${\goth g}$ is the simple Lie algebra 
${\goth {so}}({\Bbb V})$ for some vector space $\V$ of dimension 
$2\rg+1$ or $2\rg$. 
Then $\g$ is embedded into $\tilde{\g}:=\mathfrak{gl}(\V)={\rm End}(\V)$. 
For $x$ an endomorphism of ${\Bbb V}$ and for $i\in\{1,\ldots,\dim {\Bbb V}\}$, denote by
$Q_{i}(x)$ the coefficient of degree $\dim {\Bbb V}-i$ of the characteristic polynomial 
of $x$. Then, for any $x$ in ${\goth g}$, $Q_{i}(x)=0$ whenever $i$ is odd. 
Define a generating family $(\poi q1{,\ldots,}{\rg}{}{}{})$ of the algebra $\ai g{}{}$ as
follows. For $i=1,\ldots,\rg-1$, set $q_{i}:=Q_{2i}$. If $\dim {\Bbb V}=2\rg +1$, set 
$q_{\rg}=Q_{2\rg}$ and if $\dim {\Bbb V}=2\rg$, let $q_{\rg}$ be the Pfaffian that is a 
homogenous element of degree $\rg$ of $\ai g{}{}$ such that $Q_{2\rg}=q_{\rg}^{2}$. 

Let $(e,h,f)$ be an ${\goth {sl}}_{2}$-triple of ${\goth g}$. Following the notations of 
Subsection \ref{ge2}, for $i\in\{1,\ldots,\rg\}$,  
denote by ${\ie q_{i}}$ the initial homogenous component of 
the restriction to $\g^{f}$ of the polynomial 
function $x\mapsto q_{i}(e+x)$, 
and by $\delta_i$ the degree of ${\ie q_{i}}$. 
According to~\cite[Theorem~2.1]{PPY}, $\poi {\ie q}1{,\ldots,}{\rg}{}{}{}$ are 
algebraically independent if and only if 
$$\dim {\goth g}^{e}+\rg - 2(\poi {\delta }1{+\cdots +}{\rg}{}{}{}) = 0.$$
Our first aim in this subsection is to describe the sum $\dim {\goth g}^{e}+\rg - 
2(\poi {\delta }1{+\cdots +}{\rg}{}{}{})$ in term of the partition of $\dim \V$ 
associated with $e$.

\begin{rema}
The sequence of the degrees $(\poi {\delta }1{,\ldots ,}{\rg}{}{}{})$ is described by 
\cite[Remark 4.2]{PPY}. 
\end{rema}

For $\lambda=(\poi {\lambda }1{,\ldots ,}{k}{}{}{})$ a sequence of positive integers, 
with $\poi {\lambda }1{\geq \cdots \geq }{k}{}{}{}$, set: 
$$\vert \lambda  \vert := k, \qquad \qquad 
r(\lambda ) :=\poi {\lambda }1{+\cdots +}{k}{}{}{}.$$
Assume that the partition $\lambda$ of $r(\lambda )$ is associated with a nilpotent 
orbit of ${\goth {so}}(\k^{r(\lambda )})$. Then the even integers of $\lambda$ have an 
even multiplicity, \cite[\S5.1]{CMa}. Thus $k$ and $r(\lambda )$ have the same parity. 
Moreover, there is an involution $i\mapsto i'$ of 
$\{1,\ldots,k\}$ such that $i=i'$ if $\lambda_{i}$ is odd, and $i'\in \{i-1,i+1\}$ if 
$\lambda_{i}$ is even. Set:
$$S(\lambda ) := \sum_{i=i', \, i \textrm{ odd}} i - \sum_{i=i', \, i \textrm{ even}} i$$
and denote by $n_{\lambda }$ the number of even integers in the sequence $\lambda$. 

From now on, assume that $\lambda$ is the partition of $\dim {\Bbb V}$ 
associated with the nilpotent orbit $G.e$. 

\begin{lemma}\label{lca2}
{\rm (i)} If $\dim {\Bbb V}$ is odd, i.e., $k$ is odd, then 
$$ \dim {\goth g}^{e}+\rg - 2(\poi {\delta }1{+\cdots +}{\rg}{}{}{}) = 
\frac{n_{\lambda }-k-1}{2} + S(\lambda ).$$

{\rm (ii)} If $\dim {\Bbb V}$ is even, i.e., $k$ is even, then 
$$ \dim {\goth g}^{e}+\rg - 2(\poi {\delta }1{+\cdots +}{\rg}{}{}{}) = 
\frac{n_{\lambda }+k}{2} + S(\lambda ).$$
\end{lemma}

\begin{proof}
(i) If $\dim{\Bbb V}$ is odd, then by \cite[\S4.4, (14)]{PPY}, 
$$2(\poi {\delta }1{+\cdots +}{\rg}{}{}{})= \dim \g^{e} + \frac{\dim\V}{2}+
\frac{k-n_\lambda}{2}-S(\lambda),$$
whence
$$ \dim {\goth g}^{e}+\rg - 2(\poi {\delta }1{+\cdots +}{\rg}{}{}{}) = 
\frac{n_{\lambda }-k-1}{2} + S(\lambda )$$
since $\dim {\Bbb V}=2\rg+1$. 

(ii) If $\dim {\Bbb V}$ is even, then $\delta _{\rg}=k/2$ by \cite[Remark 4.2]{PPY} 
and by \cite[\S4.4, (14)]{PPY},  
$$2(\poi {\delta }1{+\cdots +}{\rg}{}{}{}) +k=\dim \g^{e} + \frac{\dim\V}{2}+
\frac{k-n_\lambda}{2}-S(\lambda)$$ 
whence
$$\dim{\goth g}^{e}+\rg-2(\poi {\delta }1{+\cdots +}{\rg}{}{}{})=
\frac{n_{\lambda }+k}{2}+S(\lambda )$$
since $\dim\V=2\rg$. 
\end{proof}

\noi
The sequence $\lambda=(\lambda_1,\ldots,\lambda_k)$ satisfies one of the following five 
conditions:  
\begin{itemize}
\item [{\rm 1)}] $\lambda_{k}$ and $\lambda_{k-1}$ are odd, 
\item [{\rm 2)}] $\lambda_{k}$ and $\lambda_{k-1}$ are even,  
\item [{\rm 3)}] $k>3$, $\lambda_{k}$ and $\lambda_{1}$ are odd and $\lambda_{i}$ is 
even for any $i\in\{2,\ldots,k-1\}$,
\item [{\rm 4)}] $k>4$, $\lambda_{k}$ is odd and there is $k'\in\{2,\ldots,k-2\}$ 
such that $\lambda_{k'}$ is odd and $\lambda_{i}$ is even for any 
$i\in\{k'+1,\ldots,k-1\}$,
\item [{\rm 5)}] $k=1$ or $\lambda_{k}$ is odd and $\lambda_{i}$ is even for any $i<k$.
\end{itemize}
For example, $(4,4,3,1)$ satisfies Condition (1);
$(6,6,5,4,4)$ satisfies Condition (2);
$(7,6,6,4,4,4,4,3)$ satisfies Condition (3);
$(8,8,7,5,4,4,2,2,3)$ satisfies Condition (4) with $k'=4$;
$(9)$ and $(6,6,4,4,3)$ satisfy Condition (5). 
If $k=2$, then one of the conditions (1) or (2) is satisfied. 

\begin{defi}\label{dca2}
Define a sequence $\lambda^{*}$ of positive integers, with 
$\vert \lambda^* \vert\leq \vert \lambda  \vert$,
 as follows: 
\begin{itemize}
\item[-] if $k=2$ or if Condition (3) or (5) is satisfied, then set $\lambda^{*}:=\lambda$,
\item[-] if Condition (1) or (2) is satisfied, then set: 
$$\lambda^{*} := (\poi {\lambda }1{,\ldots ,}{k-2}{}{}{}),$$ 
\item[-] if $k>3$ and if Condition (4) is satisfied, 
then set $$\lambda^{*} := (\poi {\lambda }1{,\ldots ,}{k'-1}{}{}{}).$$ 
\end{itemize}
\end{defi}

In any case, the partition of $r(\lambda^{*})$ corresponding to $\lambda^{*}$ is 
associated with a nilpotent orbit of ${\goth {so}}(\k^{r(\lambda^{*})})$. 
Recall that $n_\lambda$ is the number of even integers in the sequence $\lambda$.  

\begin{defi} \label{d2ca2}
Denote by $d_{\lambda }$ the integer defined by: 
\begin{itemize}
\item[-] if $k=2$, then $d_{\lambda }:=n_{\lambda }$,
\item[-] if $k>2$ and if Condition (1) or (4) is satisfied, 
then $d_{\lambda }:= d_{\lambda^{*}}$,
\item[-] if $k>2$ and if Condition (2) is satisfied,
then $d_{\lambda }:= d_{\lambda^{*}}+2$,
\item[-] if $k>2$ and if Condition (3) is satisfied, 
then $d_{\lambda }:= 0$,   
\item[-] if Condition (5) is satisfied, then $d_{\lambda }:= 0$.
\end{itemize}
\end{defi}

\begin{lemma}\label{l2ca2}
{\rm (i)} Assume that $k$ is odd. 
If Condition {\rm (1)}, {\rm (2)} or {\rm (5)} is satisfied, then
$$ \frac{n_{\lambda }-k-1}{2} + S(\lambda ) = 
\frac{n_{\lambda^{*}}-\vert \lambda^{*} \vert-1}{2} + S(\lambda^{*}).$$ 
Otherwise, 
$$\frac{n_{\lambda }-k-1}{2} + S(\lambda ) = 
\frac{n_{\lambda^{*}}-\vert \lambda^{*} \vert-1}{2} + S(\lambda^{*}) +
k-\vert \lambda^{*} \vert-2.$$ 

{\rm (ii)} If $k$ is even, then 
$$ \frac{n_{\lambda }+k}{2} + S(\lambda ) = 
\frac{n_{\lambda^{*}}+\vert \lambda^{*} \vert}{2} + S(\lambda^{*}) 
+ d_{\lambda }-d_{\lambda^{*}}.$$ 
\end{lemma}

\begin{proof}
(i) If Condition (3) or (5) is satisfied, there is nothing to prove. If Condition (1) is
satisfied, 
$$\begin{array}{ccc}
n_{\lambda }=n_{\lambda^{*}}, && S(\lambda )=S(\lambda^{*}) + 1.\end{array}$$ 
Then
$$ \frac{n_{\lambda }-k-1}{2} + S(\lambda ) = \frac{n_{\lambda^{*}}-\vert \lambda^{*} 
\vert-1}{2} - 1 +S(\lambda^{*})+1$$   
whence the assertion. If Condition (2) is satisfied, 
$$\begin{array}{ccc}
n_{\lambda }=n_{\lambda^{*}}+2, && S(\lambda )=S(\lambda^{*}). \end{array}$$ 
Then,
$$ \frac{n_{\lambda }-k-1}{2} + S(\lambda ) = \frac{n_{\lambda^{*}}-\vert \lambda^{*} 
\vert-1}{2}+ S(\lambda^{*})$$
whence the assertion. If Condition (4) is satisfied, 
$$\begin{array}{ccc}
n_{\lambda }=n_{\lambda^{*}}+k-\vert \lambda^{*} \vert -2,  && 
S(\lambda )=S(\lambda^{*})+k-\vert \lambda^{*} \vert -1 .\end{array}$$ 
Then,
$$\frac{n_{\lambda }-k-1}{2} + S(\lambda ) = \frac{n_{\lambda^{*}}-\vert \lambda^{*} 
\vert -1}{2}-1 + S(\lambda^{*}) +k- \vert \lambda^{*} \vert -1$$
whence the assertion. 

(ii) If $k=2$ or if $k>2$ and Condition (3) or (5) is satisfied, there is nothing to 
prove. Let us suppose that $k>3$. If Condition (1) is satisfied, 
$$\begin{array}{ccc}
n_{\lambda }=n_{\lambda^{*}}, && S(\lambda )=S(\lambda^{*})-1 .\end{array}$$ 
Then
$$ \frac{n_{\lambda }+k}{2} + S(\lambda ) = 
\frac{n_{\lambda^{*}}+\vert \lambda^{*} \vert}{2} 
+1+S(\lambda^{*})-1$$   
whence the assertion since $d_{\lambda }=d_{\lambda^{*}}$. If Condition (2) is satisfied, 
$$\begin{array}{ccc}
n_{\lambda }=n_{\lambda^{*}} +2, && S(\lambda )=S(\lambda^{*}) .\end{array}$$ 
Then,
$$ \frac{n_{\lambda }+k}{2} + S(\lambda ) = 
\frac{n_{\lambda^{*}}+\vert \lambda^{*} \vert}{2}+2+S(\lambda^{*})$$
whence the assertion since $d_{\lambda }-d_{\lambda^{*}}=2$. If Condition (4) is 
satisfied, 
$$\begin{array}{ccc}
n_{\lambda }=n_{\lambda^{*}} + k-\vert \lambda^{*} \vert - 2, && 
S(\lambda )=S(\lambda^{*}) + \vert \lambda^{*} \vert +1 -k .\end{array}$$ 
Then,
$$ \frac{n_{\lambda }+k}{2} + S(\lambda ) = 
\frac{n_{\lambda^{*}}+\vert \lambda^{*} \vert}{2}
+ k-\vert \lambda^{*} \vert -1 + S(\lambda^{*}) + \vert \lambda^{*} \vert - k +1 $$
whence the assertion since $d_{\lambda }=d_{\lambda^{*}}$.
\end{proof}

\begin{lemma}\label{l3ca2}
{\rm (i)} If $\lambda_{1}$ is odd and if $\lambda_{i}$ is even for $i\geq 2$, then 
$\dim {\goth g}^{e}+\rg - 2(\delta_1+\cdots+\delta_\rg)=0.$

{\rm (ii)} If $k$ is odd, then $\dim {\goth g}^{e}+\rg - 
2(\poi {\delta }1{+\cdots +}{\rg}{}{}{})= n_{\lambda }-d_\lambda.$ 

{\rm (iii)} If $k$ is even, then $\dim {\goth g}^{e}+\rg - 
2(\poi {\delta }1{+\cdots +}{\rg}{}{}{}) = d_\lambda.$
\end{lemma}

\begin{proof}
(i) By the hypothesis, $n_{\lambda }=k-1$ and $S(\lambda )=1$, whence the assertion
by Lemma \ref{lca2},(i).

(ii) Let us prove the assertion by induction on $k$. For $k=3$, if $\lambda_{1}$ and 
$\lambda_{2}$ are even, $n_{\lambda }=2$, $d_{\lambda }=0$ and $S(\lambda ) = 3$, whence 
the equality by Lemma \ref{lca2},(i). Assume that $k>3$ and suppose that 
the equality holds for the integers smaller than $k$. If Condition (1) or (2) is 
satisfied, then by Lemma \ref{lca2},(i), Lemma \ref{l2ca2},(i) and by induction 
hypothesis,
$$ \dim {\goth g}^{e}+\rg - 2(\poi {\delta }1{+\cdots +}{\rg}{}{}{}) =  
n_{\lambda^{*}}-d_{\lambda^{*}}.$$ 
But if Condition (1) or (2) is satisfied, then $n_{\lambda }-d_{\lambda }=n_{\lambda^{*}}
-d_{\lambda^{*}}$. If Condition (5) is satisfied, then 
$$\begin{array}{ccccccc}
n_{\lambda } = k-1, && S(\lambda ) = k, && d_{\lambda } = 0, \end{array}$$
whence the equality by Lemma \ref{lca2},(i). Let us suppose that Condition (4) is 
satisfied. By Lemma \ref{lca2},(i), Lemma \ref{l2ca2},(i) and by induction hypothesis,
$$\dim {\goth g}^{e}+\rg - 2(\poi {\delta }1{+\cdots +}{\rg}{}{}{})= 
n_{\lambda^{*}} - d_{\lambda^{*}} + k-\vert \lambda^{*} \vert -2 = 
n_{\lambda }-d_{\lambda } $$
whence the assertion since Condition (3) is never satisfied when $k$ is odd. 

(iii) The statement is clear for $k=2$ by Lemma \ref{lca2},(ii). Indeed, if Condition (1) 
is satisfied, then $d_\lambda=n_\lambda=0$ and $S(\lambda)=-1$ and if Condition (2) is 
satisfied, then $d_\lambda=n_\lambda=2$ and $S(\lambda)=0$. If Condition (3) is satisfied,
$n_{\lambda }=k-2$ and $S(\lambda )=1-k$, whence the statement by Lemma~\ref{lca2},(ii).
When Condition (4) is satisfied, by induction on $\vert \lambda \vert$, the statement 
results from Lemma \ref{lca2},(ii) and Lemma \ref{l2ca2},(ii), whence the assertion since
Condition (5) is never satisfied when $k$ is even.    
\end{proof}

\begin{coro}\label{cca2}
{\rm (i)} If $\lambda_{1}$ is odd and if $\lambda_{i}$ is even for all $i\geq 2$, then 
$e$ is good. 

{\rm (ii)} If $k$ is odd and if $n_{\lambda }=d_{\lambda }$, then $e$ is good. In 
particular, if $\g$ is of type ${\bf B}$, then the even nilpotent elements of ${\goth g}$ 
are good.

{\rm (iii)} If $k$ is even and if $d_\lambda=0$, then $e$ is good. 
In particular, if ${\goth g}$ is of type ${\bf {D}}$ and of odd rank, then the even 
nilpotent elements of ${\goth g}$ are good.
\end{coro}

\begin{proof}
As it has been already noticed, by \cite[Theorem~2.1]{PPY}, the polynomials 
$\poi {\ie q}1{,\ldots ,}{\rg}{}{}{}$ are algebraically independent if and only if 
$$\dim {\goth g}^{e}+\rg - 2(\poi {\delta }1{+\cdots +}{\rg}{}{}{}) = 0.$$
So, by Theorem \ref{tge2} and Lemma \ref{l3ca2}, if either $\lambda_{1}$ is odd and 
$\lambda_{i}$ is even for all $i\geq 2$, or if $k$ is odd and $n_{\lambda }=d_\lambda$, 
or if $k$ is even and $d_{\lambda }=0$, then $e$ is good. 

Suppose that $e$ is even. Then the integers $\poi {\lambda }1{,\ldots ,}{k}{}{}{}$ have 
the same parity, cf.~e.g.~\cite[\S1.3.1]{Ba}. Moreover, $n_{\lambda }=d_{\lambda }=0$ 
whenever $\poi {\lambda }1{,\ldots ,}{k}{}{}{}$ are all odd (cf.~Definition~\ref{d2ca2}).
This in particular occurs if either ${\goth g}$ is of type ${\bf {B}}$, or if ${\goth g}$
is of type ${\bf {D}}$ with odd rank. 
\end{proof}

\begin{rema}\label{rca2}
The fact that the even nilpotent elements of ${\goth g}$ without (only) even Jordan 
blocks are good if ${\goth g}$ is of type {\bf B} or  {\bf D} was already 
observed by O.~Yakimova in \cite[Corollary 8.2]{Y4} in a different formulation. Corollary
\ref{cca2} is more general. 
\end{rema}

\begin{defi}\label{d3ca2}
A sequence $\lambda=(\poi {\lambda }1{,\ldots ,}{k}{}{}{})$ is said to be 
{\em very good} if $n_\lambda=d_\lambda$ whenever $k$ is odd and if $d_\lambda=0$ 
whenever $k$ is even. A nilpotent element of $\g$ is said to be 
{\em very good} if it is associated with a very good partition of $\dim\V$. 
\end{defi}
According to Corollary~\ref{cca2}, if $e$ is very good then $e$ is good. 
The following lemma characterizes the very good sequences. 

\begin{lemma}\label{l4ca2}
{\rm (i)} If $k$ is odd then $\lambda$ is very good if and only if $\lambda_1$ is 
odd and if  $(\poi {\lambda }2{,\ldots ,}{k}{}{}{})$ 
is a concatenation of sequences satisfying Conditions~{\rm (1)} 
or~{\rm (2)}  with $k=2$.

{\rm (ii)} If $k$ is even then $\lambda$ is very good if and only if 
$\lambda $ is a concatenation of sequences satisfying Condition~{\rm (3)} or 
Condition~{\rm (1)} with $k=2$.
\end{lemma}

For example, the partitions $(5,3,3,2,2)$ and $(7,5,5,4,4,3,1,1)$ of $15$ 
and $30$ respectively are very good.  

\begin{proof}
(i) Assume that $\lambda_1$ is odd and that  $(\lambda_{2},\ldots,\lambda_k)$ 
is a concatenation of sequences satisfying Conditions~(1) or~(2) with $k=2$. 
So, if $k>1$, then $n_\lambda-d_\lambda=n_{\lambda^*}-d_{\lambda^*}$. 
Then, a quick induction proves that 
$n_\lambda-d_\lambda=n_{(\lambda_1)}-d_{(\lambda_1)}=0$ since $\lambda_1$ is odd. 
The statement is clear for $k=1$. 

Conversely, assume that $n_{\lambda}-d_{\lambda}=0$. If $\lambda$ satisfies Conditions (1)
or (2), then $n_\lambda-d_\lambda=n_{\lambda^*}-d_{\lambda^*}$ and 
$\vert \lambda^* \vert < \vert \lambda \vert$. So, we can assume that $\lambda$ does not
satisfy Conditions (1) or (2). Since $k$ is odd, $\lambda$ cannot satisfy Condition (3). 
If $\lambda$ satisfies Condition (4), then 
$n_\lambda-d_\lambda=n_\lambda-d_{\lambda^*} > n_{\lambda^*}-d_{\lambda^*}\geq 0$. 
This is impossible since $n_\lambda-d_\lambda=0$. 
If $\lambda$ satisfies Condition (5), then $n_\lambda-d_\lambda=n_\lambda$.   
So, $n_\lambda-d_\lambda=0$ if and only if $k=1$. 
Thereby, the direct implication is proven. 

(ii) Assume that $\lambda $ is a concatenation of sequences satisfying Condition~(3) or 
Condition~(1) with $k=2$. In particular, $\lambda $ does not satisfy Condition (2). 
Moreover, Condition (5) is not satisfied since $k$ is even. Then $d_{\lambda }=0$ 
by induction on $\vert \lambda  \vert$, whence $e$ is very good. 

Conversely, suppose that $d_{\lambda }=0$. If $k=2$, Condition (1) is satisfied and 
if $k=4$, then either Condition (3) is satisfied, or $\poi {\lambda }1{,\ldots,}{4}{}{}{}$
are all odd. Suppose $k>4$. Condition (2) is not satisfied since 
$d_{\lambda }=d_{\lambda _{*}}+2$ in this case. If Condition (1) is satisfied then 
$d_{\lambda _{*}}=0$ and $\lambda $ is a concatenation of $\lambda ^{*}$ and 
$(\lambda _{k-1},\lambda _{k})$. If Condition (4) is satisfied, then $d_{\lambda _{*}}=0$ 
and $\lambda $ is a concatenation of $\lambda _{*}$ and a sequence satisfying 
Condition (3), whence the assertion by induction on $\vert \lambda \vert$ since Condition
(5) is not satisfied when $k$ is even.
\end{proof}

\subsection{}\label{ca3}
Assume in this subsection that $\lambda=(\poi {\lambda }1{,\ldots ,}{k}{}{}{})$ satisfies 
the following condition: 

\med

\begin{tabular}{ccl}
$(\ast)$&&{\em For some $k'\in\{2,\ldots,k\}$, $\lambda_i$ is even for all 
$i\leq k'$,  and $(\lambda_{k'+1},\ldots,\lambda_k)$} \\ && {\em is very good.}
\end{tabular} 

\med

\noi
In particular, $k'$ is even and by Lemma \ref{l4ca2}, $\lambda _{k'+1}$ is odd and 
$\lambda$ is not very good. For example, the sequences $\lambda=(6,6,4,4,3,2,2)$ and 
$(6,6,4,4,3,3,3,2,2,1)$ satisfy the condition $(\ast)$ with $k'=4$. Define a sequence 
$\nu=(\poi {\nu }1{,\ldots ,}{k}{}{}{})$ of integers of $\{1,\ldots,\rg\}$ by 
$$\forall \, i\in\{1,\ldots,k'\},\qquad 
\nu_i:=\frac{\lambda_1+\cdots+\lambda_i}{2}.$$
If $k'=k$, then 
$\nu_k=(\poi {\lambda }1{+\cdots +}{k}{}{}{})/2=r(\lambda)/2=\rg$.
Define elements $p_1,\ldots,p_{k'}$ of S$(\g^{e})$ as follows: 
\begin{itemize}
\item[-] if $k'<k$, set for $i\in\{1,\ldots,k'\}$, $p_i:={\ie {q_{\nu_i}}}$, 

\item[-] if $k'=k$, set for $i\in\{1,\ldots,k'-1\}$, $p_i:={\ie {q_{\nu_i}}}$ 
and set $p_k:=({\ie {q_{\nu_k}}})^2$. 
In this case, set also $\tilde{p}_k:={\ie {q_{\nu_k}}}.$
\end{itemize}

Remind that $\delta_i$ is the degree of $\ie {q_i}$ for $i=1,\ldots,\rg$. 
The following lemma is a straightforward consequence of \cite[Remark 4.2]{PPY}: 

\begin{lemma}\label{lca3}
{\rm (i)} For all $i\in\{1,\ldots,k'\}$, $\deg\,p_i=i$.

{\rm (ii)} Set $\nu_0:=0$. 
Then for $i \in\{1,\ldots,k'\}$ and $r\in\{1,\ldots,\nu_{k'}-1\}$, 
$$\delta_r=i \iff \nu_{i-1} < r \leq \nu_i .$$
In particular, for $r\in\{1,\ldots,\nu_{k'}-2\}$,  
$\delta_{r}<\delta_{r+1}$ if and only if $r$ is a value of the sequence 
$\nu$. 
\end{lemma}

\begin{example}\label{eca3}
Consider the partition $\lambda=(8,8,4,4,4,4,2,2,1,1)$ of $38$. Then $k=10$, $k'=8$ and 
$\nu=(4,8,10,12,14,16,17,18)$.  We represent in Table~\ref{fig:1} the degrees of 
the polynomials $p_1,\ldots,p_8$ and ${\ie {q_{1}}},\ldots,{\ie {q_{18}}}$. 
Note that $\deg {\ie {q_{19}}}=5$. In the 
table, the common degree of the polynomials appearing on the $i$th column is $i$. 

\begin{table}[ht!]
\begin{center}
{\tiny \begin{tabular}{l c c c c c c c c c}
&${\ie {q_{4}}}\!=\!p_1$&${\ie {q_{8}}}\!=\!p_2$& & & & & & &\\\\
&${\ie {q_{3}}}$& ${\ie {q_{7}}}$ & & & & & & &\\\\
&${\ie {q_{2}}}$&${\ie {q_{6}}}$&${\ie {q_{10}}}\!=\!p_3$&${\ie {q_{12}}}\!=\!p_4$&
${\ie {q_{14}}}\!=\!p_5$&${\ie {q_{16}}}\!=\!p_6$& &\\\\
&${\ie {q_{1}}}$&${\ie {q_{5}}}$&${\ie {q_{9}}}$&${\ie {q_{11}}}$&
${\ie {q_{13}}}$&${\ie {q_{15}}}$&${\ie {q_{17}}}\!=\!p_7$&${\ie {q_{18}}}\!=\!p_8$\\\\
\hline\\
degrees&1&2&3&4&5&6&7&8\\
\small
\end{tabular}}
\end{center}
\caption{~}\label{fig:1}
\end{table}
\end{example}
Let $\goth s$ be the subalgebra of $\g$ generated by $e,h,f$ and decompose 
$\V$ into simple $\goth s$-modules $\V_1,\ldots,\V_k$ of dimension $\lambda_1,
\ldots,\lambda_k$ respectively. One can order them so that for $i\in\{1,\ldots,k'/2\}$, 
$\V_{(2(i-1)+1)'}=\V_{2i}$. For $i\in\{1,\ldots,k\}$, denote by $e_{i}$ the restriction 
to $\V_i$ of $e$ and set $\varepsilon_i:=e_i^{\lambda_i-1}$. Then $e_i$ is a regular 
nilpotent element of $\mathfrak{gl}(\V_i)$ and 
$(\ad h)\varepsilon_i=2(\lambda_i-1)\varepsilon _{i}$.

For $i\in\{1,\ldots,k'/2\}$, set 
$$\V[i]:=\V_{2(i-1)+1}+\V_{2i}$$
and set 
$$\V[0]:=\V_{k'+1}\oplus\cdots\oplus\V_k.$$ 
Then for $i\in\{0,1,\ldots,k'/2\}$, denote by $\g_i$ the simple Lie algebra 
${\goth {so}}(\V[i])$. For $i\in\{1,\ldots,k'/2\}$, $e_{2(i-1)+1}+e_{2i}$ is an even 
nilpotent element of $\g_i$ with Jordan blocks of size 
$(\lambda_{2(i-1)+1},\lambda_{2i})$. Let $i\in\{1,\ldots,k'/2\}$ and set: 
$$z_i:=\varepsilon_{2(i-1)+1}+\varepsilon_{2i}.$$
Then $z_i$ lies in the center of $\g^{e}$ and 
$$(\ad h)z_i=2(\lambda_{2(i-1)+1}-1)z_{i}=2(\lambda_{2i}-1)z_{i} .$$  
Moreover, $2(\lambda_{2i}-1)$ is the highest weight of $\ad h$ acting on 
$\g_i^{e}:=\g_i\cap\g^{e}$, and the intersection of the $2(\lambda_{2i}-1)$-eigenspace of
$\ad h$ with $\g_i^{e}$ is spanned by $z_i$, see for instance~\cite[\S1]{Y4}. Set
\begin{eqnarray*}
\overline{\g}&:=&\g_0\oplus \g_1\oplus\cdots\oplus\g_{k'/2}
\;=\;{\goth {so}}(\V[0])\oplus{\goth {so}}(\V[1])\oplus\cdots\oplus{\goth {so}}(\V[k'/2] )
\end{eqnarray*}
and denote by $\overline{\g}^{e}$ (resp.~$\overline{\g}^{f}$) the centralizer of $e$ 
(resp.~$f$) in $\overline{\g}$. For $p\in \es S{\g^{e}}$, denote by $\overline{p}$ its 
restriction to $\overline{\g}^{f}\simeq(\overline{\g}^{e})^*$; it is an element of 
$\es S{\overline{\g}^{e}}$. Our goal is to describe the elements 
$\overline{p}_1,\ldots,\overline{p}_{k'}$ (see Proposition \ref{pca3}). 
The motivation comes from Lemma~\ref{l2ca3}.

Let $\g_{\rm reg}^{f}$ (resp.~$\overline{\g}_{\rm reg}^{f}$) be 
the set of elements $x\in\g^{f}$ (resp.~$\overline{\g}^f$) such that $x$ is a regular 
linear form on $\g^{e}$ (resp.~$\overline{\g}^{e}$). 

\begin{lemma}\label{l2ca3} 
{\rm (i)} The intersection $\g_{\rm reg}^{f}\cap\overline{\g}^{f}$ is a dense open subset
of $\overline{\g}_{{\rm reg}}^{f}$. 

{\rm (ii)} The morphism 
$$\theta: \quad G^{e}_{0} \times \overline{\g}^{f} \longrightarrow \g^{f}, 
\quad (g,x) \longmapsto g.x$$
is a dominant morphism from $G^{e}_{0}\times \overline{\g}^{f}$ to ${\goth g}^{f}$.
\end{lemma}

\begin{proof}
(i) Since $\lambda$ satisfies the condition $(\ast)$, it satisfies the condition (1) 
of the proof of ~\cite[\S4, Lemma~3]{Y1} and so, $\g_{\rm reg}^{f}\cap\overline{\g}^{f}$ 
is a dense open subset of $\overline{\g}^{f}$. Moreover, since $\g^{e}$ and 
$\overline{\g}^{e}$ have the same index by \cite[Theorem 3]{Y1}, 
$\g_{\rm reg}^{f}\cap\overline{\g}^{f}$ is contained in $\overline{\g}_{{\rm reg}}^{f}$. 

(ii) Let ${\goth m}$ be the orthogonal complement to $\overline{{\goth g}}$ in 
${\goth g}$ with respect to the Killing form $\dv ..$. Since the restriction to 
$\overline{\g}\times\overline{\g}$ of $\dv ..$ is nondegenerate, 
$\g=\overline{\g}\oplus\mathfrak{m}$ and 
$[\overline{\g},\mathfrak{m}]\subset\mathfrak{m}$. Set 
$\mathfrak{m}^{e}:=\mathfrak{m}\cap\g^{e}$. 
Since the restriction to $\overline{\g}^{f}\times\overline{\g}^{e}$ of $\dv ..$ is 
nondegenerate, we get the decomposition 
$$\g^{e}=\overline{\g}^{e}\oplus\mathfrak{m}^{e}$$ 
and $\mathfrak{m}^{e}$ is the orthogonal complement to $\overline{\g}^{f}$ in $\g^{e}$. 
Moreover, $[\overline{\g}^{e},\mathfrak{m}^{e}]\subset\mathfrak{m}^{e}$. 

By (i), $\g_{\rm reg}^{f}\cap\overline{\g}^{f}\not=\varnothing$. 
Let $x\in \g_{\rm reg}^{f}\cap\overline{\g}^{f}$. The tangent map at $(1_{{\goth g}},x)$ 
of $\theta$ is the linear map
$$\begin{array}{ccc}
{\goth g}^{e}\times \overline{\g}^{f} \longrightarrow {\goth g}^{f}, && 
(u,y) \longmapsto u.x + y, \end{array}$$
where $u.$ denotes the coadjoint action of $u$ on $\g^{f}\simeq(\g^{e})^*$. The index of 
$\overline{\g}^{e}$ is equal to the index of ${\g}^{e}$ and $[\overline{\g}^{e},
\mathfrak{m}^{e}]\subset\mathfrak{m}^{e}$. So, the stabilizer of $x$ in 
$\overline{\g}^{e}$ coincides with the stabilizer of $x$ in ${\g}^{e}$. In particular, 
$\dim\mathfrak{m}^{e}.x=\dim\mathfrak{m}^{e}$. As a result, $\theta$ is a submersion at 
$(1_{{\goth g}},x)$ since 
$\dim {\goth g}^{f}=\dim \mathfrak{m}^{e}+\dim \overline{\g}^{f}$. 
In conclusion, $\theta$ is a dominant morphism from $G^{e}_{0}\times \overline{\g}^{f}$ 
to ${\goth g}^{f}$.
\end{proof}

Let $(\mu_1,\ldots,\mu_m)$ be the strictly decreasing sequence of the values of the 
sequence $(\poi {\lambda }1{,\ldots ,}{k'}{}{}{})$ and let $k_{1},\ldots,k_m$ be the 
multiplicity of $\poi {\mu }1{,\ldots ,}{m}{}{}{}$ respectively in this sequence. By our 
assumption, the integers $\mu_1,\ldots,\mu_m,k_{1},\ldots,k_m$ are all even. Notice that 
$\poi {k}1{+\cdots +}{m}{}{}{}=k'$. The set $\{1,\ldots,k'\}$ decomposes into parts 
$\poi {K}1{,\ldots ,}{m}{}{}{}$ of cardinality $k_1,\ldots,k_m$ respectively given by:   
$$\forall\, s\in\{1,\ldots,m\},\qquad 
K_s:=\{\poi k0{+\cdots +}{s-1}{}{}{}+1,\ldots,k_0+\cdots+k_{s}\}.$$
Here, the convention is that $k_0:=0$. 

\begin{rema} \label{rca3}
For $s\in\{1,\ldots,m\}$ and $i\in K_s$,
$$\nu_{i}:= k_0 (\frac{\mu_0}{2})+\cdots +k_{s-1} (\frac{\mu_{s-1}}{2})
+j (\frac{\mu_s}{2}),$$
where $j=i-(k_0+\cdots+k_{s-1})$ and $\mu _{0}=0$.
\end{rema}

Decompose also the set $\{1,\ldots,k'/2\}$ into parts $\poi I1{,\ldots,}{m}{}{}{}$ of 
cardinality $k_1/2,\ldots,k_m/2$ respectively, with 
$$\forall\, s\in\{1,\ldots,m\},\qquad 
I_s:=\{\frac{k_0+\cdots+k_{s-1}}{2}+1,\ldots,\frac{k_0+\cdots+k_{s}}{2}\}.$$
For $p\in {\rm S}(\g^{e})$ an eigenvector of $\ad h$,  denote by ${\rm wt}(p)$ its 
$\ad h$-weight.

\begin{lemma}\label{l3ca3}
Let $s \in\{1,\ldots,m\}$ and $i\in K_s$.

{\rm (i)} Set $j=i-(k_0+\cdots+k_{s-1})$. Then, 
$${\rm wt}(\overline{p}_i)=2(2\nu_i-i)=\sum_{l=1}^{s-1} 2k_l(\mu_l -1)+2j(\mu_s-1).$$
Moreover, if $p\in\{{\ie {q_1}},\ldots,{\ie {q_{\rg-1}}},({\ie {q_{\rg}}})^2\}$ is of 
degree $i$, then 
${\rm wt}(p)={\rm wt}(\overline{p}) \leq 2(2 \nu_i-i)$ and the equality holds if and only
if $p=p_i$. 

{\rm (ii)} The polynomial $\overline{p}_i$ is in 
$\k[z_{l},\; l\in I_1\cup\ldots \cup I_s]$. 
\end{lemma}

\begin{proof}
(i) This is a consequence of \cite[Lemma 4.3]{PPY}  (or \cite[Theorem~6.1]{Y4}), 
Lemma \ref{lca3} and Remark~\ref{rca3}. 

(ii) Let $ \tilde{\g}^{f}$ be the centralizer of $f$ in $\tilde{\g}=\mathfrak{gl}(\V)$, 
and let ${\ie {\overline{Q}_{2\nu_i}}}$ be the initial homogenous component of the 
restriction to 
$$\left(\mathfrak{gl}(\V[0])\oplus\mathfrak{gl}(\V[1])\oplus\cdots \oplus
\mathfrak{gl}(\V[k'/2])\right)\,\cap\, \tilde{\g}^{f}$$ 
of the polynomial function $x\mapsto Q_{2\nu_i}(e+x)$. Since $\overline{p}_i\not= 0$, 
$\overline{p}_i$ is the restriction to $\overline{{\goth g}}^{f}$ of 
${\ie {\overline{Q}_{2\nu_i}}}$ and we have  
$${\rm wt}({\ie {\overline{Q}_{2\nu_i}}})={\rm wt}(\overline{p}_i)=2(2\nu_i-i),
\qquad \deg {\ie {\overline{Q}_{2\nu_i}}}=\deg \overline{p}_i=i.$$
Then, by (i) and \cite[Lemma 4.3]{PPY}, ${\ie {\overline{Q}_{2\nu_i}}}$ is a sum of 
monomials whose restriction to $\overline{{\goth g}}^{f}$ is zero and of monomials of the
form 
$$(\varepsilon_{\varsigma^{(1)} 1}\ldots \varepsilon_{\varsigma^{(1)} k_1}) \cdots
(\varepsilon_{\varsigma^{(s-1)} 1 } \ldots \varepsilon_{\varsigma^{(s-1)} k_{s-1}})
(\varepsilon_{\varsigma^{(s)} j_1}\ldots \varepsilon_{\varsigma^{(s)} j_i})$$
where $\poi j1{<\cdots <}{i}{}{}{}$ are integers of $K_s$, and 
$\varsigma^{(1)},\ldots,\varsigma^{(s-1)}$, $\varsigma^{(s)}$ 
are permutations of $K_1,\ldots,K_{s-1}$, $\{j_1,\ldots,j_i\}$ respectively. Hence, 
$\overline{p}_i$ is in $\k[z_l,\; l\in I_1\cup\ldots \cup I_s]$. More precisely, for 
$l\in I_1\cup\ldots \cup I_s$, the element $z_l$ appears in $\overline{p}_i$ with a 
multiplicity at most $2$ since $z_l=\varepsilon_{2(l-1)+1}+\varepsilon_{2l}$. 
\end{proof}

Let $s\in\{1,\ldots,m\}$ and $i\in K_s$. In view of Lemma \ref{l3ca3},(ii), we aim to give
an explicit formula for $\overline{p}_i$ in term of the elements $z_1,\ldots,z_{k'/2}$. 
Besides, according to Lemma~\ref{l3ca3},(ii), we can assume that $s=m$. As a first step, 
we state inductive formulae. If $k'>2$, set 
$$\overline{\g}':={\goth {so}}(\V[1])\oplus\cdots\oplus {\goth {so}}(\V[k'/2-1]),$$ 
and let $\overline{p}'_1,\ldots,\overline{p}'_{k'}$ be the restrictions to 
$(\overline{\g}')^{f}:=\overline{\g}'\cap\g^{f}$ of 
$\overline{p}_1,\ldots,\overline{p}_{k'}$ respectively. Note 
that $\overline{p}'_{k'-1}=\overline{p}'_{k'}=0$. 
Set by convention $k_0:=0$, $p_0:=1$, $p'_0:=1$ and $p_{-1}:=0$. It will be also 
convenient to set 
$$k^*:=k_0+\cdots+k_{m-1}.$$ 

\begin{lemma} \label{l4ca3} 
{\rm (i)} If $k_m=2$, then  
$$\overline{p}_{k^*+1}=-2 \,\overline{p}'_{k^*}\, z_{k'/2}\quad \text{ and }
\quad\overline{p}_{k^*+2}= \overline{p}'_{k^*}\,(z_{k'/2})^2.$$

\small

{\rm (ii)} If $k_m>2$, then 
$$\overline{p}_{k^*+1}=\overline{p}'_{k^*+1}-
2\, \overline{p}'_{k^*}\,z_{k'/2}$$
and for $j=2,\ldots,k_m$, 
$$\overline{p}_{k^*+j}=\overline{p}'_{k^*+j}-
2\, \overline{p}'_{k^*+j-1}\,z_{k'/2}+\overline{p}'_{k^*+j-2}\,(z_{k'/2})^2.$$
\end{lemma}

\begin{proof}
For $i=1,\ldots,k'/2$, let $w_i$ be the element of $\g_i^{f}:=\g_i\cap \g^{f}$ such that 
$$(\ad h)w_i=-2(\lambda_{2i}-1)w_{i} \quad \textrm{ and }\quad \det (e_i+w_i)=1.$$ 
Remind that $p_i(y)$, for $y\in\g^{f}$, is the initial homogenous component of the 
coefficient of the term $T^{\dim\V-2\nu_i}$ in the expression $\det(T-e-y)$. By Lemma 
\ref{l3ca3},(ii), in order to describe $\overline{p}_i$, it suffices to compute 
$\det(T-e-s_1 w_1-\cdots - s_{k'/2}w_{k'/2})$, with $\poi s1{,\ldots,}{k'/2}{}{}{}$ in 
$\k$.

\small

\noi
1) To start with, consider the case $k'=k_m=2$. By Lemma \ref{l3ca3}, 
$p_1=a z_1$ and $p_2= b z_1^2$ for some $a,b\in\k$. One has,  
$$\det(T-e-s_1w_1)=T^{2\mu_1}-2s_1 T^{\mu_1}+s_1^2.$$
As a result, $a=-2$ and $b=1$. This proves (i) in this case. 

\small

\noi
2) Assume from now that $k'>2$. Setting $e':=e_1+\cdots+e_{k'/2-1}$, observe that 
\begin{eqnarray}\label{eqca3}
&&\det(T-e-s_1 w_1-\cdots - s_{k'/2}w_{k'/2})\\\nonumber
&&\hspace{2cm}= \det(T-e'-s_1 w_1-\cdots - s_{k'/2-1}w_{k'/2-1})\, 
\det(T-e_{k'/2}- s_{k'/2}w_{k'/2})\\
\nonumber
&&\hspace{2cm}=\det(T-e'-s_1 w_1-\cdots - s_{k'/2-1}w_{k'/2-1})\, 
(T^{2 \mu_{m}}-2  s_{k'/2} T^{\mu_m} + s_{k'/2} ^2)
\end{eqnarray}
where the latter equality results from Step (1). 

\small
 
(i) If $k_m=2$, then $k^*=k'-2$ and the constant term in 
$\det(T-e'-s_1 w_1-\cdots -s_{k'/2-1}w_{k'/2-1})$ is $\overline{p}'_{k^*}$. By 
Lemma~\ref{l3ca3},(i), 
$${\rm wt}(\overline{p}_{k^*+1}) = {\rm wt}(\overline{p}'_{k^*}) + {\rm wt}(z_{k'/2})$$ 
and $\overline{p}'_{k^*}$ is the only element appearing in the coefficients of 
$\det(T-e'-s_1 w_1-\cdots - s_{k'/2-1}w_{k'/2-1})$ 
of this weight. Similarly, 
$${\rm wt}(\overline{p}_{k^*+2}) = 
{\rm wt}(\overline{p}'_{k^*}) +{\rm wt}((z_{k'/2})^2)$$ 
and $\overline{p}'_{k^*}$ is the only element appearing in the coefficients of 
$\det(T-e'-s_1 w_1-\cdots - s_{k'/2-1}w_{k'/2-1})$ of this weight.  
As a consequence, the equalities follow. 

\small

(ii) Suppose $k_m>2$. Then by Lemma \ref{l3ca3},(i), 
$${\rm wt}(\overline{p}_{k^*+1})={\rm wt}(\overline{p}'_{k^*+1})
={\rm wt}(\overline{p}'_{k^*}) +{\rm wt}(z_{k'/2}).$$
Moreover, $\overline{p}'_{k^*+1}$ and $\overline{p}'_{k^*}$ are the 
only elements appearing in the coefficients of 
$\det(T-e'-s_1 w_1-\cdots - s_{k'/2-1}w_{k'/2-1})$ of this weight with degree $k^*+1$ and
$k^*$ respectively. Similarly, by Lemma \ref{l3ca3},(i), for $j \in\{2,\ldots,k_m\}$, 
$${\rm wt}(\overline{p}_{k^*+j})={\rm wt}(\overline{p}'_{k^*+j})
={\rm wt}(\overline{p}'_{k^*+j-1}) +{\rm wt}(z_{k'/2})
={\rm wt}(\overline{p}'_{k^*+j-2})+{\rm wt}((z_{k'/2})^2).$$ 
Moreover, $\overline{p}'_{k^*+j}$, $\overline{p}'_{k^*+j-1}$ and 
$\overline{p}'_{k^*+j-2}$ are the 
only elements appearing in the coefficients of 
$\det(T-e'-s_1 w_1-\cdots - s_{k'/2-1}w_{k'/2-1})$  
of this weight with degree $k^*+j$, $k^*+j-1$ and $k^*+j-2$ 
respectively. 

In both cases, this forces the inductive formula (ii) through the 
factorization (\ref{eqca3}).
\end{proof}

For a subset $I=\{\poi i1{,\ldots,}{l}{}{}{}\} \subseteq \{1,\ldots,k'/2\}$ of 
cardinality $l$, denote by $\poi {\sigma }{I,1}{,\ldots,}{I,l}{}{}{}$ the elementary 
symmetric functions of $\poi z{i_{1}}{,\ldots,}{i_{l}}{}{}{}$:
$$\forall \, j \in\{1,\ldots,l\}, \qquad 
\sigma_{I,j}=\sum_{1 \leq a_1 < a_2<\cdots < a_{j} \leq l} 
z_{i_{a_1}}z_{i_{a_2}}\ldots z_{i_{a_j}}.$$
Set also $\sigma_{I,0}:=1$ and $\sigma_{I,j}:=0$ if $j>l$ so that $\sigma_{I,j}$ is well 
defined for any nonnegative integer $j$. Set at last $\sigma_{I,j}:=1$ for any $j$ if 
$I=\varnothing$. If $I=I_s$, with $s\in\{1,\ldots,m\}$, denote by $\sigma_j^{(s)}$, for 
$j \geq 0$, the elementary symmetric function $\sigma_{I_s,j}$. 

\begin{prop}\label{pca3}
Let $s\in\{1,\ldots,m\}$ and $j\in\{1,\ldots,k_s\}$. Then 
$$\overline{p}_{k_0+\cdots+k_{s-1}+j}\,=\,(-1)^{j} \, \overline{p}_{k_0+\cdots+k_{s-1}} 
\,\sum_{r=0}^{j} \sigma_{j-r}^{(s)}\sigma_r^{(s)}
\,=\,(-1)^{j}\,  (\sigma_{k_0/2}^{(1)}\ldots \sigma_{k_{s-1}/2}^{(s-1)})^2 \,
\sum_{r=0}^{j} \sigma_{j-r}^{(s)}\sigma_r^{(s)}.$$
\end{prop}

\begin{example}\label{e2ca3}
If $m=1$, then $k'=k_1$ and 
$$p_1=-\sigma_1^{(1)}\sigma_0^{(1)}-\sigma_0^{(1)}\sigma_1^{(1)}=-2\sigma_1^{(1)}
=-2(z_1+\cdots+z_{k'/2}),$$ 
$$p_2=\sigma_2^{(1)}\sigma_0^{(1)}+(\sigma_1^{(1)})^2+\sigma_0^{(1)}\sigma_2^{(1)}
=2\sigma_2^{(1)}+(\sigma_1^{(1)})^2,$$ $$\cdots,$$
$$\overline{p}_{k'}=(\sigma_{k'/2}^{(1)})^2=(z_1z_2\ldots z_{k'/2})^2.$$
\end{example}

\begin{proof} 
By Lemma \ref{l3ca3},(ii), we can assume that $s=m$. Retain the notations 
of Lemma~\ref{l4ca3}. In particular, set again 
$$k^*:=k_0+\cdots+k_{m-1}.$$ 

We prove the statement by induction on $k'/2$.  
If $k'=2$, then $m=1$, $k_m=k'=2$ and the statement 
follows from by Lemma~\ref{l4ca3},(i). 
Assume now that $k'>2$ and the statement true for the  
polynomials $\overline{p}'_1,\ldots,\overline{p}'_{k'-1}$. 

If $k_m=2$, the statement follows from Lemma~\ref{l4ca3},(i). 

Assume $k_m>2$. 
For any $r \geq 0$, we set $\sigma'_r:=\sigma_{I',r}$ where 
$I'=\{\frac{k^*}{2}+1,\ldots,\frac{k'}{2}-1\}\subset I_m$.  In particular, 
$\sigma'_0=1$ by convention. Observe that for any $r\geq 1$, 
$$\sigma_r^{(m)}=\sigma'_r +\sigma'_{r-1} z_{k'/2}.$$
Setting $\sigma'_{-1}:=0$, the above equality remains true for $r=0$. 
By the induction hypothesis and by 
Lemma~\ref{l4ca3},(ii), for $j\in\{2,\ldots,k_m\}$, 
\begin{eqnarray*} 
\overline{p}_{k^*+j} &=& \overline{p}'_{k^*+j} -2\,\overline{p}'_{k^*+j-1}\,z_{k'/2}
+\overline{p}'_{k^*+j-2}\,(z_{k'/2})^2 \quad\\
&=& \overline{p}_{k^*} \big( (-1)^{j}
\sum_{r=0}^{j} \sigma'_{j-r}\sigma'_r - 2 (-1)^{j-1}
\sum_{r=0}^{j-1} \sigma'_{j-r-1}\sigma'_r \, z_{k'/2}
+ (-1)^{j-2}\sum_{r=0}^{j-2} \sigma'_{j-r-2}\sigma'_r \, z_{k'/2}^2 \big) .\\
&=& (-1)^{j} \,\overline{p}_{k^*} \big( \sum_{r=0}^{j} 
\sigma'_{j-r}\sigma'_r + 2\,( \sum_{r=0}^{j-1} \sigma'_{j-r-1}\sigma'_r )\, z_{k'/2}
+ ( \sum_{r=0}^{j-2} \sigma'_{j-r-2}\sigma'_r )\, z_{k'/2}^2\big)
\end{eqnarray*}
since $ \overline{p}'_{k^*}= \overline{p}_{k^*}$. 
On the other hand, we have  
\begin{eqnarray*}
\sum_{r=0}^{j} \sigma_{j-r}^{(m)}\sigma_r^{(m)} &=&  
\sum_{r=0}^{j} (\sigma'_{j-r} +\sigma'_{j-r-1} z_{k'/2})
(\sigma'_r +\sigma'_{r-1} z_{k'/2})\\
&=&  \sum_{r=0}^{j} \sigma'_{j-r}\sigma'_r +( \sum_{r=0}^{j} \sigma'_{j-r-1} \sigma'_r 
+ \sum_{r=0}^{j} \sigma'_{j-r}\sigma'_{r-1}) \,z_{k'/2}
+ ( \sum_{r=0}^{j} \sigma'_{j-r-1} \sigma'_{r-1}) \,z_{k'/2}^2\\
&=&\sum_{r=0}^{j} \sigma'_{j-r}\sigma'_r + 
2 \,( \sum_{r=0}^{j-1} \sigma'_{j-r-1}\sigma'_r ) \,z_{k'/2}
+( \sum_{r=0}^{j-2} \sigma'_{j-r-2}\sigma'_r ) \,z_{k'/2}^2.
\end{eqnarray*}
Thereby, for any $j\in\{2,\ldots,k_m\}$, we get
$$\overline{p}_{k^*+j} = (-1)^{j} \,\overline{p}_{k^*}\sum_{r=0}^{j} 
\sigma_{j-r}^{(m)}\sigma_r^{(m)}.$$
For $j=1$, since $ \overline{p}'_{k^*}= \overline{p}_{k^*}$, 
by Lemma~\ref{l4ca3},(ii), and our induction hypothesis, 
\begin{eqnarray*} 
\overline{p}_{k^*+1} = \overline{p}'_{k^*+1} -2\,\overline{p}'_{k^*}\,z_{k'/2}
= \overline{p}_{k^*} (-2\sigma'_1)-2\,\overline{p}_{k^*}\,z_{k'/2}
= \overline{p}_{k^*} (-2\sigma_1^{(m)}).
\end{eqnarray*}
This proves the first equality of the proposition. 

\small

For the second one, it suffices to prove by induction on $s\in\{1,\ldots,m\}$ that 
$$\overline{p}_{k_0+\cdots+k_{s-1}}=
(\sigma_{k_0/2}^{(1)}\ldots \sigma_{k_{s-1}/2}^{(s-1)})^2.$$ 
For $s=1$, then $\overline{p}_{k_0+\cdots+k_{s-1}}=\overline{p}_0=1$ and 
$\sigma_{\varnothing,0}=1$ by convention. Assume $s>2$ and the statement true for 
$1,\ldots,s-1$. By the first equality with $j=k_s$, 
$\overline{p}_{k_0+\cdots+k_{s}}=(-1)^{k_s}\, \overline{p}_{k_0+\cdots+k_{s-1}} 
(\sigma_{k_s/2}^{(s)})^2$, whence the statement by induction hypothesis since $k_s$ is 
even. 
\end{proof}

\begin{rema}\label{r3ca3}
Remind that the polynomial ${\tilde{p}}_{k}$ was defined 
before Lemma \ref{lca3}.
As a by product of the previous formula, whenever $k'=k$, we obtain 
$$\overline{\tilde{p}}_{k}=\sigma_{k_0/2}^{(1)}\ldots \sigma_{k_{m}/2}^{(m)} .$$ 
\end{rema}

For $s\in\{1,\ldots,m\}$ and $j\in\{1,\ldots,k_s\}$, set 
$$\rho_{k_0+\cdots+k_{s-1}+j}:=\displaystyle{\frac{\overline{p}_{k_0+\cdots+k_{s-1}
+j}}{\overline{p}_{k_0+\cdots+k_{s-1}}}}.$$
Proposition \ref{pca3} says that $\rho_{k_0+\cdots+k_{s-1}+j}$ is an element of 
${\rm Frac}({\rm S}(\g^{e})^{\g^{e}})\cap {\rm S}(\g^{e})={\rm S}(\g^{e})^{\g^{e}}$.

\begin{lemma}\label{l5ca3}
Let $s\in\{1,\ldots,m\}$ and $j\in\{k_s/2+1,\ldots,k_s\}$. There is a polynomial 
$R_j^{(s)}$ of degree $j$ such that 
$$\rho_{k_0+\cdots+k_{s-1}+j}=R_j^{(s)}(\rho_{k_0+\cdots+k_{s-1}+1},\ldots,
\rho_{k_0+\cdots+k_{s-1}+k_s/2}).$$
In particular, for any $j\in\{k_1/2+1,\ldots,k_1\}$, we have  
$$\overline{p}_{j}=R_j^{(1)}(\overline{p}_{1},\ldots,\overline{p}_{k_1/2}).$$
\end{lemma}

\begin{proof}
1) Prove by induction on $j\in\{1,\ldots,k_s/2\}$ that for some polynomial $T_j^{(s)}$ of
degree $j$, 
$$\sigma_j^{(s)}=T_j^{(s)}(\rho_{k_0+\cdots+k_{s-1}+1},\ldots,
\rho_{k_0+\cdots+k_{s-1}+j}).$$ 
By Proposition \ref{pca3}, $\rho_{k_0+\cdots+k_{s-1}+1}=-(\sigma_1^{(s)}\sigma_0^{(s)}+
\sigma_0^{(s)}\sigma_1^{(s)})=-2\sigma_1^{(s)}$. Hence, the statement is true for $j=1$. 
Suppose $j\in\{2,\ldots,k_s/2\}$ and the statement true for 
$\sigma_1^{(s)},\ldots,\sigma_{j-1}^{(s)}$. 
Since $j \leq k_s/2$, $\sigma_{j}^{(s)}\not=0$, and by Proposition \ref{pca3}, 
$$\rho_{k_0+\cdots+k_{s-1}+j}\,=\,(-1)^j (\sigma_{j}^{(s)}\sigma_0^{(s)}+
\sigma_0\sigma_j^{(s)}) + (-1)^{j}\sum_{r=1}^{j-1} \sigma_{j-r}^{(s)}\sigma_r^{(s)}
\,=\,2 (-1)^j \sigma_{j}^{(s)} + 
(-1)^{j}\sum_{r=1}^{j-1} \sigma_{j-r}^{(s)}\sigma_r^{(s)}.$$
So, the statement for $j$ follows from our induction hypothesis. 

\small 

\noi
2) Let $j\in\{k_s/2+1,\ldots,k_s\}$. Proposition~\ref{pca3} shows that 
$\rho_{k_0+\cdots+k_{s-1}+j}$ is a polynomial in 
$\sigma_1^{(s)},\ldots,\sigma_{k_s/2}^{(s)}$. Hence, by Step 1), 
$\rho_{k_0+\cdots+k_{s-1}+j}$ is a polynomial in 
$$\rho_{k_0+\cdots+k_{s-1}+1},\ldots,\rho_{k_0+\cdots+k_{s-1}+k_s/2}.$$ 
Furthermore, by Proposition \ref{pca3} and Step (1), this polynomial has degree $j$. 
\end{proof}

\begin{rema}\label{r4ca3}
By Remark~\ref{r3ca3} and the above proof, if $k'=k$ then for some polynomial 
$\tilde{R}$ of degree $k_m/2$, 
$$\displaystyle{\frac{\overline{\tilde{p}}_{k}}{\sigma_{k_0/2}^{(1)}\ldots 
\sigma_{k_{m-1}/2}^{(m-1)}}}=\sigma_{k_{m}/2}^{(m)}=\tilde{R}(\rho_{k_0+\cdots+k_{m-1}+1},
\ldots,\rho_{k_0+\cdots+k_{m-1}+k_m/2}).$$
\end{rema}

\begin{theorem}\label{tca3}
{\rm (i)} Assume that $\lambda$ satisfies the condition $(\ast)$ and that  
$\lambda_1=\cdots=\lambda_{k'}$. Then $e$ is good. 

{\rm (ii)} Assume that $k=4$ and that $\lambda_1,\lambda_2,\lambda_3,\lambda_4$ 
are even. Then $e$ is good. 
\end{theorem}

For example, $(6,6,6,6,5,3)$ satisfies the hypothesis of (i) and 
$(6,6,4,4)$ satisfies the hypothesis of (ii). 

\begin{rema} \label{r2ca3}
If $\lambda$ satisfies the condition $(\ast)$ then by Lemma~\ref{l3ca2}, 
$$\dim {\goth g}^{e}+\rg - 2 (\delta_1+\cdots+\delta_\rg) =k'.$$
Indeed, if $k$ is odd, then $n_\lambda-d_\lambda=n_{\lambda'}-d_{\lambda'}$ 
where $\lambda'=(\lambda_1,\ldots,\lambda_{k'},\lambda_{k'+1})$ so that 
$n_\lambda-d_\lambda=n_{\lambda'}-d_{\lambda'}=n_{\lambda'}=k'$ since 
$\lambda_{k'+1}$ is odd.  If $k$ is even, then 
$d_\lambda=n_{\lambda'}=k'$ where $\lambda'=(\lambda_1,\ldots,\lambda_{k'})$. 
\end{rema}

\begin{proof}
(i) In the previous notations, the hypothesis means that $m=1$ 
and $k'=k_m$. According to 
Lemma~\ref{l5ca3} and Lemma \ref{l2ca3}, for $j\in\{k'/2+1,\ldots,k'-1\}$, 
$${p}_j=R_j^{(1)}({p}_1,\ldots,{p}_{k'/2}),$$
where $R_j^{(1)}$ is a polynomial of degree $j$. Moreover, if $k'=k$, then by 
Remark~\ref{r4ca3} and Lemma \ref{l2ca3},   
$${\tilde{p}}_{k}=\tilde{R}({p}_1,\ldots,{p}_{k/2}),$$
where $\tilde{R}$ is a polynomial of degree $k/2$. 

\small

-\; If $k'<k$, set for any $j\in\{k'/2+1,\ldots,k'\}$, 
$$r_j:=q_{\nu_j}-R_j^{(1)}({q}_{\nu_1},\ldots,{q}_{\nu_{k'/2}}).$$
Then by Lemma~\ref{lca3}, 
$$\forall \, j\in\{k'/2+1,\ldots,k'\},\quad \deg {\ie {r_j}} \geq j+1.$$

\small

-\; If $k'=k$, set for $j\in\{k/2+1,\ldots,k'-1\}$, 
$$r_j:=q_{\nu_j}-R_j^{(1)}({q}_{\nu_1},\ldots,{q}_{\nu_{k'/2}}) \; \text{ and }\;
r_k:=q_{\nu_k}-\tilde{R}({q}_{\nu_1},\ldots,{q}_{\nu_{k/2}}).$$
Then by Lemma~\ref{lca3}, 
$$\forall \, j\in\{k/2+1,\ldots,k-1\},\quad \deg {\ie {r_j}} \geq j+1 
\quad \textrm{ and }\quad \deg {\ie {r_k}} \geq k/2+1.$$ 
In both cases, $$\{ q_{j} \; \vert \; j \in \{1,\ldots,\rg\}\smallsetminus 
\{\nu_{k'/2+1},\ldots,\nu_{k'}\} \} 
\;\cup \; \{{{r_{k'/2+1}}},\ldots,{{r_{k'}}}\}$$ 
is a homogenous generating system of S$(\g)^\g$. Denote by $\hat{\delta}$ the sum 
of the degrees of the polynomials 
$${\ie {q_j}}, \;  j \in \{1,\ldots,\rg\}\smallsetminus \{\nu_{k'/2+1},\ldots,\nu_{k'}\},
\quad {\ie {r_{k'/2+1}}},\ldots,{\ie {r_{k'}}}.$$
The above discussion shows that $\hat{\delta} \geq \delta_1+\cdots+\delta_\rg +k'/2$. 
By Remarks \ref{r2ca3}, we obtain   
$$\dim {\goth g}^{e}+\rg - 2 \hat{\delta} \leq 0.$$
In conclusion, by \cite[Theorem 2.1]{PPY} and Theorem \ref{tge2}, $e$ is good. 

(ii) In the previous notations, the hypothesis means that $k'=k=4$.  
If $m=1$ the statement is a consequence of (i). Assume that $m=2$. Then by 
Proposition \ref{pca3}, 
$\overline{p}_1=-2z_1$, $\overline{p}_2=z_1^2$, $\overline{p}_3=-2z_1^2z_2$ and 
$\overline{p}_4=(z_1z_2)^2$. Moreover, $\overline{\tilde{p}}_4=z_1z_2$. 
Hence, by Lemma \ref{l2ca3}, $p_2=\frac{1}{4}p_1^2$ and $p_3=p_1\tilde{p}_4$.  
Set $r_2:=q_{\nu_2}-\frac{1}{4}q_{\nu_1}^2$ and $r_3:=q_{\nu_3}-q_{\nu_1}q_{\nu_4}$. 
Then $\deg {\ie {r_2}}\geq 3$ and $\deg {\ie {r_3}}\geq 4$. Moreover, 
$$\{\poi q1{,\ldots,}{\rg}{}{}{}\} \smallsetminus \{q_{\nu_2},q_{\nu_3}\} \cup 
\{r_2,r_3\}$$
is a homogenous generating system of S$(\g)^\g$. Denoting by $\hat{\delta}$ the sum 
of the degrees of the polynomials 
$$\{{\ie {q_1}},\ldots,{\ie {q_\rg}}\} \smallsetminus 
\{{\ie {q_{\nu_2}}},{\ie {q_{\nu_3}}}\} \cup 
\{{\ie {r_2}},{\ie {r_3}}\},$$
we obtain that $\hat{\delta} \geq \delta_1+\cdots+\delta_\rg +2$. But 
$\dim {\goth g}^{e}+\rg
- 2 (\delta_1+\cdots+\delta_\rg)=k'=4$ by Remark \ref{r2ca3}. So, $\dim {\goth g}^{e}
+\rg -2 \hat{\delta} \leq 0$. 
In conclusion, by \cite[Theorem~2.1]{PPY} and Theorem \ref{tge2}, $e$ is good. 
\end{proof}

\begin{rema} \label{r5ca3} 
Assume that $\g=\mathfrak{so}(\mathbb{V})$, with $\dim \mathbb{V}=12$, and that 
$e$ belongs to the nilpotent orbit of $\g$ associated with the partition $(5,5,1,1)$ of 
$12$. Then the degrees of $\ie{q_1},\ie{q_2},\ie{q_3},\ie{q_4},\ie{q_5},\ie{q_6}$ 
are $1,1,2,2,2,2$ respectively. Since $10=1+1+2+2+2+2 =(\dim\g^{e}+\ell)/2$, the 
polynomial functions $\ie{q_1},\ie{q_2},\ie{q_3},\ie{q_4},\ie{q_5},\ie{q_6}$ are 
algebraically independent, and by Theorem \ref{tge2}, $S(\g^{e})^{\g^{e}}$ is polynomial. 
One can satisfy that $\ie{q_5}=z^2$ for some $z$ in the center $\z(\g^{e})$ 
of $\g^{e}$. Since $\z(\g^{e})$ has dimension 3, for any other choice of homogenous  
generators $q_1,\ldots,q_\ell$ of $S(\g)^\g$, $S(\g^{e})^{\g^{e}}$ cannot be generated 
by the elements $\ie{q_1},\ie{q_2},\ie{q_3},\ie{q_4},\ie{q_5},\ie{q_6}$ for degree 
reasons. 

This shows that Condition (2) of Theorem \ref{ti2} cannot be removed to ensure 
that $S(\g^{e})^{\g^{e}}$ is a polynomial algebra in the variables 
$\ie{q_1},\ie{q_2},\ie{q_3},\ie{q_4},\ie{q_5},\ie{q_6}$. However, one can sometimes, 
as in this example, provide explicit generators. 
\end{rema}

\section{Examples in simple exceptional Lie algebras} \label{et}
We give in this section examples of good nilpotent elements in simple exceptional Lie 
algebras of type ${\bf E}_{6}$, ${\bf F}_{4}$ or ${\bf G}_{2}$ which are not covered 
by \cite{PPY}. These examples are all obtained through Theorem \ref{tge2}. 

According to 
\cite[Theorem~0.4]{PPY} and Theorem \ref{tge2}, the elements of the minimal nilpotent 
orbit of $\g$, for $\g$ not of type ${\bf E}_{8}$, are good. In addition, as it is 
explained in the Introduction, the elements of the regular, or subregular, nilpotent orbit 
of $\g$ are good. So we do not consider here these cases.

\begin{example}\label{eet} 
Suppose that ${\goth g}$ has type ${\bf E}_{6}$. Let ${\Bbb V}$ be the module
of highest weight the fundamental weight $\varpi _{1}$ with the notation of Bourbaki.
Then ${\Bbb V}$ has dimension 27 and ${\goth g}$ identifies with a subalgebra of 
${\goth {sl}}_{27}(\k)$. For $x$ in ${\goth {sl}}_{27}(\k)$ and for $i=2,\ldots,27$, let 
$p_{i}(x)$ be the coefficient of $T^{27-i}$ in $\det (T -x)$ and denote by $q_{i}$ the 
restriction of $p_{i}$ to ${\goth g}$. Then $(q_{2},q_{5},q_{6},q_{8},q_{9},q_{12})$ is a 
generating family of $\ai g{}{}$ since these polynomials are algebraically independent, 
\cite{Me}. 
Let $(e,h,f)$ be an ${\goth {sl}}_{2}$-triple of ${\goth g}$. Then $(e,h,f)$
is an ${\goth {sl}}_{2}$-triple of ${\goth {sl}}_{27}(\k)$. We denote by ${\ie {p_i}}$ 
the initial homogenous component of the restriction to $e+\tilde{\g}^{f}$ of 
$p_i$ where $\tilde{\g}^{f}$ is the centralizer of $f$ in ${\goth {sl}}_{27}(\k)$. 
As usual, ${\ie {q_i}}$ denotes the initial homogenous component of the 
restriction to $e+{\g}^{f}$ of $q_i$. 
For $i=2,5,6,8,9,12$,
$$ \deg \ie p_{i} \leq \deg \ie q_{i}.$$ 
In some cases, from the knowledge of the maximal eigenvalue of the 
restriction of $\ad h$ to ${\goth g}$ and the $\ad h$-weight of $\ie p_{i}$, 
it is possible to deduce that $\deg {\ie {p_{i}}}<\deg {\ie {q_{i}}}$. On the other hand, 
$$\deg \ie q_{2} + \deg \ie q_{5} + \deg \ie q_{6} + \deg \ie q_{8} + \deg \ie q_{9} + 
\deg \ie q_{12} \leq \frac{1}{2}(\dim {\goth g}^{e} + 6),$$
with equality if and only if 
$\ie q_{2},\ie q_{5},\ie q_{6},\ie q_{8},\ie q_{9},\ie q_{12}$ are algebraically 
independent. From this, it is possible to deduce in some cases that $e$ is good.  
These cases are listed in Table \ref{tabet} where the nine columns are indexed in the 
following way:
\begin{itemize}
\item [1:] the label of the orbit $G.e$ in the Bala-Carter classification, 
\item [2:] the weighted Dynkin diagram of $G.e$, 
\item [3:] the dimension of ${\goth g}^{e}$,
\item [4:] the partition of $27$ corresponding to the nilpotent element $e$ of 
${\goth {sl}}_{27}(\k)$,
\item [5:] the degrees of 
$\ie p_{2},\ie p_{5},\ie p_{6},\ie p_{8},\ie p_{9},\ie p_{12}$,
\item [6:] their $\ad h$-weights,
\item [7:] the maximal eigenvalue $\nu$ of the restriction of $\ad h$ to 
${\goth g}$,
\item [8:] the sum $\Sigma$ of the degrees of
$\ie p_{2},\ie p_{5},\ie p_{6},\ie p_{8},\ie p_{9},\ie p_{12}$,
\item [9:] the sum $\Sigma'=\frac{1}{2}(\dim\g^{e}+\rg )$. 
\end{itemize}

{\tiny \begin{table}[ht!]
\begin{tabular}{ccccccccccc}
&&Label&{\setlength{\unitlength}{0.01in}
\begin{picture}(90,20)
\put(3,0){\circle{6}}
\put(23,0){\circle{6}}
\put(43,0){\circle{6}}
\put(63,0){\circle{6}}
\put(83,0){\circle{6}}    
\put(43,-20){\circle{6}} 
\put(6,0){\line(1,0){14}}
\put(26,0){\line(1,0){14}}
\put(46,0){\line(1,0){14}}
\put(66,0){\line(1,0){14}}
\put(43,-3){\line(0,-1){14}}
\end{picture}} 
& $\dim\g^{e}$ & partition & $\deg \ie p_{i}$ & weights & $\nu$ & $\Sigma$& $\Sigma'$\\
\bigskip\\ 
\hline
{\bf\footnotesize 1.}&&$D_5$ &{\tiny$\begin{array}{ccccc}\\
2&0&2&0&2\\
&&2&&\end{array}$}
& 10 & (11,9,5,1,1) & 1,1,1,1,1,1 & 2,8,10,14,16,22 & 14 & 6 & 8 \\
{\bf\footnotesize 2.}&& $E_6(a_3)$ &{\tiny$\begin{array}{ccccc}\\
2&0&2&0&2\\
&&0&&\end{array}$}& 
12 & $(9,7,5^{2},1)$ & 1,1,1,1,1,2 & 2,8,10,14,16,20 & 10 & 7 & 9 \\
{\bf\footnotesize 3.}&& $D_5(a_1)$ &{\tiny$\begin{array}{ccccc}\\
1&1&0&1&1\\
&&2&&\end{array}$}
& 14 & (8,7,6,3,2,1) & 1,1,1,1,2,2 & 2,8,10,14,14,20 & 10 & 8 & 10 \\
{\bf\footnotesize 4.}&&$A_5$ &{\tiny$\begin{array}{ccccc}\\
2&1&0&1&2\\
&&1&&\end{array}$}
& 14 & $(9,6^{2},5,1)$ & 1,1,1,1,1,2 & 2,8,10,14,16,20 & 10 & 7 & 10 \\
{\bf\footnotesize 5.}&&$A_4+A_1$ &{\tiny$\begin{array}{ccccc}\\
1&1&0&1&1\\
&&1&&\end{array}$}
& 16 & $(7,6,5,4,3,2)$ & 1,1,1,2,2,2 & 2,8,10,12,14,20 & 8 & 9 & 11 \\
{\bf\footnotesize 6.}&&$D_4$ &{\tiny$\begin{array}{ccccc}\\
0&0&2&0&0\\
&&2&&\end{array}$}
& 18 & $(7^{3},1^{6})$ & 1,1,1,2,2,2 & 2,8,10,12,14,20 & 10 & 9 & 12 \\
{\bf\footnotesize 7.}&&$D_4(a_1)$ &{\tiny$\begin{array}{ccccc}\\
0&0&2&0&0\\
&&0&&\end{array}$}
& 20 & $(5^{3},3^{3},1^{3})$ & 1,1,2,2,2,3 & 2,8,8,12,14,18 & 6 & 11 & 13 \\
{\bf\footnotesize 8.}&&$2A_2+A_1$ &{\tiny$\begin{array}{ccccc}\\
1&0&1&0&1\\
&&0&&\end{array}$}
& 24 & $(5,4^{2},3^{3},2^{2},1)$ & 1,1,2,2,2,3 & 2,8,8,12,14,18 & 5 & 11 & 15 
\small\\
\hline
\small
\end{tabular}
\caption{\footnotesize Data for {\bf E}$_6$} \label{tabet}
\end{table}}

\noindent
In all cases, we observe that $\Sigma<\Sigma'$, i.e., 
$$\deg \ie p_{2} + \deg \ie p_{5} + \deg \ie p_{6} + \deg \ie p_{8} + \deg \ie p_{9} + 
\deg \ie p_{12} < \frac{1}{2}(\dim {\goth g}^{e} + 6).$$
So, we need some arguments that we give below. 
\begin{itemize}
\item [{\bf 1.}] Since $14 < 16$, $\deg \ie p_{i} < \deg \ie q_{i}$ for $i=9,12$.
\item [{\bf 2.}] Since $10 < 14$, $\deg \ie p_{i} < \deg \ie q_{i}$ for $i=8,9$.
\item [{\bf 3.}] Since $10 < 14$, $\deg \ie p_{8} < \deg \ie q_{8}$. Moreover,
the multiplicity of the weight $10$ equals $1$. So, either  
$\deg \ie q_{6} > 1$, or $\deg \ie q_{12} > 2$, or $\ie q_{12} \in \k \ie q_{6}^2$.
\item [{\bf 4.}] Since $10 < 14$, $\deg \ie p_{i} < \deg \ie q_{i}$ for $i=8,9$.
Moreover, the multiplicity of the weight $10$ equals $1$. So, either  
$\deg \ie q_{6} > 1$, or $\deg \ie q_{12} > 2$, or $\ie q_{12} \in \k \ie q_{6}^2$.
\item [{\bf 5.}] Since $8 < 10$ and $2\mul 8 < 20$, $\deg \ie p_{i} < \deg \ie q_{i}$ 
for $i=6,12$.
\item [{\bf 6.}] Since the center of ${\goth g}^{e}$ has dimension $2$ and the weights
of $h$ in the center are $2$ and $10$, $\deg \ie p_{5}<\deg \ie q_{5}$. Moreover, since
the weights of $h$ in ${\goth g}^{e}$ are $0,2,6,10$, $\deg \ie p_{9} < \deg \ie q_{9}$
and since the multiplicity of the weight $10$ equals $1$, either 
$\deg \ie q_{6} > 1$, or $\deg \ie q_{12} > 2$, or $\ie q_{12} \in \k \ie q_{6}^2$.
\item [{\bf 7.}] Since $6 < 8$ and $2\mul 6 < 14$, $\deg \ie p_{i} < \deg \ie q_{i}$ 
for $i=5,9$.
\item [{\bf 8.}] Since $5 < 8$, $2\mul 5 < 12$ and $3\mul 5 < 18$, 
$\deg \ie p_{i} < \deg \ie q_{i}$ for $i=5,8,9,12$.  
\end{itemize}
In cases ${\bf 1,2,5,7,8}$, the discussion shows that 
$$\deg \ie q_{2} + \deg \ie q_{5} + \deg \ie q_{6} + \deg \ie q_{8} + \deg \ie q_{9} + 
\deg \ie q_{12} = \frac{1}{2}(\dim {\goth g}^{e} + 6).$$
Hence, $\ie q_{2},\ie q_{5},\ie q_{6},\ie q_{8},\ie q_{9},\ie q_{12}$ are algebraically 
independent and by Theorem \ref{tge2}, $e$ is good. 
In cases ${\bf 3,4,6}$, if the 
above equality does not hold, then for some $a$ in $\k^{*}$,
$$\deg \ie q_{2} + \deg \ie q_{5} + \deg \ie q_{6} + \deg \ie q_{8} + \deg \ie q_{9} + 
\deg \ie (q_{12}-aq_{6}^{2}) = \frac{1}{2}(\dim {\goth g}^{e} + 6).$$
Hence $\ie q_{2},\ie q_{5},\ie q_{6},\ie q_{8},\ie q_{9},\ie (q_{12}-aq_{6}^{2})$ are 
algebraically independent and by Theorem \ref{tge2}, $e$ is good. 

In conclusion, it remains nine unsolved nilpotent orbits in type {\bf E}$_6$. 
\end{example}

\begin{example}\label{e2et} 
Suppose that ${\goth g}$ is simple of type ${\bf F}_{4}$. Let ${\Bbb V}$ be the module
of highest weight the fundamental weight $\varpi _{4}$ with the notation of Bourbaki.
Then ${\Bbb V}$ has dimension 26 and ${\goth g}$ identifies with a subalgebra of 
${\goth {sl}}_{26}(\k)$. For $x$ in ${\goth {sl}}_{26}(\k)$ and for $i=2,\ldots,26$, let 
$p_{i}(x)$ be the coefficient of $T^{26-i}$ in $\det (T -x)$ and denote by $q_{i}$ the 
restriction of $p_{i}$ to ${\goth g}$. Then $(q_{2},q_{6},q_{8},q_{12})$ is a generating 
family of $\ai g{}{}$ since these polynomials are algebraically independent, \cite{Me}. 
Let $(e,h,f)$ be an ${\goth {sl}}_{2}$-triple of ${\goth g}$. Then $(e,h,f)$
is an ${\goth {sl}}_{2}$-triple of ${\goth {sl}}_{26}(\k)$. As in Example~\ref{eet}, in 
some cases, it is possible to deduce that $e$ is good. These 
cases are listed in Table \ref{tab2et}, indexed as in Example~\ref{eet}.

{\tiny \begin{table}[ht!]
\begin{tabular}{ccccccccccc}
&&Label&{\setlength{\unitlength}{0.01in}
\begin{picture}(105,30)
\put(20,4){\circle{6}}
\put(40,4){\circle{6}}
\put(60,4){\circle{6}}
\put(80,4){\circle{6}}
\put(23,4){\line(1,0){14}}
\put(63,4){\line(1,0){14}}
\put(42,6){\line(1,0){16}}
\put(42,2){\line(1,0){16}}
\put(45,0){\normalsize $>$}
\end{picture}}
& $\dim\g^{e}$ & partition& $\deg \ie p_{i}$ & weights & $\nu$ & 
$\Sigma$& $\Sigma'$
\small \\
\hline
{\bf\footnotesize 1.}&&$F_4(a_2)$ &{\tiny$\begin{array}{ccccc}\\
0&2&0&2\\\\\end{array}$}
& 8 & $(9,7,5^{2})$ & 1,1,1,2 & 2,10,14,20 & 10 & 5 & 6 \\
{\bf\footnotesize 2.}&&$C_3$ &{\tiny$\begin{array}{ccccc}\\
1&0&1&2\\\\\end{array}$}
& 10 & $(9,6^{2},5)$ & 1,1,1,2 & 2,10,14,20 & 10 & 5 & 7 \\
{\bf\footnotesize 3.}&&$B_3$ &{\tiny$\begin{array}{ccccc}\\
2&2&0&0\\\\\end{array}$}
& 10 & $(7^{3},1^{5})$ & 1,1,2,2 & 2,10,12,20 & 10 & 6 & 7 \\
{\bf\footnotesize 4.} &&$F_4(a_3)$&{\tiny$\begin{array}{ccccc}\\
0&2&0&0\\\\\end{array}$}
& 12 & $(5^{3},3^{3},1^{2})$ & 1,2,2,3 & 2,8,12,18 & 6 & 8 & 8 \\
{\bf\footnotesize 5.}&&$C_3(a_1)$ &{\tiny$\begin{array}{ccccc}\\
1&0&1&0\\\\\end{array}$}
& 14 & $(5^{2},4^{2},3,2^{2},1)$ & 1,2,2,3 & 2,8,12,18 & 6 & 8 & 9 \\ 
{\bf\footnotesize 6.} &&$\tilde{A}_2+A_1$&{\tiny$\begin{array}{ccccc}\\
0&1&0&1\\\\\end{array}$}
& 16 & $(5,4^{2},3^{3},2^{2})$ & 1,2,2,3 & 2,8,12,18 & 5 & 8 & 10\\
\hline
\small
\end{tabular}
\caption{\footnotesize Data for {\bf F}$_4$} \label{tab2et}
\end{table}}

For the orbits {\bf 1}, {\bf 2}, {\bf 3}, {\bf 5},{\bf 6}, we observe 
that $\Sigma<\Sigma'$. So, we need some more arguments to conclude as in 
Example~\ref{eet}.
\begin{itemize}
\item [ {\bf 1.}] Since $10 < 14$, $\deg \ie p_{8} < \deg \ie q_{8}$.
\item [ {\bf 2.}] Since $10 < 14$, $\deg \ie p_{8} < \deg \ie q_{8}$. Moreover,
the multiplicity of the weight $10$ equals $1$ so that 
$\deg \ie q_{6} > 1$ or $\deg \ie q_{12} > 2$ or $\ie q_{12} \in \k \ie q_{6}^2$.
\item [{\bf 3.}] The multiplicity of the weight $10$ equals $1$. So, either  
$\deg \ie q_{6} > 1$, or $\deg \ie q_{12} > 2$, or $\ie q_{12} \in \k \ie q_{6}^2$.
\item [{\bf 5.}] Suppose that $\ie q_{2},\ie q_{6},\ie q_{8},\ie q_{12}$ have 
degree $1,2,2,3$. We expect a contradiction. Since the center has dimension
$2$ and since the multiplicity of the weight $6$ equals $1$, for $z$ of weight $6$
in the center, $\ie q_{6}\in \k ez$, $\ie q_{8}\in \k z^{2}$, $\ie q_{12} \in \k z^{3}$.
So, for some $a$ and $b$ in $\k^{*}$,
$$ \begin{array}{ccc}
 \ie q_{2}^{2} \ie q_{8}-a\ie q_{6}^{2} = 0, && \ie q_{12}^{2} - b \ie q_{8}^{3} = 0
\end{array}$$
Hence, $q_{2},q_{6},q_{2}^{2}q_{8}-aq_{6}^{2},q_{12}^{2}-bq_{8}^{3}$ are algebraically 
independent element of $\ai g{}{}$ such that
$$ \deg \ie q_{2} + \deg \ie q_{6} + \deg \ie (q_{2}^{2}q_{8}-aq_{6}^{2})
+ \deg \ie (q_{12}^{2}-bq_{8}^{3}) \geq  1 + 2 + 5 + 7 > 2 + 3 + 9$$
whence a contradiction by \cite[Theorem~2.1]{PPY} (see Lemma~\ref{lrc1}). 
\item [{\bf 6.}] Since $2\mul 5 < 12$ and $3 \mul 5 < 18$, 
$\deg \ie q_{8} > \deg \ie p_{8}$ and $\deg \ie q_{12} > \deg \ie p_{12}$.  
\end{itemize}

In conclusion, it remains six unsolved nilpotent orbits in type {\bf F}$_4$. 
\end{example}

\begin{example}\label{e3et} 
Suppose that ${\goth g}$ is simple of type ${\bf G}_{2}$. Let ${\Bbb V}$ be the module
of highest weight the fundamental weight $\varpi _{1}$ with the notation of Bourbaki.
Then ${\Bbb V}$ has dimension $7$ and ${\goth g}$ identifies with a subalgebra of 
${\goth {sl}}_{7}(\k)$. For $x$ in ${\goth {sl}}_{7}(\k)$ and for $i=2,\ldots,7$, let 
$p_{i}(x)$ be the coefficient of $T^{7-i}$ in $\det (T -x)$ and denote by $q_{i}$ the 
restriction of $p_{i}$ to ${\goth g}$. Then $q_{2},q_{6}$ is a generating 
family of $\ai g{}{}$ since these polynomials are algebraically independent, \cite{Me}. 
Let $(e,h,f)$ be an ${\goth {sl}}_{2}$-triple of ${\goth g}$. Then $(e,h,f)$
is an ${\goth {sl}}_{2}$-triple of ${\goth {sl}}_{7}(\k)$. 
There is only one nonzero nilpotent orbit 
which is neither regular, subregular or minimal. 
For $e$ in it, we deduce that $e$ is good from Table \ref{tab3et}, indexed as in Example~\ref{eet}, 
since $\Sigma=\Sigma'$.  

{\tiny 
\begin{table}[ht!]
\begin{tabular}{ccccccccc}
Label&{\setlength{\unitlength}{0.01in}
\begin{picture}(40,17)(0,4)
  \put(8,5){\circle{6}}
  \put(30,5){\circle{6}}
  \put(8,2){\line(1,0){22}}
  \put(11,5){\line(1,0){16}}
  \put(8,7.7){\line(1,0){22}}
  \put(14,1){\normalsize $<$}
\end{picture}}
& $\dim\g^{e}$ & partition& $\deg \ie p_{i}$ & weights & $\nu$ & 
$\Sigma$\hspace{1cm}& $\Sigma'$
\small \\
\hline
$\tilde{A}_1$ &{\tiny$\begin{array}{ccccc}\\
1&0\\\\\end{array}$}
& 6 & ($3,2^{2}$) & 1,3 & 2,6 & 3 & 4 & 4 \\
\hline
\small
\end{tabular}
\caption{\footnotesize Data for {\bf G}$_2$} \label{tab3et}
\end{table}}
In conclusion, all elements are good in type ${\bf G}_{2}$. 
\end{example}

\section{Other examples, remarks and a conjecture} \label{rc}
This section provides examples of nilpotent elements which satisfy the polynomiality 
condition but that are not good. We also obtain an example of nilpotent element in 
type {\bf D}$_{7}$ which does not satisfy the polynomiality condition 
(cf.~Example \ref{erc3}). Then we conclude with some remarks and a conjecture. 

\subsection{Some general results}\label{rc1}
In this subsection, ${\goth g}$ is a simple Lie algebra over $\k$ and $(e,h,f)$ is 
an ${\goth {sl}}_{2}$-triple of ${\goth g}$. For $p$ in $\e Sg$, $\ie p$ is the initial
homogenous component of the restriction of $p$ to the Slodowy slice $e+{\goth g}^{f}$.
Recall that $\k[e+{\goth g}^{f}]$ identifies with $\es S{{\goth g}^{e}}$ by the Killing
form $\dv ..$ of ${\goth g}$. 

Let $\eta _{0}\in \tk {\k}{{\goth g}^{e}}\ex 2{{\goth g}^{f}}$ be the bivector
defining the Poisson bracket on $\es S{{\goth g}^{e}}$ induced from the Lie bracket. 
According to the main theorem
of~\cite{Pr}, $\es S{{\goth g}^{e}}$ is the graded algebra relative   
to the Kazhdan filtration of the finite $W$-algebra associated with 
$e$ so that $\es S{{\goth g}^{e}}$ inherits another Poisson structure. 
The so-obtained graded algebra structure is the Slodowy graded algebra 
structure (see Subsection \ref{sg1}). 
Let $\eta\in\tk {\k}{\es S{{\goth g}^{e}}}\ex 2{{\goth g}^{f}}$ be the bivector 
defining this other Poisson structure. According to~\cite[Proposition 6.3]{Pr} 
(see also \cite[\S2.4]{PPY}), 
$\eta _{0}$ is the initial homogenous component of $\eta $. Denote by $r$ the 
dimension of ${\goth g}^{e}$ and set:
$$\begin{array}{ccc}
\omega := \eta ^{(r-\rg)/2} \in \tk {\k}{\es S{{\goth g}^{e}}}\ex {r-\rg}{{\goth g}^{f}},
&& \omega _{0} := 
\eta _{0}^{(r-\rg)/2} \in \tk {\k}{\es S{{\goth g}^{e}}}\ex {r-\rg}{{\goth g}^{f}}.  
\end{array}$$ 
Then $\omega _{0}$ is the initial homogenous component of $\omega $.

Let $\poi v1{,\ldots,}{r}{}{}{}$ be a basis of ${\goth g}^{f}$. For $\mu $ in 
$\tk {\k}{\es S{{\goth g}^{e}}}\ex {i}{{\goth g}^{e}}$, denote by $j(\mu )$ the image of 
$\poi v1{\wedge \cdots \wedge }{r}{}{}{}$ by the right interior product of $\mu$ so that
$$j(\mu ) \in \tk {\k}{\es S{{\goth g}^{e}}}\ex {r-i}{{\goth g}^{f}}.$$

\begin{lemma}\label{lrc1}
Let $\poi q1{,\ldots,}{\rg}{}{}{}$ be some homogenous generators of $\ai g{}{}$ and 
let $\poi r1{,\ldots,}{\rg}{}{}{}$ be algebraically independent homogenous elements
of $\ai g{}{}$.

{\rm (i)} For some homogenous element $p$ of $\ai g{}{}$, 
$$ \poi {\dd r}1{\wedge \cdots \wedge }{\rg}{}{}{} = 
p\, \poi {\dd q}1{\wedge \cdots \wedge }{\rg}{}{}{}.$$

{\rm (ii)} The following inequality holds:
$$ \sum_{i=1}^{\rg} \deg {\ie r}_{i} \leq \deg {\ie p} + 
\frac{1}{2}(\dim {\goth g}^{e}+\rg).$$

{\rm (iii)} The polynomials $\poi {\ie r}1{,\ldots,}{\rg}{}{}{}$ are algebraically 
independent if and only if 
$$ \sum_{i=1}^{\rg} \deg {\ie r}_{i} = \deg {\ie p} + 
\frac{1}{2}(\dim {\goth g}^{e}+\rg).$$
\end{lemma}

\begin{proof}
(i) Since $\poi q1{,\ldots,}{\rg}{}{}{}$ are generators of $\ai g{}{}$, for 
$i\in\{1,\ldots,\rg\}$, $r_{i}=R_{i}(\poi q1{,\ldots,}{\rg}{}{}{})$ where $R_{i}$ is a 
polynomial in $\rg$ indeterminates, whence the assertion with
$$p= \det (\frac{\partial R_{i}}{\partial q_{j}}, \; 1\leq i,j \leq \rg).$$

(ii) Remind that for $p$ in $\e Sg$, 
$\kappa (p)$ denotes the restriction to $\g^{f}$ of the polynomial function 
$x\mapsto p(e+x)$. 
According to~\cite[Theorem~1.2]{PPY},
$$ j(\dd \kappa(q_1)\wedge \cdots \wedge \dd \kappa(q_\rg)) = a \omega $$
for some $a$ in $\k^{*}$. Hence by (i), 
$$ j(\dd \kappa(r_1)\wedge \cdots \wedge \dd \kappa(r_\rg))  = a \kappa (p) \omega.$$
The initial homogenous component of the right-hand side is 
$a\ie p \omega _{0}$ and the degree of the initial homogenous component of the 
left-hand side is at least 
$$ \poi {\deg {\ie r}}1{+\cdots +}{\rg}{}{}{}-\rg.$$
The assertion follows since $\omega _{0}$ has degree 
$$ \frac{1}{2}(\dim {\goth g}^{e}-\rg).$$

(iii) If $\poi {\ie r}1{,\ldots,}{\rg}{}{}{}$ are algebraically independent, then the 
degree of the initial homogenous component of 
$j(\poi {\dd r}1{\wedge \cdots \wedge }{\rg}{}{}{})$ equals 
$$ \poi {\deg {\ie r}}1{+\cdots +}{\rg}{}{}{}-\rg$$
whence
$$ \poi {\deg {\ie r}}1{+\cdots +}{\rg}{}{}{} = \deg {\ie p} + 
\frac{1}{2}(\dim {\goth g}^{e}+\rg)$$
by the proof of (ii). Conversely, if the equality holds, then 
\begin{eqnarray}\label{eqrc1}
j(\poi {\dd {\ie r}}1{\wedge \cdots \wedge }{\rg}{}{}{}) = a \ie{p} \omega _{0}
\end{eqnarray}
by the proof of (ii). In particular, $\poi {\ie r}1{,\ldots,}{\rg}{}{}{}$ are 
algebraically independent.
\end{proof}

\begin{coro}\label{crc1}
For $i=1,\ldots,\rg$, let $r_{i} := R_{i}(\poi q1{,\ldots,}{i}{}{}{})$ be a homogenous
element of $\ai g{}{}$ such that 
$\displaystyle{\frac{\partial R_{i}}{\partial q_{i}}}\not=0$. Then 
$\poi {\ie r}1{,\ldots,}{\rg}{}{}{}$ are algebraically independent if and only if 
$$ \poi {\deg \ie r}1{+\cdots +}{\rg}{}{}{} = 
\sum_{i=1}^{\rg} \deg \ie{p_{i}} +
\frac{1}{2}(\dim {\goth g}^{e}+\rg) $$
with $p_{i}=\displaystyle{\frac{\partial R_{i}}{\partial q_{i}}}$ for $i=1,\ldots,\rg$.
\end{coro}

\begin{proof}
Since $\displaystyle{\frac{\partial R_{i}}{\partial q_{i}}}\not=0$ 
for all $i$, $\poi r1{,\ldots,}{\rg}{}{}{}$ are 
algebraically independent and 
$$\poi {\dd r}1{\wedge \cdots \wedge }{\rg}{}{}{} = 
\prod_{i=1}^{\rg} \frac{\partial R_{i}}{\partial q_{i}} 
\poi {\dd q}1{\wedge \cdots \wedge }{\rg}{}{}{}$$
whence the corollary by Lemma~\ref{lrc1},(iii).
\end{proof}

Let ${\goth g}_{\rm sing}^{f}$ be the set of nonregular elements of the dual $\g^{f}$ of 
${\goth g}^{e}$. Recall that if ${\goth g}_{\rm sing}^{f}$ has codimension at least $2$ 
in ${\goth g}^{f}$, we will say that ${\goth g}^{e}$ is {\em nonsingular}.

\begin{coro}\label{c2rc1}
Let $\poi r1{,\ldots,}{\rg}{}{}{}$ and $p$ be as in 
Lemma~\ref{lrc1} and such that $\poi {\ie r}1{,\ldots,}{\rg}{}{}{}$ are algebraically
independent.

{\rm (i)} If $\ie p$ is a greatest  divisor of 
$\poi {\dd {\ie r}}1{\wedge \cdots \wedge }{\rg}{}{}{}$ in 
$\tk {\k}{\es S{{\goth g}^{e}}}\ex {\rg}{{\goth g}^{e}}$, then 
${\goth g}^{e}$ is nonsingular. 

{\rm (ii)} Assume that there are homogenous polynomials 
$\poi p1{,\ldots,}{\rg}{}{}{}$ in $\es S{{\goth g}^{e}}^{\g^{e}}$ 
satisfying the following conditions:
\begin{itemize}
\item [{\rm 1)}] $\poi {\ie r}1{,\ldots,}{\rg}{}{}{}$ are in 
$\k[\poi p1{,\ldots,}{\rg}{}{}{}]$,
\item [{\rm 2)}] if $d$ is the degree of a greatest  divisor of 
$\poi {\dd p}1{\wedge \cdots \wedge }{\rg}{}{}{}$ in $\es S{{\goth g}^{e}}$, then
$$ \poi {\deg p}1{+\cdots +}{\rg}{}{}{} = d+\frac{1}{2}(\dim {\goth g}^{e}+\rg) .$$
\end{itemize}
Then ${\goth g}^{e}$ is nonsingular. 
\end{coro}

\begin{proof}
(i) Suppose that $\ie p$ is a greatest  divisor of 
$\poi {\dd {\ie r}}1{\wedge \cdots \wedge }{\rg}{}{}{}$ in 
$\tk {\k}{\es S{{\goth g}^{e}}}\ex {\rg}{{\goth g}^{e}}$. Then for some $\omega _{1}$
in $\tk {\k}{\es S{{\goth g}^{e}}}\ex {\rg}{{\goth g}^{e}}$ whose nullvariety in 
${\goth g}^{f}$ has codimension at least $2$,
$$ \poi {\dd {\ie r}}1{\wedge \cdots \wedge }{\rg}{}{}{} = \ie p\,\omega _{1}.$$
Therefore $j(\omega _{1})=a\omega _{0}$ by Equality~(\ref{eqrc1}). Since  
$x$ is in ${\goth g}^{f}_{\rm sing}$ if and only if $\omega _{0}(x)=0$, 
we get (i). 

(ii) By Condition (1),
$$ \poi {\dd {\ie r}}1{\wedge \cdots \wedge }{\rg}{}{}{} = q\, 
\poi {\dd p}1{\wedge \cdots \wedge }{\rg}{}{}{}$$
for some $q$ in $\ai ge{}$, and for some greatest  divisor $q'$ of 
$\poi {\dd p}1{\wedge \cdots \wedge }{\rg}{}{}{}$ in 
$\tk {\k}{\es S{{\goth g}^{e}}}\ex {\rg}{{\goth g}^{e}}$, 
$$\poi {\dd p}1{\wedge \cdots \wedge }{\rg}{}{}{} = q'\omega _{1}.$$
So, by Equality~(\ref{eqrc1}),
\begin{eqnarray} \label{eq2rc1}
qq'j(\omega _{1}) = a \ie p\, \omega _{0} ,
\end{eqnarray} 
so that $\ie p$ divides $qq'$ in $\es S{{\goth g}^{e}}$. By Condition (2) and Equality 
(\ref{eq2rc1}),
$\omega _{0}$ and $\omega _{1}$ have the same degree. Then $qq'$ is in $\k ^{*}\ie p$, 
and for some $a'$ in $\k^{*}$,
$$ j(\omega _{1}) = a' \omega _{0},$$ 
whence (ii), again since $x$ is in ${\goth g}^{f}_{\rm sing}$ if
and only if $\omega _{0}(x)=0$.
\end{proof}
 
The following proposition is a particular case of \cite[\S5.7]{JS}. 
More precisely, part (i) follows from \cite[Remark 5.7]{JS} and 
part (ii) follows from \cite[Theorem 5.7]{JS}. 

\begin{prop}\label{prc1}
Suppose that ${\goth g}^{e}$ is nonsingular. 

{\rm(i)} If there exist algebraically independent homogenous polynomials 
$\poi p1{,\ldots,}{\rg}{}{}{}$ in $\es S{{\goth g}^{e}}^{\g^{e}}$ such that
$$ \poi {\deg p}1{+\cdots +}{\rg}{}{}{} = \frac{1}{2}(\dim {\goth g}^{e}+\rg)$$
then $\ai ge{}{}$ is a polynomial algebra generated by $\poi p1{,\ldots,}{\rg}{}{}{}$.

{\rm (ii)} Suppose that the semiinvariant elements of $\es S{{\goth g}^{e}}$
are invariant. If $\ai ge{}$ is a polynomial algebra then it is generated by 
homogenous polynomials $\poi p1{,\ldots,}{\rg}{}{}{}$ such that
$$ \poi {\deg p}1{+\cdots +}{\rg}{}{}{} = \frac{1}{2}(\dim {\goth g}^{e}+\rg).$$
\end{prop}

\subsection{New examples}\label{rc2}
To produce new examples, our general strategy is to apply Proposition \ref{prc1},(i). 
To that end, we first apply Corollary \ref{c2rc1} in order to prove that ${\goth g}^{e}$ 
is nonsingular. Then, we search for independent homogenous 
polynomials $\poi p1{,\ldots,}{\rg}{}{}{}$ in $\es S{{\goth g}^{e}}^{\g^{e}}$ satisfying 
the conditions of Corollary \ref{c2rc1},(ii) with $d=0$.  

\begin{example}\label{erc2} 
Let $e$ be a nilpotent element of ${\goth {so}}(\k^{10})$ 
associated with the partition $(3,3,2,2)$. 
Then S$(\g^{e})^{\g^{e}}$ is a polynomial algebra but $e$ is not good 
as explained below. 

In this case, $\rg =5$ and let $\poi q1{,\ldots,}{5}{}{}{}$ be as in Subsection~\ref{ca2}.
The degrees of $\poi {\ie q}1{,\ldots,}{5}{}{}{}$ are $1,2,2,3,2$ respectively. By 
a computation performed by Maple, $\poi {\ie q}1{,\ldots,}{5}{}{}{}$ satisfy the
algebraic relation:
$$ \ie q_{4}^{2} - 4 \ie q_{3}\ie q_{5}^{2}=0.$$
Set:
$$ r_{i} := \left \{\begin{array}{ccc} q_{i} & \mbox{ if } & i=1,2,3,5 \\
q_{4}^{2} - 4 q_{3} q_{5}^{2} & \mbox{ if } & i=4. \end{array} \right.$$

The polynomials $\poi r1{,\ldots,}{5}{}{}{}$ are algebraically independent over 
$\k$ and
$$ \poi {\dd r}1{\wedge \cdots \wedge }{5}{}{}{} = 
2\,q_{4}\, \poi {\dd q}1{\wedge \cdots \wedge }{5}{}{}{}$$
Moreover, $\ie r_{4}$ has degree at least 7. Then, by Corollary~\ref{crc1}, 
$\poi {\ie r}1{,\ldots,}{5}{}{}{}$ are algebraically independent since
$$ \frac{1}{2}(\dim {\goth g}^{e}+5)+3 = 14 = 1+2+2+2+7,$$
and by Lemma~\ref{lrc1},(ii) and (iii), $\ie r_{4}$ has degree $7$. 

A precise computation performed by Maple shows that $\ie r_{3}=p_{3}^{2}$ for some 
$p_{3}$ in the center of ${\goth g}^{e}$, and that 
$\ie r_{4}=p_{4}\ie r_{5}$ for some polynomial $p_4$ of degree $5$ in $\ai ge{}$. 
Setting $p_i:=\ie r_{i}$ for $i=1,2,5$, the polynomials $\poi p1{,\ldots,}{5}{}{}{}$ are 
algebraically independent homogenous polynomials of degree $1,2,1,5,2$ respectively. 
Furthermore, a computation performed by Maple proves that the greatest  divisors of
$\poi {\dd p}1{\wedge \cdots \wedge }{5}{}{}{}$ in $\es S{{\goth g}^{e}}$ have degree 
$0$, and that $p_{4}$ is in the ideal of $\es S{{\goth g}^{e}}$ generated by $p_{3}$ and 
$p_{5}$. So, by Corollary~\ref{c2rc1},(ii), $\g^{e}$ is nonsingular, and by 
Proposition~\ref{prc1},(i), $\ai ge{}$ is a polynomial algebra generated by 
$\poi p1{,\ldots,}{5}{}{}{}$. Moreover, $e$ is not good since the nullvariety of 
$\poi p1{,\ldots,}{5}{}{}{}$ in $(\g^{e})^*$ has codimension at most $4$. 
\end{example}

\begin{example}\label{e2rc2}
In the same way, for the nilpotent element $e$ of ${\goth {so}}(\k^{11})$ 
associated with the partition $(3,3,2,2,1)$, we can prove that $\ai ge{}$ is a 
polynomial algebra generated by polynomials of degree $1,1,2,2,7$, $\g^{e}$ is 
nonsingular but $e$ is not good.  

We also obtain that for the nilpotent element $e$ of ${\goth {so}}(\k^{12})$ 
(resp.~${\goth {so}}(\k^{13})$) associated with the partition (5,3,2,2) or (3,3,2,2,1,1) 
(resp.~(5,3,2,2,1), (4,4,2,2,1), or (3,3,2,2,1,1,1)), 
$\ai ge{}$ is a polynomial algebra, $\g^{e}$ is nonsingular but $e$ is not good. 
\end{example}

We can summarize our conclusions for the small ranks.  
Assume that $\g=\mathfrak{so}(\V)$ for some vector space $\V$ 
of dimension $2\rg +1$ or $2\rg$ and let $e\in\g$ be a nilpotent element of $\g$ 
associated with the partition $\lambda=(\lambda_1,\ldots,\lambda_k)$ of $\dim \V$. 
If $\rg \leq 6$, our previous results (Corollary \ref{cca2}, Lemma \ref{l4ca2}, 
Theorem \ref{tca3}, Examples \ref{erc2} and \ref{e2rc2}) show that either $e$ is good, or 
$e$ is not good but S$(\g^{e})^{\g^{e}}$ is nevertheless a polynomial algebra  
and $\g^{e}$ is nonsingular. 
We describe in Table \ref{tabrc2} the partitions $\lambda$ 
corresponding to good $e$, and those corresponding to the case where   
$e$ is not good. The third column of the table gives the degrees of the generators in the 
latter case. 

{\tiny 
\begin{table}[ht!]
\begin{tabular}{lclclcl}
Type && $e$ is good && S$(\g^{e})^{\g^{e}}$ is polynomial, $\g^{e}$ is nonsingular &&
degrees of the generators \\
&&&& but $e$ is not good && 
\small\\
\hline
\small\\
{\bf B}$_n$, {\bf D}$_n$, $n\leq 4$ && any $\lambda$ && $\varnothing$ && \\
{\bf B}$_5$ && $\lambda\not=(3,3,2,2,1)$ && $\lambda=(3,3,2,2,1)$ && $1,1,2,2,7$ \\
{\bf D}$_5$ && $\lambda\not=(3,3,2,2)$ && $\lambda=(3,3,2,2)$ && 1,1,2,2,5 \\
{\bf B}$_6$ && $\lambda\not\in\{(5,3,2,2,1),(4,4,2,2,1),$ && 
$\lambda \in \{(5,3,2,2,1),(4,4,2,2,1),$ && $\{1,1,1,2,2,7; 1,1,2,2,3,6 ;$ \\
&& $(3,3,2,2,1,1,1)\}$ && $ (3,3,2,2,1,1,1)\}$ && $ 1,1,2,2,6,7\}$\\
{\bf D}$_6$ && $\lambda\not\in\{(5,3,2,2),(3,3,2,2,1,1)\}$ && 
$\lambda\in\{(5,3,2,2),(3,3,2,2,1,1)\}$ && \{1,1,1,2,2,5; 1,1,2,2,3,7\}\\
\\
\hline
\small\\
\end{tabular}
\caption{\footnotesize 
Conclusions for $\g$ of type {\bf B}$_\rg$ or {\bf D}$_\rg$ with $\rg\leq 6$} 
\label{tabrc2}
\end{table}}

\begin{rema}\label{rrc3}
The above discussion shows that there are good nilpotent elements for which 
the codimension of $(\g^{e})^*_{\rm sing}$ in $(\g^{e})^*$ is 1. 
Indeed, by \cite[\S3.9]{PPY}, for some nilpotent element $e'$ in ${\bf B}_3$, 
the codimension of $(\g^{e'})^*_{\rm sing}$ in $(\g^{e'})^*$ is 1 but, in ${\bf B}_3$, 
all nilpotent elements are good (cf.~Table \ref{tabrc2}).
\end{rema}

\subsection{A counter-example}\label{rc3} 
From the rank $7$, there are elements that do no satisfy the polynomiality condition. 
The following example provides a new counter-example to Premet's conjecture. 

\begin{example} \label{erc3}
Let $e$ be a nilpotent element of ${\goth {so}}(\k^{14})$ associated with the partition 
$(3,3,2,2,2,2)$. Then $e$ does not satisfy the polynomiality condition. 

In this case, $\rg =7$ and let $\poi q1{,\ldots,}{7}{}{}{}$ be as in Subsection~\ref{ca2}.
The degrees of $\poi {\ie q}1{,\ldots,}{7}{}{}{}$ are $1,2,2,3,4,5,3$ respectively. By 
a computation performed by Maple, we can prove that 
$\poi {\ie q}1{,\ldots,}{7}{}{}{}$ satisfy the two
following algebraic relations:
$$\begin{array}{ccc}
16\ie q_{3}^2{\ie q_{5}}^2+\ie q_{4}^4-8\ie q_{3}\ie q_{5}\ie q_{4}^2-64\ie q_{3}^3
{\ie q_{7}}^2 = 0, &&
\ie q_{3}\ie q_{6}^2-\ie q_{7}^2{\ie q_{4}}^2 = 0 \end{array}$$
Set:
$$r_{i} := \left \{\begin{array}{ccc} q_{i} & \mbox{ if } & i=1,2,3,4,7 \\
16\,q_{3}^2{q_{5}}^2+q_{4}^4-8\,q_{3}q_{5}q_{4}^2-64\,q_{3}^3{q_{7}}^2 & \mbox{ if } &
i=5 \\
q_{3}q_{6}^2-q_{7}^2q_{4}^2 & \mbox{ if } & i=6 \end{array} \right.$$
The polynomials $\poi r1{,\ldots,}{7}{}{}{}$ are algebraically independent over 
$\k$ and
$$ \poi {\dd r}1{\wedge \cdots \wedge }{7}{}{}{} = 
2q_{3}q_{6}\,(32 q_{3}^{2}q_{5}-8q_{3}q_{4}^{2}) 
\,\poi {\dd q}1{\wedge \cdots \wedge }{7}{}{}{}$$
Moreover, $\ie r_{5}$ and $\ie r_{6}$ have degree at least 13 and 
$\ie (2q_{3}q_{6}(32 q_{3}^{2}q_{5}-8q_{3}q_{4}^{2}))$ has degree $15$. Then, by 
Corollary~\ref{crc1}, $\poi {\ie r}1{,\ldots,}{7}{}{}{}$ are algebraically
independent since
$$ \frac{1}{2}(\dim {\goth g}^{e}+7)+15 = 37 = 1+2+2+3+3+26 $$
and by Lemma~\ref{lrc1},(ii) and (iii), $\ie r_{5}$ and $\ie r_{6}$ have degree $13$. 

A precise computation performed by Maple shows that $\ie r_{3}=p_{3}^{2}$ for some 
$p_{3}$ in the center of ${\goth g}^{e}$, $\ie r_{4}=p_{3}p_{4}$ for some polynomial 
$p_4$ of degree $2$ in $\ai ge{}$, $\ie r_{5} = p_{3}^{3}\ie q_{7}p_{5}$ for some 
polynomial $p_{5}$ of degree $7$ in $\ai ge{}$, and $\ie r_{6}=p_{4}\ie r_{7}p_{6}$ for 
some polynomial $p_{6}$ of degree $8$ in $\ai ge{}$. Setting $p_{i}:=\ie r_{i}$ for 
$i=1,2,7$, the polynomials $\poi p1{,\ldots,}{7}{}{}{}$ are algebraically independent 
homogenous polynomials of degree $1,2,1,2,7,8,3$ respectively. 
Let ${\goth l}$ be a reductive factor of $\g^{e}$. 
According to \cite[Ch.\,13]{Ba}, 
$${\goth l}\simeq \mathfrak{so}_2(\k) \times \mathfrak{sp}_4(\k)
\simeq \k\times\mathfrak{sp}_4(\k) .$$
In particular, the center of ${\goth l}$ has dimension 1. 
Let $\{\poi x1{,\ldots,}{37}{}{}{}\}$ be a basis of ${\goth g}^{e}$ such that $x_{37}$ 
lies in the center of ${\goth l}$ and such that 
$\poi x1{,\ldots,}{36}{}{}{}$ are in $[{\goth l},{\goth l}]+{\goth g}^{e}_{\uu}$ with 
${\goth g}^{e}_{\uu}$ the nilpotent radical of ${\goth g}^{e}$. Then $p_{2}$ is a 
polynomial in $\k[\poi x1{,\ldots,}{37}{}{}{}]$ depending on $x_{37}$. As a result, 
by~\cite[Theorems 3.3 and 4.5]{DDV}, the semiinvariant polynomials of 
$\es S{{\goth g}^{e}}$ are invariant.

\begin{claim}\label{clrc3}
The algebra $\g^{e}$ is nonsingular.
\end{claim}

\begin{proof}{[Proof of Claim~\ref{clrc3}]}
The space $\k^{14}$ is the orthogonal direct sum of two subspaces ${\Bbb V}_{1}$ and 
${\Bbb V}_{2}$ of dimension $6$ and $8$ respectively and such that $e,h,f$ are 
in $\overline{{\goth g}} := {\goth {so}}({\Bbb V}_{1})\oplus {\goth {so}}({\Bbb V}_{2})$.
Then $\overline{{\goth g}}^{e}=\overline{{\goth g}}\cap\g^{e}$ is a subalgebra of 
dimension $21$ containing the center
of ${\goth g}^{e}$. For $p$ in $\es S{{\goth g}^{e}}$, denote by $\overline{p}$ its 
restriction to $\overline{{\goth g}}^{f}$. 
The partition $(3,3,2,2,2,2)$ satisfies Condition (1) of the proof 
of~\cite[\S4, Lemma 3]{Y1}. So, the proof of Lemma~\ref{l2ca3} remains valid, and the 
morphism 
$$\begin{array}{ccc}
G_{0}^{e}\times \overline{{\goth g}}^{f} \longrightarrow {\goth g}^{f}, &&
(g,x) \longmapsto g(x) \end{array}$$ 
is dominant. As a result, for $p$ in $\ai ge{}$, the differential of $\overline{p}$ is 
the restriction to $\overline{{\goth g}}^{f}$ of the differential of $p$. A computation 
performed by Maple proves that $\overline{p_{3}}^{10}$ is a greatest divisor of 
$\dd \overline{p_{1}}\wedge \cdots \wedge \dd \overline{p_{7}}$ in 
$\es S{\overline{{\goth g}^{e}}}$. If $q$ is a greatest divisor of 
$\dd {p_{1}}\wedge \cdots \wedge \dd{p_{7}}$ in $\es S{{\goth g}^{e}}$, then $q$ is in 
$\ai ge{}$ since the semiinvariant polynomials are invariant. So $q=p_{3}^{d}$ for some 
nonnegative integer $d$. One can suppose that $\{\poi x1{,\ldots,}{16}{}{}{}\}$ is a 
basis of the orthogonal complement to $\overline{{\goth g}}^{f}$ in ${\goth g}^{e}$. Then
the Pfaffian of the matrix
$$\left([x_{i},x_{j}], \; 1\leq i,j \leq 16\right)$$
is in $\k^{*} p_{3}^{8}$ so that $p_{3}^{2}$ is a greatest divisor of 
$\poi {\dd p}1{\wedge \cdots \wedge }{7}{}{}{}$ in $\es S{{\goth g}^{e}}$. Since
$$ \deg p_{1} + \cdots + \deg p_{7} = 2 + 22 = 2 + \frac{1}{2}(\dim {\goth g}^{e}+\rg),$$
we conclude that $\g^{e}$ is nonsingular by Corollary~\ref{c2rc1},(ii).
\end{proof}

\begin{claim}\label{cl2rc3}
Suppose that $\ai ge{}{}$ is a polynomial algebra. Then for some homogenous polynomials
$p'_{5}$ and $p'_{6}$ of degrees at least $5$ and at most $8$ respectively, $\ai ge{}$ is
generated by $p_{1}$, $p_{2}$, $p_{3}$, $p_{4}$, $p'_{5}$, $p'_{6}$, $p_{7}$. 
Furthermore, the possible values for $(\deg p'_{5},\deg p'_{6})$ are $(5,8)$ or $(6,7)$.
\end{claim}

\begin{proof}{[Proof of Claim~\ref{cl2rc3}]}
Since the semiinvariants are invariants, by Claim~\ref{clrc3} and 
Proposition~\ref{prc1},(ii), there are homogenous generators 
$\poi {\varphi }1{,\ldots,}{\rg}{}{}{}$ of $\ai ge{}$ such that
$$ \poi {\deg \varphi }1{ \leq \cdots \leq }{\rg}{}{}{} ,$$
and 
$$ \poi {\deg \varphi }1{+ \cdots +}{\rg}{}{}{} = \frac{1}{2}(\dim {\goth g}^{e}+\rg) 
= 22.$$
According to~\cite[Theorem 1.1.8]{Mo} or \cite{Y2}, the center of ${\goth g}^{e}$ has 
dimension $2$. Hence, $\varphi _{1}$ and $\varphi _{2}$ have degree $1$. Thereby, we can 
suppose that $\varphi _{1}=p_{1}$ and $\varphi _{2}=p_{3}$ since $p_{1}$ and $p_{3}$ are 
linearly independent elements of the center of ${\goth g}^{e}$. Since $p_{2}$ and $p_{4}$
are homogneous elements of degree $2$ such that $\poi p1{,\ldots,}{4}{}{}{}$ are 
algebraically indepent, $\varphi _{3}$ and $\varphi _{4}$ have degree $2$ and we can 
suppose that $\varphi _{3}=p_{2}$ and $\varphi _{4}=p_{4}$. Since $p_{7}$ has degree $3$, 
$\varphi _{5}$ has degree at most $3$ and at least $2$ since the center of 
${\goth g}^{e}$ has dimension $2$. Suppose that $\varphi _{5}$ has degree $2$. A 
contradiction is expected. Then 
$$ \deg \varphi _{6} + \deg \varphi _{7} = 22-(1+1+2+2+2)=14 .$$ 
Moreover, since $\poi p1{,\ldots,}{7}{}{}{}$ are algebraically independent, $\varphi _{7}$
has degree at most $8$ and $\varphi _{6}$ has degree at least $6$. Hence $p_{7}$ is in 
the ideal of $\k[p_{1},p_{3},\varphi _{3},\varphi _{4},\varphi _{5}]$ generated by 
$p_{1}$ and $p_{3}$. But a computation shows that the restriction of $p_{7}$ to the 
nullvariety of $p_{1}$ and $p_{3}$ in ${\goth g}^{f}$ is different from $0$, whence the 
expected contradiction. As a result, $\varphi _{5}$ has degree $3$ and 
$$ \deg \varphi _{6} + \deg \varphi _{7} = 13 .$$
One can suppose $\varphi _{5}=p_{7}$ and the possible values for
$(\deg \varphi _{6},\deg \varphi _{7})$ are $(5,8)$ and $(6,7)$ since $\varphi _{7}$ has 
degree at most $8$. 
\end{proof}

Suppose that $\ai g{e}{}$ is a polynomial algebra. A contradiction is expected. 
Let $p'_{5}$ and $p'_{6}$ be as in Claim~\ref{cl2rc3} and such that 
$\deg p'_{5}<\deg p'_{6}$. Then $(\deg p'_{5},\deg p'_{6})$ equals $(5,8)$ or $(6,7)$. 
A computation shows that we can choose a basis $\{\poi x1{,\ldots,}{37}{}{}{}\}$ of 
${\goth g}^{e}$ with $x_{37}=p_{3}$, with $p_{1}$, $p_{2}$, $p_{3}$, $p_{4}$, $p_{7}$ 
in $\k[\poi x3{,\ldots,}{37}{}{}{}]$ and with $p_{5}$, $p_{6}$ of degree $1$ in $x_{1}$. 
Moreover, the coefficient of $x_{1}$ in $p_{5}$ is a prime element of 
$\k[\poi x3{,\ldots,}{37}{}{}{}]$, the coefficient of $x_{1}$ in $p_{6}$ is a prime 
element of $\k[\poi x2{,\ldots,}{37}{}{}{}]$ having degree $1$ in $x_{2}$, and the 
coefficient of $x_{1}x_{2}$ in $p_{6}$ equals $a^2p_{3}^{2}$ with $a$ a prime homogenous
polynomial of degree $2$ such that $a,p_{1},p_{2},p_{3},p_{4}$ are algebraically 
independent. In particular, $a$ is not invariant. If $p'_{5}$ has degree $5$, then 
$$ p_{5} = p'_{5}r_{0} + r_{1}$$
with $r_{0}$ in $\k[p_{1},p_{2},p_{3},p_{4}]$ and $r_{1}$ in 
$\k[p_{1},p_{2},p_{3},p_{4},p_{7}]$ so that $p'_{5}$ has degree 1 in $x_{1}$, and the 
coefficient of $x_{1}$ in $p_{5}$ is the product of $r_{0}$ and the coefficient of
$x_{1}$ in $p'_{5}$. But this is impossible since this coefficient is prime. So, 
$p'_{5}$ has degree $6$ and $p'_{6}$ has degree $7$. We can suppose that 
$p'_{6}=p_{5}$. Then 
$$ p_{6} = p_{5}r_{0} + p'_{6}r_{1} + r_{2}$$ 
with $r_{0}$ homogenous of degree $1$ in $\k[p_{1},p_{3}]$, $r_{1}$ homogenous of 
degree $2$ in $\k[p_{1},p_{2},p_{3},p_{4}]$, and $r_{2}$ homogenous of degree 
$8$ in $\k[p_{1},p_{2},p_{3},p_{4},p_{7}]$. According to the above remarks on $p_{5}$
and the coefficient of $x_{1}x_{2}$ in $p_{6}$, $r_{1}$ is in $\k^{*}p_{3}^{2}$
since $r_{1}$ has degree $2$. 

For $p$ in $\es S{{\goth g}^{e}}$, denote by $\overline{p}$ its image in 
$\es S{{\goth g}^{e}}/p_{3}\es S{{\goth g}^{e}}$. A computation shows that for some
$u$ in $\es S{{\goth g}^{e}}/p_{3}\es S{{\goth g}^{e}}$, 
$$\begin{array}{ccc}
\overline{p_{5}} = \overline{p_{4}}^{2}u, && 
\overline{p_{6}} = -\overline{p_{4}}\overline{p_{7}}u.\end{array}$$
Furthermore, $\overline{p_{4}}$ and $\overline{p_{7}}$ are different prime elements of 
$\es S{{\goth g}^{e}}/p_{3}\es S{{\goth g}^{e}}$ and the coefficient $u_{1}$ of 
$x_{1}$ in $u$ is the product of two different polynomials of degree $1$. The coefficient
of $x_{1}$ in $\overline{p_{6}}$ is $u_{1}\overline{p_{4}}^{2}\overline{r_{0}}$ since 
$$\overline{p_{6}}= \overline{p_{5}}\overline{r_{0}}+\overline{r_{2}}. $$
On the other hand, the coefficient of $x_{1}$ in $\overline{p_{6}}$ is 
$-u_{1}\overline{p_{4}}\overline{p_{7}}$, whence the contradiction since $r_{0}$ has 
degree $1$.
\end{example}

\subsection{A conjecture} \label{rc5}
All examples of good elements we achieved 
satisfy the hypothesis of Theorem \ref{tge2}. 

\begin{conj} \label{cjrc5}
Let $\g$ be a simple Lie algebra and let $e$ be a nilpotent of $\g$.
If $e$ is good then for some homogenous generating sequence 
$(\poi q1{,\ldots,}{\rg}{}{}{})$ in $\ai g{}{}$, $\poi {\ie q}1{,\ldots,}{\rg}{}{}{}$ are 
algebraically independent over $\k$. In other words, the converse implication of 
Theorem {\rm \ref{tge2}} holds.  
\end{conj}

Notice that it may happen that for some $r_1,\ldots,r_\ell$ in $\ai g{}{}$, 
the elements $\poi {\ie r}1{,\ldots,}{\rg}{}{}{}$ are algebraically independent over 
$\k$, and that however $e$ is not good. This is the case for instance for the nilpotent 
elements in $\mathfrak{so}(\k^{12})$ associated with the partition $(5,3,2,2)$, 
cf.\ Example \ref{e2rc2}. 

In fact, according to \cite[Corollary 2.3]{PPY}, for any nilpotent $e$ of 
$\g$, there exist $\poi r1{,\ldots,}{\rg}{}{}{}$ in $\ai g{}{}$ such that 
$\poi {\ie r}1{,\ldots,}{\rg}{}{}{}$ are algebraically independent over $\k$.

\end{document}